\documentclass[a4paper]{amsart}
\usepackage{hyperref}
\usepackage{enumerate}

\usepackage{a4wide,amsthm,stix}

\usepackage{amsmath}
\usepackage{amsfonts}
\usepackage{color}
\usepackage{dsfont}

\usepackage{graphicx}
\usepackage{subfig}
\usepackage{bbm}
\usepackage{tikz}

\newtheorem{thm}{Theorem}

\newtheorem{question}{Question}

\newtheorem{lemma}{Lemma}[section]
\newtheorem{proposition}[lemma]{Proposition}
\newtheorem{theorem}[lemma]{Theorem}
\newtheorem{corollary}[lemma]{Corollary}

\theoremstyle{definition}

\theoremstyle{remark}
\newtheorem{remark}[lemma]{Remark}
\newtheorem{notation}[lemma]{Notation}

\DeclareMathOperator{\supp}{supp}
\DeclareMathOperator{\spec}{spec}

\newcommand{\dst}{\displaystyle}
\newcommand{\ffi}{\varphi}
\newcommand{\eps}{\varepsilon}

\def\C{{\mathds{C}}}
\def\D{{\mathbb{D}}}

\def\N{{\mathds{N}}}

\def\R{{\mathds{R}}}

\def\Z{{\mathds{Z}}}

\def\dd{{\mathcal D}}

\def\ff{{\mathcal F}}

\def\pp{{\mathcal P}}

\def\ss{{\mathcal S}}

\newcommand{\norm}[1]{{\left\|{#1}\right\|}}
\newcommand{\ent}[1]{{\left[{#1}\right]}}
\newcommand{\abs}[1]{{\left|{#1}\right|}}
\newcommand{\scal}[1]{{\left\langle{#1}\right\rangle}}

\DeclareMathOperator{\sinc}{sinc}

\newcommand{\ud}{\,\mathrm{d}}

\numberwithin{equation}{section}

\begin{document}
\title{From Ingham to Nazarov's inequality: a survey on some trigonometric inequalities}
\author{Philippe Jaming \& Chadi Saba}
\address{Univ. Bordeaux, CNRS, Bordeaux INP, IMB, UMR 5251,  F-33400, Talence, France}

\begin{abstract}
The aim of this paper is to give an overview of some inequalities about $L^p$-norms ($p=1$ or $p=2$)
of harmonic (periodic) and non-harmonic trigonometric polynomials.
Among the material covered, we mention Ingham's Inequality about $L^2$ norms of non-harmonic trigonometric polynomials, the proof of the Littlewood conjecture by Mc\,Gehee, Pigno and Smith on the lower bound
of the $L^1$ norm of harmonic trigonometric polynomials as well as its counterpart in the non-harmonic
case due to Nazarov. For the latter one, we give a quantitative estimate that completes our
recent result with an estimate of $L^1$-norms over small intervals.

We also give some stronger lower bounds when the frequencies satisfy some more restrictive conditions
(lacunary Fourier series, ``multi-step arithmetic sequences'').

Most proofs are close to existing ones and some open questions are mentioned at the end.
\end{abstract}


\maketitle 

\tableofcontents

\section{Introduction}

The aim of this paper is to give an overview of the main estimates of $L^p$-norms
of harmonic and non-harmonic trigonometric polynomials, when $p=2$ and $p=1$.
In many fields of mathematics, ranging from number theory ({\it see e.g.} the previous survey \cite{BeKo})
to signal processing and PDEs, one is lead to investigate such norms. Our main motivation
comes from the use of Ingham's inequality (lower and upper estimates
of $L^2$-norms of non-harmonic trigonometric polynomials) in control theory of PDEs.
We refer the interested reader to the book by Komornik and Loreti \cite{KL}

Let us not be more precise. We will here restrict attention to $p=2$ or $p=1$
and investigate $L^p$-norm estimates of (harmonic) trigonometric polynomials
$$
\int_{-1/2}^{1/2}\abs{\sum_{k=-N}^N c_ke^{2i\pi n_k t}}^p\ud t\quad c_k\subset\C,\ n_k\subset\N
$$
or non-harmonic trigonometric polynomials
$$
\frac{1}{T}\int_{-T/2}^{T/2}\abs{\sum_{k=-N}^N c_ke^{2i\pi \lambda_k t}}^p\ud t\quad T>0,\ 
c_k\subset\C,\ \lambda_k\subset\R
$$
as well as their limits when $T\to+\infty$
$$
\limsup_{T\to+\infty}\frac{1}{T}\int_{-T/2}^{T/2}\abs{\sum_{k=-N}^N c_ke^{2i\pi \lambda_k t}}^p\ud t\quad T>0,\ 
c_k\subset\C,\ \lambda_k\subset\R
$$
(that is Besicovitch norms). Note for future use that, when the $\lambda_k$'s are all integers, we recover the
$L^p([-1/2,1/2])$-norms.

The simplest case $p=2$ for harmonic trigonometric polynomials known by Parseval's relation states that
$$
\int_{-1/2}^{1/2}\abs{\sum_{k=-N}^N c_ke^{2i\pi k t}}^2\ud t=\sum_{k=-N}^N |c_k|^2.
$$
The situation becomes much complicated for non-harmonic trigonometric polynomials
but is well understood thanks to a deep result by Ingham:

\begin{thm}[Ingham]
Let $T>1$, then there are two constants $A$ and $B$ such that

-- for every sequence $(\lambda_k)_{k\in\Z}$ with $\lambda_{k+1}-\lambda_k\geq 1$,

-- for every complex sequence $(c_k)_{k=-N,\ldots,N}$

for every integer $N\geq 1$,
\begin{equation}
\label{eq:inghamintro}
A\sum_{k=-N}^N|c_k|^2\leq\frac{1}{T}\int_{-T/2}^{T/2}\sum_{k=-N}^N\abs{c_ke^{2i\pi\lambda_k t}}^2\,\ud t
\leq B\sum_{k=-N}^N|c_k|^2.
\end{equation}

 The condition $T>1$ is optimal in the following sense; When $T=1$, there is a sequence $(\lambda_k)_{k\in\Z}$ with $\lambda_{k+1}-\lambda_k\geq 1$
and a sequence $(c_k)\in\ell^2$ such that the lower bound in \eqref{eq:inghamintro} does not hold for every $N$
for any $A>0$.
\end{thm}

Note that the upper bound in \eqref{eq:inghamintro} holds for every $T>0$, the restriction $T>1$ is only
needed for the lower bound, which is the most interesting one for control theory.
In the next section, we will present Ingham's proof and also compute the constants $A$ and $B$.
We will also present an argument due to Haraux that allows to obtain the lower bound with $T$ arbitrarily small
when $|\lambda_{k+1}-\lambda_k|\to+\infty$ when $k\to\pm\infty$.

Once the $L^2$-theory settled, we will move to the $L^1$-theory which is much richer.
To start with, the behavior of the $L^1$-norm of a trigonometric polynomial
may depend on the frequencies. There are two extreme cases

-- the frequencies form an arithmetic sequence, {\it e.g.} the Dirichlet kernel,
one might expect the $L^1$-norm to be small in view of the estimate: when $N\to+\infty$
$$
\int_{-1/2}^{1/2}\left|\sum_{k=-N}^N e^{2i\pi k t}\right|\,\mbox{d}t
= \frac{4}{\pi^2}\ln N+O(1),
$$

-- the frequencies form a geometric sequence, or more generally a lacunary sequence ($n_{k+1}\geq qn_k$ with $q>1$)
then the $L^1$-norms are much large: for $N$ large enough
$$
\int_{-1/2}^{1/2}\left|\sum_{k=0}^N e^{2i\pi n_k t}\right|\,\mbox{d}t\geq C\sqrt{N}.
$$

We will prove these two estimates in Section 3.
In view of those estimates, Littlewood conjectured that the Dirichlet kernel has worse possible
behavior, namely that
$$
L_N:=\inf_{n_0<n_1<\cdots<n_N}\int_{-1/2}^{1/2}\left|\sum_{k=0}^N e^{2i\pi n_k t}\right|\,\mbox{d}t
\geq C\log N
$$
for some constant $C\leq\dfrac{4}{\pi^2}$. 
The first non-trivial estimate was obtained by Cohen \cite{Coh} who proved that
$$
L_N\geq C(\ln N/\ln\ln N)^{1/8}
$$
for $N\geq 4$.
Subsequent improvements are due to Davenport \cite{Da}, Fournier \cite{Fo1}
and crucial contributions by Pichorides \cite{Pi1,Pi2,Pi3,Pi4}
leading to $L_N\geq C\ln N/(\ln\ln N)^2$.
Finally, Littlewood's conjecture was proved independently by Konyagin \cite{Kon} and Mc Gehee, Pigno, Smith \cite{MPS} in 1981. In both papers, Littlewood's conjecture is actually obtained as a corollary of a stronger result
(and they are not consequences of one another). The second one is the  one we will focus on here
and is the following result:

\begin{thm}[Mc Gehee, Pigno \& Smith \cite{MPS}]
\label{th:MPS}
For $n_1<n_2<\cdots<n_N$ integers 
and $a_1,\ldots,a_N$ complex numbers,
$$
\int_{-1/2}^{1/2}\left|\sum_{k=1}^N a_ke^{2i\pi n_k t}\right|\,\mathrm{d}t\geq C_{MPS}\sum_{k=1}^N\frac{|a_k|}{k}
$$
where $C_{MPS}$ is a universal constant ($C_{MPS}=1/30$ would do).
\end{thm}

We will present its proof below.

We would like to insist that this is the worse possible behavior. To start, as we already mentioned,
in the lacunary case, the lower bound is $\sqrt{N}$. We will show a recent result of Hanson \cite{Ha}
who considered a family of sets (so-called ``strongly multi-dimensional sets'') of which the simplest
example is a multi-step arithmetic sequence. By that, we mean a sequence $n_{jM+k}=jD+kd$, $j\in\Z$, $k=0,\ldots,M-1$ and $Md\ll D$. In other words,
we take an arithmetic sequence with large step $D$ and, after each element of the sequence, we add a small
piece of an arithmetic sequence with small step. It is then shown that
$$
\int_{-1/2}^{1/2}\left|\sum_{j=0}^N\sum_{k=0}^M e^{2i\pi(jD+kd) t}\right|\,\mbox{d}t\geq C\ln N\ln M.
$$
We also present an example of Newman that constructed $a_0,\ldots,a_N$ all of modulus 1 such that
$$
\int_{-1/2}^{1/2}\left|\sum_{k=0}^N a_ke^{2i\pi k t}\right|\,\mbox{d}t\geq\sqrt{N}-C.
$$
This shows that it is possible to be far away from the lower bound (and actually near to the best possible
upper bound).

\smallskip

The last part of this survey is devoted to $L^1$-norms of non-harmonic trigonometric polynomials.
The best result to day is the following:

\begin{thm}
\label{th:AA}
Let $\lambda_0<\lambda_2<\cdots<\lambda_N$ be $N$ distinct real numbers
and $a_0,\ldots,a_N$ be complex numbers. Then
\begin{enumerate}
\renewcommand{\theenumi}{\roman{enumi}}
\item we have
$$
\lim_{T\to+\infty}\frac{1}{T}\int_{-T/2}^{T/2}\left|\sum_{k=0}^N a_ke^{2i\pi \lambda_k t}\right|\,\mathrm{d}t
\geq \frac{1}{26}\sum_{k=1}^N\frac{|a_k|}{k+1}.
$$

\item If further $a_0,\ldots,a_N$ all have modulus larger than $1$, $|a_k|\geq 1$ then
$$
\lim_{T\to+\infty}\frac{1}{T}\int_{-T/2}^{T/2}\left|\sum_{k=0}^N a_ke^{2i\pi \lambda_k t}\right|\,\mathrm{d}t
\geq \frac{4}{\pi^3}\ln N.
$$

\item Assume further that for $k=0,\ldots,N-1$, $\lambda_{k+1}-\lambda_k\geq 1$, then,
for every $T>1$, there exists a constant $C(T)$ such that, for
every $a_0,\ldots,a_N\in\C$,
$$
\frac{1}{T}\int_{-T/2}^{T/2}\left|\sum_{k=0}^N a_ke^{2i\pi \lambda_k t}\right|\,\mathrm{d}t
\geq C(T)\sum_{k=1}^N\frac{|a_k|}{k+1}.
$$
Moreover, 
\begin{enumerate}
\item for $T\geq 72$ we can take $C(T)=\dfrac{1}{122}$;

\item for $1<T\leq 2$, $C(T)=O\bigl((T-1)^{15/2}\bigr)$.
\end{enumerate}
\end{enumerate}
\end{thm}

Let us comment on this theorem. To start with, the theorem is essentially a quantitative version
of a result of F. Nazarov \cite{Naz} and is due to the authors together with K. Kellay \cite{JKS}.
Only the case $1<T\leq 2$ has not been presented before. To be more precise,
when the $\lambda_k$'s are integers, this is of course the result
of Mc Gehee, Pigno and Smith while Point 2 comes from a slight modification of their argument 
by Stegeman and Yabuta. Note that the constants are
the same as those in Theorem \ref{th:MPS} (and are actually a bit better). 
While it is obvious that this inequality implies Theorem \ref{th:MPS} (with the same constant), 
there is an elegant argument by Hudson and Leckband \cite{HL} that allows to show that the converse is also true.
We will present this argument below.
Point 3 is due to Nazarov with non-explicit constant. A direct proof of this theorem is given in \cite{JKS} 
with the explicit constants mentioned above. Only the case
of small $T$ has not been presented so far and is thus the main novelty of the present paper.
It is obtained by a very mild modification of Nazarov's proof. Note also that once we have established Point 3
for some $T_0$, it is valid for all $T\geq T_0$ so that we actually recover Nazarov's original result (with much worse constants than the ones stated above).

\smallskip

The remaining of the paper is organised as follows: Section 2 is devoted to the $L^2$ case and we prove
the inequalities in Ingham's Theorem in Section 2.1 and the necessity of the condition $T>1$ in Section 2.2.
Further, Haraux's argument is presented in Section 2.3.

In Section 3, we investigate the $L^1$ norms when the frequencies are integers. We start with the classical asymptotic
estimate of the Dirichlet kernel and devote Section 2.2 to lacunary trigonometric polynomials. Section 3.3
is devoted to Mc Gehee, Pigno, Smith's proof of Littlewood's conjecture. In Section 3.4 we present Hanson's result
on strongly multidimensional sequences and we conclude in Section 3.5 with the result of Newman.

Section 4 is devoted to the case of non-harmonic trigonometric polynomials. We start in Section 4.1 with the argument of Hudson and Leckband and then extend the case of lacunary trigonometric polynomials
to the non-harmonic setting and the Besicovitch $L^1$-norms. We conclude this section with the proof of the quantitative version of Nazarov's theorem for small $T$.

In the last section, we present a few open questions.

\section{$L^2$ estimates}

\subsection{Ingham's inequalities}

The aim of this section is to show the following:
let $\Lambda=\{\lambda_k\}_{k\in\Z}\subset \R$ be a $1$-separated sequence, 
$|\lambda_k-\lambda_\ell|\geq 1$ if $k\not=\ell$. Let $\pp(\Lambda)$
be the set of (non-harmonic) trigonometric polynomials
$$
\pp(\Lambda)=\left\{P(t):=\sum_{k\in\Z}a_ke^{2i\pi\lambda_k t}\,:(a_k)_{k\in\Z}\subset\C\mbox{ with finite support}\right\}.
$$
Note that, if $P\in\pp(\Lambda)$ is given, then the $a_k$'s are determined by
$$
a_k=\lim_{T\to+\infty}\frac{1}{T}\int_{-T/2}^{T/2}P(t)e^{-2i\pi\lambda_k t}\,\mbox{d}t.
$$
We can then define two natural norms on $\pp(\Lambda)$, namely
$$
\|P\|_{L^2([-T/2,T/2])}:=\left(\frac{1}{T}\int_{-T/2}^{T/2}|P(t)|^2\,\mbox{d}t\right)^{\frac{1}{2}}
$$
and
$$
\|P\|_{\ell^2}:=\left(\sum_{k\in\Z}|a_k|^2\right)^{\frac{1}{2}}.
$$
Our aim in this section is to show that, when $T>1$, these two norms
are equivalent. This is done by proving two inequalities.
The first one is the direct inequality:

\begin{proposition}[Ingham's direct inequality]
Let $(a_k)_{k\in\Z}$ be a finitely supported sequence of
complex numbers and $(\lambda_k)_{k\in\Z}$ be a sequence of real numbers with
$\lambda_{k+1}-\lambda_k\geq 1$. For every $T>0$, 
\begin{equation}
\label{eq:inghamd}
\frac{1}{T}\int_{-T/2}^{T/2}\abs{\sum_{k\in\Z}a_ke^{2i\pi\lambda_k t}}^2\,\mbox{d}t\leq 2\frac{T+1}{T}\sum_{k\in\Z}|a_k|^2.
\end{equation}
\end{proposition}

\begin{proof} 
We consider the function $H$ on $\R$ defined by
$$
h(x)=\begin{cases}\cos\pi t&\mbox{when }|x|\leq \dfrac{1}{2}\\
0&\mbox{otherwise}\end{cases}.
$$
As $h$ is real and even, its Fourier transform is given by
$$
\widehat{h}(t)=2\int_0^{1/2}\cos\pi x\cos 2\pi xt\,\mbox{d}x
=\frac{2}{\pi}\frac{\cos\pi t}{1-4t^2}
$$
with the understanding that $\widehat{h}(1/2)=\dfrac{1}{2}$.
From this, one shows that $\widehat{h}(x)\geq \dfrac{1}{2}$ for $|x|\leq\dfrac{1}{2}$.
Finally, let 
$$
g(x)=h*h(x)=\begin{cases}\dfrac{\sin\pi|x|-\pi(|x|-1)\cos\pi x}{2\pi}&\mbox{when }|x|\leq 1\\
0&\mbox{otherwise}\end{cases}.
$$
One easily shows that $g$ is even, non-negative, supported in $[-1,1]$ and that $g(0)=\dfrac{1}{2}$.
Further its Fourier transform is $\widehat{g}(t)=\widehat{h}(t)^2$. In particular, $\widehat{g}(t)\geq 0$
and $\widehat{g}(t)\leq \dfrac{1}{4}$ for $|t|\leq\dfrac{1}{2}$.

But then, if $(a_k)_{k\in\Z}$ is finitely supported and
$P(t)=\dst\sum_{k\in\Z}a_ke^{2i\pi\lambda_k t}$,
\begin{eqnarray*}
\int_{-1/2}^{1/2}|P(t)|^2\,\mbox{d}t&\leq&4
\int_{-1/2}^{1/2}\widehat{g}(t)|P(t)|^2\,\mbox{d}t
\leq 4\int_{\R}\widehat{g}(x)\abs{\sum_{k\in\Z}a_ke^{2i\pi\lambda_k t}}^2\,\mbox{d}t\\
&=&4 \sum_{j,k\in\Z}a_j\overline{a_k}\int_{\R}\widehat{g}(t)e^{2i\pi(\lambda_j-\lambda_k)t}\,\mbox{d}t\\
&=&4 \sum_{j,k\in\Z}a_j\overline{a_k}g(\lambda_j-\lambda_k).
\end{eqnarray*}
Note that the sums are actually finite. Further, if $j\not=k$ then $|\lambda_j-\lambda_k|\geq 1$
and, as $g$ is supported in $[-1,1]$, we then have $g(\lambda_j-\lambda_k)=0$. This implies that
$$
\int_{-1/2}^{1/2}|P(t)|^2\,\mbox{d}t\leq 4g(0)\sum_{j\in\Z}|a_j|^2
$$
so that the inequality is proven for $T=1$ since $4g(0)=2$.

For $T<1$ we simply write
$$
\frac{1}{T}\int_{-T/2}^{T/2}|P(t)|^2\,\mbox{d}t\leq\frac{1}{T}\int_{-1/2}^{1/2}|P(t)|^2\,\mbox{d}t
\leq  \frac{2}{T}\sum_{j\in\Z}|a_j|^2.
$$

To conclude, notice first that, if $I=[a-1/2,a+1/2]$ and $P(t)=\dst\sum_{j\in\Z}a_je^{2i\pi\lambda_j t}$
then
\begin{eqnarray*}
\int_{I}|P(t)|^2\,\mbox{d}t&=&\int_{-1/2}^{1/2}|P(a+t)|^2\,\mbox{d}t
=\int_{-1/2}^{1/2}\abs{\sum_{j\in\Z}a_je^{2i\pi\lambda_j a}e^{2i\pi\lambda_j t}}^2\,\mbox{d}t\\
&\leq& 2\sum_{j\in\Z}|a_je^{2i\pi\lambda_j a}|^2=2\sum_{j\in\Z}|a_j|^2
\end{eqnarray*}
from the case $T=1$.

Now let $T>1$ and cover the interval $[-T/2,T/2]$ by $K=\lceil T\rceil\leq T+1$ intervals $I_1,\ldots,I_K$ of length 1.
Then
$$
\frac{1}{T}\int_{-T/2}^{T/2}|P(t)|^2\,\mbox{d}t\leq\frac{1}{T}\sum_{j=1}^K\int_{I_j}|P(t)|^2\,\mbox{d}t
\leq 2\frac{T+1}{T}\sum_{j\in\Z}|a_j|^2.
$$
This completes the proof.
\end{proof}

We now show that a converse inequality also holds, but this time with the extra condition that $T>1$:

\begin{proposition}[Ingham's converse inequality]
Let $(a_k)_{k\in\Z}$ be a finitely supported sequence of
complex numbers and $(\lambda_k)_{k\in\Z}$ be a sequence of real numbers with
$\lambda_{k+1}-\lambda_k\geq 1$. For every $T>1$, 
\begin{equation}
\label{eq:ingham}
\frac{1}{T}\int_{-T/2}^{T/2}\abs{\sum_{k\in\Z} a_ke^{2i\pi\lambda_kt}}^2\,\mbox{d}t\geq C(T)\sum_{k\in\Z}|a_k|^2
\end{equation}
with
$$
C(T)=\begin{cases}
\dfrac{\pi^2}{8}\dfrac{T^2-1}{T^3}&\mbox{for }1<T\leq 2\\
\dfrac{3\pi^2}{64}&\mbox{for }T\geq 2\\
\end{cases}.
$$
\end{proposition}

\begin{proof}
We will do so in three steps. We first establish this inequality for $1<T\leq 2$.

Let $H$ be defined by $H(x)=\mathbbm{1}_{[-1/2,1/2]}\cos \pi x$ and notice that
$$
\widehat{H}(t)=\frac{2}{\pi}\frac{\cos\pi t}{1-4t^2}.
$$
Note that $H$ being non-negative, $\widehat{H}$ is maximal at $0$ (which may be checked directly).

Next consider $H*H$ and notice that, as $H$ is non-negative, even, continuous with support $[-1/2,1/2]$,
then $H*H$ is non-negative, even, continuous with support $[-1,1]$ and its Fourier transform is
$$
\widehat{H*H}=\widehat{H}^2.
$$
Next $H\in H^1(\R)$ with $H'=-\pi\mathbbm{1}_{[-1/2,1/2]}\sin \pi x$ and
$$
\widehat{H'}(t)=4it\frac{\cos\pi t}{1-4t^2}
$$
thus
$$
\widehat{H'*H'}(t)=-(2\pi t)^2\widehat{H}^2(t)
$$
We now consider
$
G_T=\pi^2 T^2 H*H+H'*H'
$
so that $G_T$ is continuous, real valued, even and supported in $[-1,1]$.
$$
\widehat{G_T}(t)=\pi^2\left(T^2-4t^2\right)\widehat{H}^2(t)
$$
is even (so $G_T$ is the Fourier transform of $\widehat{G_T}$)
and in $L^1$. 
Further $\widehat{G}$ is non-negative on $[-T/2,T/2]$ and negative on $\R\setminus[-T/2,T/2]$.

This implies that
\begin{eqnarray*}
\int_{-T/2}^{T/2}\widehat{G_T}(t)\abs{\sum_{k\in\Z}a_ke^{2i\pi\lambda_k t}}^2\,\mbox{d}t
&\geq&\int_{\R}\widehat{G_T}(t)\abs{\sum_{k\in\Z}a_ke^{2i\pi\lambda_k t}}^2\,\mbox{d}t\\
&=&\sum_{k,\ell\in\Z}a_k\overline{a_\ell}\int_{\R}\widehat{G_T}(t)e^{2i\pi(\lambda_k-\lambda_\ell)t}\,\mbox{d}t\\
&=&\sum_{k,\ell\in\Z}a_k\overline{a_\ell}G_T(\lambda_k-\lambda_\ell)
=\sum_{k\in\Z}|a_k|^2G_T(0).
\end{eqnarray*}
In the last line, we use that $|\lambda_k-\lambda_\ell|\geq 1$ when $k\not=\ell$ thus 
$G_T(\lambda_k-\lambda_\ell)=0$.

Now, for $\xi\in[-T/2,T/2]$,
$$
\widehat{G_T}(\xi)=\pi^2(T^2-4\xi^2)\widehat{H}^2(\xi)
\leq \pi^2(T^2-4\xi^2)\widehat{H}^2(0)\leq 4T^2
$$
while
$$
G_T(0)=\pi^2\int_{-1/2}^{1/2}T^2\cos^2\pi t-\sin^2\pi t\,\mbox{d}t
=\frac{\pi^2}{2}(T^2-1)
$$
which leads to
\begin{equation}
\int_{-T/2}^{T/2}\abs{\sum_{k\in\Z}a_ke^{2i\pi\lambda_k t}}^2\,\mbox{d}t
\geq\frac{\pi^2}{8}\frac{T^2-1}{T^2} \sum_{k\in\Z}|a_k|^2.
\label{eq:ingham1}
\end{equation}

\smallskip

For $2\leq T\leq 6$, we simply write
$$
\int_{-T/2}^{T/2}\abs{\sum_{k\in\Z}a_kje^{2i\pi\lambda_k t}}^2\,\mbox{d}t
\geq \int_{-1}^{1}\abs{\sum_{k\in\Z}a_ke^{2i\pi\lambda_k t}}^2\,\mbox{d}t
\geq\frac{3\pi^2}{32}\sum_{k\in\Z}|a_k|^2
$$
where the second inequality is \eqref{eq:ingham1} with $T=2$, establishing
\eqref{eq:ingham} with $C=\dfrac{3\pi^2}{32T}\geq\dfrac{\pi^2}{64}$.

\smallskip

Now let $T\geq 6$ and $M_T=[T/2]$ so that $M_T\geq \dfrac{T}{2}-1\geq \dfrac{T}{3}$. For $j=0,\ldots,M_T-1$, let 
$t_j=-T/2+j+1$ so that the intervals
$[t_j-1,t_j+1[$ are disjoint and $\dst\bigcup_{j=0}^{N-1}[t_j-1,t_j+1[\subset [-T/2,T/2]$ thus
\begin{eqnarray*}
\int_{-T/2}^{T/2}\abs{\sum_{k\in\Z}b_ke^{2i\pi\lambda_k t}}^2\,\mbox{d}t
&\geq&\sum_{j=0}^{M_T-1}\int_{t_j-1}^{t_j+1}\abs{\sum_{k\in\Z}b_ke^{2i\pi\lambda_k t}}^2\,\mbox{d}t\\
&=&\sum_{j=0}^{M_T-1}\int_{-1}^{1}\abs{\sum_{k\in\Z}b_ke^{2i\pi\lambda_k t_j}e^{2i\pi\lambda_k t}}^2\,\mbox{d}t.
\end{eqnarray*}
Now, apply \eqref{eq:ingham1} with $a_k=b_ke^{2i\pi\lambda_k t_j}$ and $T=2$ to get
$$
\frac{1}{T}\int_{-T/2}^{T/2}\abs{\sum_{j=1}^Nb_ke^{2i\pi\lambda_k t}}^2\,\mbox{d}t\geq 
\frac{3\pi^2}{32}\dfrac{M_T}{T}\sum_{k=1}^N|a_k|^2\geq \frac{\pi^2}{32}\sum_{k=1}^N|a_k|^2,
$$
establishing \eqref{eq:ingham} with $C=\dfrac{\pi^2}{32}$.
\end{proof}

Finally, we notice that, with a change of variable, and a simple limiting argument to remove the condition on the support of $(a_k)$,  we have just proved the following;

\begin{theorem}[Ingham]
Let $\gamma>0$ and $T>\dfrac{1}{\gamma}$.
There exist two constants $C=C(\gamma T)$, $B=B(\gamma T)$ such that

-- for every sequence of real numbers  $\Lambda=\{\lambda_k\}_{k\in\Z}$ such that $\lambda_{k+1}-\lambda_k\geq\gamma$;

-- for every sequence $(a_k)_{k\in\Z}\in\ell^2(\Z,\C)$,

\begin{equation}
\label{eq:inghamtg}
C\sum_{k\in\Z}|a_k|^2\leq \frac{1}{T}\int_{-T/2}^{T/2}\abs{\sum_{k\in\Z}a_ke^{2i\pi\lambda_k t}}^2\,\mbox{d}t
\leq B\sum_{k\in\Z}|a_k|^2.
\end{equation}
\end{theorem}

\subsection{The condition $T>1$}

We now show that the condition $T>1$ is mandatory to establish \eqref{eq:ingham}
for every $\Lambda$ and every $P\in\pp(\Lambda)$.

\begin{proposition}[Ingham]
There exists a real sequence $\{\lambda_k\}_{k\in\Z}$ with $\lambda_{k+1}-\lambda_k\geq 1$
and a sequence $(a_k)_{k\in\Z}$ with $\dst\sum_{k\in\Z}|a_k|^2=1$ such that, 
\begin{equation}
\label{eq:inghamcont}
\int_{-1/2}^{1/2}\abs{\sum_{k=-N}^N a_ke^{2i\pi\lambda_kt}}^2\,\mbox{d}t\to 0
\end{equation}
when $N\to+\infty$.
\end{proposition} 

In other words, for \eqref{eq:ingham} to hold for $T=1$ for this sequence $(\lambda_k)$
and every sequence $(a_k)$ then $C(1)=0$ so that the condition $T>1$ is necessary
for this inequality to be interesting.

\begin{proof}
Let $\alpha>0$ and define, for $|z|<1$,
$$
g_\alpha(z)=(1+z)^{-\alpha}
=\frac{\exp\left(-i\alpha\arctan\frac{\Im(z)}{1+\Re(z)}\right)}{|1+z|^\alpha}\quad,\quad|z|<1.
$$
Of course, we may also write $g_\alpha$ as a power series
$$
g_\alpha(z)=
\sum_{n=0}^{+\infty} \frac{(-\alpha)_n}{n!}z^n
$$
where $(\alpha)_0=1$, $(-\alpha)_n=-\alpha(-\alpha-1)\cdots(-\alpha-n+1)$.

Next define
\begin{eqnarray*}
f_\alpha(r,t)&=&2\Re\bigl(e^{i\pi(\alpha+1)t}g_\alpha(re^{2i\pi t})\bigr)\\
&=&e^{i\pi(\alpha+1)t}g_\alpha(re^{2i\pi t})+e^{-i\pi(\alpha+1)t}g_\alpha(re^{-2i\pi t})\\
&=&\sum_{n=0}^{+\infty}\frac{(-\alpha)_n}{n!}r^n e^{2i\pi\left(n+\frac{\alpha+1}{2}\right)t}
+\sum_{n=0}^{+\infty} \frac{(-\alpha)_n}{n!}r^ne^{-2i\pi\left(n+\frac{\alpha+1}{2}\right)t}.
\end{eqnarray*}

Now set $\Lambda=\{\lambda_j\}_{j\in\Z}$ with $\lambda_j=j+\dfrac{\alpha+1}{2}$ when $j\geq 0$
and $\lambda_j=j+1-\dfrac{\alpha+1}{2}$ for $j\leq -1$, then $\lambda_{j+1}-\lambda_j\geq 1$
(and even $=1$ excepted for  $\lambda_0-\lambda_{-1}=1+\alpha$). In particular,
if we set
$$
P_{m,r}(t)=\sum_{n=0}^{m} \frac{(-\alpha)_n}{n!}r^ne^{2i\pi\left(n+\frac{\alpha+1}{2}\right)t}
+\sum_{n=0}^{m} \frac{(-\alpha)_n}{n!}r^ne^{-2i\pi\left(n+\frac{\alpha+1}{2}\right)t}
:=\sum_{k\in\Z}a_{m,r}(k)e^{2i\pi\lambda_k t}\in\pp(\Lambda)
$$
and $P_{m,r}\to f_\alpha$ when $m\to+\infty$, uniformly over $t\in[-1/2,1/2]$.

Further, Parseval's relation reads
$$
\sum_{n=0}^{+\infty}\abs{\frac{(-\alpha)_n}{n!}r^n}^2
=\int_{-1/2}^{1/2}\abs{\sum_{n=0}^{+\infty}\frac{(-\alpha)_n}{n!}r^ne^{2i\pi nt}}^2\,\mbox{d}t
=\int_{-1/2}^{1/2}|g_\alpha(re^{2i\pi t})|^2\,\mbox{d}t
$$
thus
$$
\lim_{m\to+\infty}\sum_{k\in\Z}|a_{m,r}(k)|^2=
\lim_{m\to+\infty}2\sum_{n=0}^{+\infty}\abs{\frac{(\alpha)_n}{n!}r^n}^2
=2\int_{-1/2}^{1/2}|g_\alpha(re^{2i\pi t})|^2\,\mbox{d}t.
$$
It follows that, if we had
\begin{equation}
\label{eq:inghamcont1}
\int_{-1/2}^{1/2}|P_{m,r}(t)|^2\,\mbox{d}t\geq C\sum_{k\in\Z}|a_{m,r}(k)|^2
\end{equation}
then, letting $m\to+\infty$, we would also have
\begin{equation}
\label{eq:inghamcont2}
\int_{-1/2}^{1/2}|f_\alpha(r,t)|^2\,\mbox{d}t\geq 2C\int_{-1/2}^{1/2}|g_\alpha(re^{2i\pi t})|^2\,\mbox{d}t
\end{equation}
for every $0<r<1$ and every $0<\alpha<1$.

But, if we fix $t\in]-1/2,1/2[$ then, when $r\to 1$,
$$
g_\alpha(re^{\pm 2i\pi t})=\frac{1}{\bigl(1+re^{\pm 2i\pi t}\bigr)^\alpha}
\to\frac{1}{\bigl(1+e^{\pm 2i\pi t}\bigr)^\alpha}=\frac{e^{\mp i\alpha \pi t}}{2^\alpha\cos^\alpha\pi t}
$$
(this is where we use that $T\leq 1$) while
$$
|g_\alpha(re^{\pm 2i\pi t})|^2=\frac{1}{\bigl((1-r)^2+4r\cos^2\pi t)^\alpha}\leq\frac{1}{4\cos^{2\alpha}\pi t}
$$
for $\frac{1}{2}<r<1$. Similar bounds follow for $f_\alpha(r,t)$:
$$
f_\alpha(r,t)=e^{i\pi(\alpha+1)t}g_\alpha(re^{2i\pi t})+e^{-i\pi(\alpha+1)t}g_\alpha(re^{-2i\pi t})
\to\frac{e^{i\pi t}+e^{-i\pi t}}{2^\alpha\cos^\alpha\pi t}=\frac{1}{2^{\alpha-1}\cos^{\alpha-1}\pi t}
$$
while
$$
|f_\alpha(r,t)|^2\leq\frac{1}{\cos^{2\alpha}\pi t} .
$$

When $2\alpha<1$ the majorants are integrable
so that we can let $r\to 1$ in \eqref{eq:inghamcont2}. This leads to
\begin{equation}
\label{eq:inghamcont3}
2^{2-2\alpha}\int_{-1/2}^{1/2}\frac{\mbox{d}t}{\cos^{2\alpha-2}\pi t}\geq 2^{1-2\alpha}C\int_{-1/2}^{1/2}\frac{\mbox{d}t}{\cos^{2\alpha}\pi t}.
\end{equation}
Letting $\alpha\to\dfrac{1}{2}$ the right hand side goes to $+\infty$
while the left hand side stays bounded. We obtain the desired contradiction.
\end{proof}

\subsection{Sequences with large gaps}

\begin{lemma}[Haraux]
Let $\Lambda\subset\R$ be a sequence such that
$\gamma(\Lambda)=\inf_{k\not=\ell\in\Z}|\lambda_k-\lambda_\ell|>1$
and let $T>\gamma(\Lambda)^{-1}$.

Assume that there exists $0<C\leq 1\leq B$ such that, for every $(a_\lambda)_{\lambda\in\Lambda}$
finitely supported sequence of complex numbers,
$$
C\sum_{\lambda\in\Lambda}|a_\lambda|^2
\leq\frac{1}{T}\int_{-T/2}^{T/2}\abs{\sum_{\lambda\in\Lambda}a_\lambda e^{2i\pi\lambda t}}\,\mathrm{d}t
\leq B\sum_{\lambda\in\Lambda}|a_\lambda|^2.
$$

Let $\mu\in\R\setminus\Lambda$ and, for sake of simplicity, assume that $\gamma(\Lambda\cup\{\mu\})\geq 1$.
Let $0<\delta<\min\left(T,\dfrac{1}{4}\right)$, then there is a $D$ with $0<D<C$ such that,
for every $(a_\lambda)_{\lambda\in\Lambda\cup\{\mu\}}$
$$
C\sum_{\lambda\in\Lambda}|a_\lambda|^2
\leq\frac{1}{T+\delta}\int_{-\frac{T+\delta}{2}}^{\frac{T+\delta}{2}}\abs{\sum_{\lambda\in\Lambda\cup\{\mu\}}
a_\lambda e^{2i\pi\lambda t}}\,\mathrm{d}t
$$
Moreover, $D$ can be taken of the form
$$
D=\kappa\frac{C}{B}\delta^4
$$
for some absolute constant $\kappa>0$ ($\kappa=8^{-4}$ would do).
\end{lemma}

Note that, as $\delta$ is arbitrarily small, as
$T>\gamma(\Lambda)^{-1}$ is arbitrarily near to $\gamma(\Lambda\setminus\{\lambda_0\})^{-1}$ 
so is $T+\delta$.

\begin{proof}
Set
$$
f(t)=\sum_{\lambda\in\Lambda\cup\{\mu\}}a_\lambda e^{2i\pi\lambda t}.
$$
For $0<\delta<\dfrac{1}{4}$, we define
\begin{eqnarray*}
g_\delta(t)&=&f(t)-\frac{1}{\delta}\int_{-\delta/2}^{\delta/2} e^{-2i\pi\mu s}f(t+s)\,\mbox{d}s\\
&=&\sum_{j\in\Lambda}a_\lambda\left(1-\frac{\sin \pi(\lambda-\mu)\delta}{\pi(\lambda-\mu)\delta}\right)
e^{2i\pi\lambda t}:=\sum_{\lambda\in\Lambda}b_\lambda(\delta)e^{2i\pi\lambda t}.
\end{eqnarray*}

The first observation is that, as $|\lambda-\mu|\geq 1$,
$$
\left|1-\frac{\sin \pi(\lambda-\mu)\delta}{\pi(\lambda-\mu)\delta}\right|^2
\geq \eta_\delta(\mu):=\sup_{s\geq \pi \delta}\left|1-\frac{\sin s}{s}\right|^2>0.
$$
It is not difficult to see that, if $\delta\leq\dfrac{1}{4}$, $\eta_\delta\geq\dfrac{\delta^4}{100}$.
This is what needs to be adjusted if we do not assume that $\mathrm{dist}(\mu,\Lambda)=\min_{\lambda\in\Lambda}|\lambda-\mu|\geq 1$.

Further, as for $\lambda\not=\lambda'\in\Lambda$, then $|\lambda-\lambda'|\geq\gamma(\Lambda)$, the hypothesis of the lemma states that, as $T>\dfrac{1}{\gamma(\Lambda)}$,
\begin{equation}
\label{eq:haraux1}
\frac{1}{T}\int_{-T/2}^{T/2}|g_\delta(t)|^2\,\mbox{d}t
\geq C\sum_{\lambda\in\Lambda}|b_\lambda(\delta)|^2
\geq C\eta_\delta\sum_{\lambda\in\Lambda}|a_\lambda|^2
\geq \frac{C\delta^4}{100}\sum_{\lambda\in\Lambda}|a_\lambda|^2
\end{equation}

On the other hand, from Cauchy-Schwarz,
\begin{eqnarray*}
|g_\delta(t)|^2&\leq& 2|f(t)|^2
+2\abs{\frac{1}{\delta}\int_{-\delta/2}^{\delta/2} e^{-2i\pi\lambda_0 s}f(t+s)\,\mbox{d}s}^2\\
&\leq&2|f(t)|^2+2\frac{1}{\delta}\int_{-\delta/2}^{\delta/2} |f(t+s)|^2\,\mbox{d}s\\
&=&2|f(t)|^2+\frac{1}{\delta}\int_{t-\delta}^{t+\delta} |f(s)|^2\,\mbox{d}s.
\end{eqnarray*}
It follows that
\begin{eqnarray*}
\frac{1}{T}\int_{-T/2}^{T/2}|g_\delta(t)|^2\,\mbox{d}t
&\leq&\frac{2}{T}\int_{-T/2}^{T/2}|f(t)|^2\,\mbox{d}t
+\frac{2}{\delta}\frac{1}{T}\int_{-T/2}^{T/2}\int_{t-\delta}^{t+\delta} |f(s)|^2,\mbox{d}s\,\mbox{d}t\\
&=&\frac{2}{T}\int_{-T/2}^{T/2}|f(t)|^2\,\mbox{d}t
+\frac{2}{T}\int_{-T/2-\delta}^{T/2+\delta} |f(s)|^2\frac{1}{\delta}
\int_{\max(-T/2,s-\delta/2)}^{\min(T/2,s+\delta/2)}
\,\mbox{d}t\,\mbox{d}s\\
&\leq&\frac{2}{T}\int_{-T/2}^{T/2}|f(t)|^2\,\mbox{d}t
+\frac{2}{T}\int_{-T/2-\delta/2}^{T/2+\delta/2} |f(s)|^2\,\mbox{d}s\\
&\leq&4\frac{1}{T}\int_{-\frac{T+\delta}{2}}^{\frac{T+\delta}{2}} |f(s)|^2\,\mbox{d}s
\leq 8\frac{1}{T+\delta}\int_{-\frac{T+\delta}{2}}^{\frac{T+\delta}{2}} |f(s)|^2\,\mbox{d}s
\end{eqnarray*}
since we assumed that $T+\delta\leq 2T$.

With \eqref{eq:haraux1}, we have thus shown that
\begin{equation}
\label{eq:haraux2}
\sum_{\lambda\in\Lambda}|a_\lambda|^2\leq \frac{800}{C\delta^4}\frac{1}{T+\delta}\int_{-\frac{T+\delta}{2}}^{\frac{T+\delta}{2}} |f(s)|^2\,\mbox{d}s.
\end{equation}

It remains to add $|a_\mu|^2$ on the left hand side. As 
$\dst a_\mu=f-\sum_{\lambda\in\Lambda}a_\lambda e^{2i\pi\lambda t}$,
\begin{eqnarray*}
|a_\mu|^2&=&\frac{1}{T}\int_{-\frac{T}{2}}^{\frac{T}{2}}|a_\mu|^2\,\mbox{d}t
\leq 2\frac{1}{T}\int_{-\frac{T}{2}}^{\frac{T}{2}}|f(t)|^2\,\mbox{d}t
+2\frac{1}{T}\int_{-\frac{T}{2}}^{\frac{T}{2}}
\abs{\sum_{\lambda\in\Lambda}a_\lambda e^{2i\pi\lambda t}}^2\,\mbox{d}t\\
&\leq&4\frac{1}{T+\delta}\int_{-\frac{T+\delta}{2}}^{\frac{T+\delta}{2}}|f(t)|^2\,\mbox{d}t
+2B\sum_{\lambda\in\Lambda}|a_\lambda|^2
\end{eqnarray*}
where we used again that $T+\delta\leq 2T$ and the hypothesis of the lemma.
But then, from \eqref{eq:haraux2}, we get
$$
|a_\mu|^2\leq \left(4+ \frac{1600B}{C\delta^4}\right)
\frac{1}{T+\delta}\int_{-\frac{T+\delta}{2}}^{\frac{T+\delta}{2}} |f(s)|^2\,\mbox{d}s.
$$
Adding this to \eqref{eq:haraux2}, we finally get that
\begin{eqnarray*}
\sum_{\lambda\in\Lambda\cup\{\mu\}}|a_\lambda|^2&\leq&
\left(4+ \frac{800\bigl(2B+1\bigr)}{C\delta^4}\right)
\frac{1}{T+\delta}\int_{-\frac{T+\delta}{2}}^{\frac{T+\delta}{2}} |f(s)|^2\,\mbox{d}s\\
&\leq&
\frac{1}{\kappa}\frac{B}{C\delta^4}
\frac{1}{T+\delta}\int_{-\frac{T+\delta}{2}}^{\frac{T+\delta}{2}} |f(s)|^2\,\mbox{d}s
\end{eqnarray*}
with $\kappa^{-1}=800\times 3+4$ since $C(\gamma_0T)\leq B(\gamma_0T)$, $\delta\leq 1$ and $B(\gamma_0T)\geq 1$.
We thus obtain the inequality with
$$
D=\kappa\frac{C}{B}\delta^4
$$
as claimed.
\end{proof}

We can now use this lemma to improve Ingham's Inequality. 

\begin{corollary}
Let $\Lambda$ be a sequence such that $\gamma(\Lambda)=\min_{\lambda\not=\lambda'\in\Lambda}|\lambda-\lambda'|\geq 1$ 
and let $F=(\lambda_1,\ldots,\lambda_n)\subset \Lambda$.

Let $S>\dfrac{1}{\gamma(\Lambda\setminus F)}$. Then there exists two constants
$0<D<1<B$ such that, for every finitely supported sequence $(a_\lambda)_{\lambda\in\Lambda}$,
\begin{equation}
\label{eq:haraux3}
D\sum_{\lambda\in\Lambda}|a_\lambda|^2\leq
\frac{1}{S}\int_{-S/2}^{S/2}\abs{\sum_{\lambda\in\Lambda}a_\lambda e^{2i\pi\lambda t}}^2\ud t
\leq B\sum_{\lambda\in\Lambda}|a_\lambda|^2.
\end{equation}
\end{corollary}

\begin{proof}
Set $\Lambda_0=\Lambda$ and, for $j=1,\ldots,n$,
$\Lambda_j=\Lambda\setminus\{\lambda_1,\ldots,\lambda_n\}$ and
$\gamma_j=\min_{\lambda\not=\lambda'\in\Lambda_j}|\lambda-\lambda'|\geq 1$.

From Ingham's Theorem, we know that

-- for every $2>T>\dfrac{1}{\gamma_n}$ 
$$
\dfrac{1}{T}\int_{-T/2}^{T/2}\abs{\sum_{\lambda\in\Lambda_n}a_\lambda e^{2i\pi\lambda t}}^2\ud t
\geq C_n\sum_{\lambda\in\Lambda_n}|a_\lambda|^2
$$
with $C_n=\dfrac{\pi^2}{2^6}(T\gamma_n-1)$;

-- for every $2>T>0$, and $j=0,\ldots,n$,
$$
\dfrac{1}{T}\int_{-T/2}^{T/2}\abs{\sum_{\lambda\in\Lambda_j}a_\lambda ae^{2i\pi\lambda t}}^2\ud t
\leq B_j\sum_{\lambda\in\Lambda_j}|a_\lambda|^2
$$
with $B_j=\dfrac{6(T\gamma_j+1)}{T\gamma_j}=6\left(1+\dfrac{1}{\gamma_jT}\right)$.

In particular, the upper bound in \eqref{eq:haraux3} is already established.

Now let $\dfrac{1}{\gamma(\Lambda\setminus F)}<T<S$ and $\delta=\dfrac{S-T}{n}$.
One might take $T=\dfrac{1}{2}\left(S+\dfrac{1}{\gamma(\Lambda\setminus F)}\right)$ so that
$\delta=\dfrac{1}{2n}\left(S-\dfrac{1}{\gamma(\Lambda\setminus F)}\right)$

From Hauraux's lemma, we get
$$
\dfrac{1}{T+\delta}\int_{-\frac{T+\delta}{2}}^{\frac{T+\delta}{2}}
\abs{\sum_{\lambda\in\Lambda_{n-1}}a_\lambda e^{2i\pi\lambda t}}^2\ud t
\geq C_{n-1}\sum_{\lambda\in\Lambda_{n-1}}|a_\lambda|^2
$$
with $C_{n-1}=\kappa\delta^4\dfrac{C_n}{B_n}$. We can therefore apply Haraux's lemma again, till we reach $\Lambda_0$.
At each step, the constant $C_{n-j}$ is replace by $C_{n-j-1}=\kappa\delta^4\dfrac{C_{n-j}}{B_{n-j}}$ while the integral ranges over $\ent{-\frac{T+(j+1)\delta}{2},\frac{T+(j+1)\delta}{2}}$.

This finally leads to
$$
\dfrac{1}{T+n\delta}\int_{-\frac{T+n\delta}{2}}^{\frac{T+n\delta}{2}}
\abs{\sum_{\lambda\in\Lambda_{n-1}}a_\lambda e^{2i\pi\lambda t}}^2\ud t
\geq C_0\sum_{\lambda\in\Lambda}|a_\lambda|^2
$$
with
$$
C_0=\kappa^n\delta^{4n}\dfrac{C_n}{B_1B_2\cdots B_{n}}.
$$
As $T+n\delta=S$, this establishes \eqref{eq:haraux3}.
\end{proof}

Note that, taking $\kappa=8^{-4}$ we obtain
$$
D=\frac{\pi^2}{2^6}\left(\frac{3\delta^4}{2^{13}}\right)^{n}\frac{\gamma_nT-1}{
\prod_{j=0}^n\left(1+\frac{1}{\gamma_jT}\right)}.
$$

\section{$L^1$ estimates with integer frequencies}

\subsection{The $L^1$ norm of the Dirichlet kernel}

Recall that the Dirichlet kernel is given by
$$
D_N(t)=\sum_{k=-N}^Ne^{2i\pi kt}=\frac{\sin (2N+1)\pi t}{\sin\pi t}.
$$
The following is a classical estimate of the $L^1$-norm of this kernel, which is called the
{\em Lebesgue constant}.

\begin{lemma}
When $N\to+\infty$,
$$
\norm{D_N}_1:=\int_{-1/2}^{1/2}|D_N(t)|\,\mathrm{d}t=\frac{4}{\pi^2}\ln N+O(1).
$$
\end{lemma}

\begin{proof}
The proof is based on the following inequality which is easy to establish and whose proof is left to the reader:
for $-\dfrac{\pi}{2}\leq s\leq \dfrac{\pi}{2}$, $\dst |s|\left(1-\frac{s^2}{3!}\right)\leq |\sin s|\leq |s|$.
In particular, for $|t|\leq\dfrac{1}{2}$,
$$
\frac{1}{\pi |t|}\leq\frac{1}{|\sin\pi t|}\leq\frac{1}{\pi |t|}\frac{1}{1-\pi^2t^2/6}
\leq\frac{1}{|\pi t|}+\frac{\pi |t|}{3}
$$ 
since $\frac{1}{1-u}\leq 1+2u$ for $|u|\leq \dfrac{1}{2}$. 
It follows that
\begin{eqnarray*}
\int_{-1/2}^{1/2}\frac{|\sin(2N+1)\pi t|}{|\pi t|}\,\mathrm{d}t
&\leq& \int_{-1/2}^{1/2}\frac{|\sin(2N+1)\pi t|}{\sin \pi t}\,\mathrm{d}t\\
&\leq&
\int_{-1/2}^{1/2}\frac{|\sin(2N+1)\pi t|}{|\pi t|}\,\mathrm{d}t
+\int_{-1/2}^{1/2}|\sin(2N+1)\pi t|\frac{\pi |t|}{3}\,\mathrm{d}t.
\end{eqnarray*}
As
$$
\int_{-1/2}^{1/2}|\sin(2N+1)\pi t|\frac{\pi |t|}{3}\,\mathrm{d}t\leq
\int_{-1/2}^{1/2}\frac{\pi |t|}{3}\,\mathrm{d}t=\frac{\pi}{12}
$$
we get
$$
\norm{D_N}_1=\int_{-1/2}^{1/2}\frac{|\sin(2N+1)\pi t|}{|\pi t|}\,\mathrm{d}t+O(1).
$$
Using parity and the change of variable $s=(2N+1)\pi t$ we obtain
$$
\int_{-1/2}^{1/2}\frac{|\sin(2N+1)\pi t|}{|\pi t|}\,\mathrm{d}t=
2\int_0^{1/2}\frac{|\sin(2N+1)\pi t|}{|\pi t|}\,\mathrm{d}t
=\frac{2}{\pi}\int_0^{N\pi+\pi/2}\frac{|\sin s|}{s}\,\mathrm{d}s.
$$
Note that 
$$
\int_{N\pi}^{N\pi+\pi/2}\frac{|\sin s|}{s}\,\mathrm{d}s\leq \int_{N\pi}^{N\pi+\pi/2}\frac{\mathrm{d}s}{s}
=\ln\frac{N\pi+\pi/2}{N\pi}=O(1/N)
$$
while
$$
\int_0^\pi\frac{|\sin s|}{s}\,\mathrm{d}s\leq \int_0^\pi 1\,\mathrm{d}s=\pi.
$$
It remains to estimate
\begin{eqnarray*}
\frac{2}{\pi}\int_\pi^{N\pi}\frac{|\sin s|}{s}\,\mathrm{d}s
&=&\frac{2}{\pi}\sum_{j=1}^{N-1}\int_{j\pi}^{(j+1)\pi}\frac{|\sin s|}{s}\,\mathrm{d}s\\
&=&\frac{2}{\pi}\sum_{j=1}^{N-1}\int_0^\pi\frac{|\sin s|}{s+j\pi}\,\mathrm{d}s
=\frac{2}{\pi}\sum_{j=1}^{N-1}\int_0^\pi\frac{\sin s}{s+j\pi}\,\mathrm{d}s
\end{eqnarray*}
as $\sin s\geq 0$ on $[0,\pi]$. But then
$$
\frac{2}{(j+1)\pi}=\int_0^\pi\frac{\sin s}{\pi+j\pi}\,\mathrm{d}s\leq
\int_0^\pi\frac{\sin s}{s+j\pi}\,\mathrm{d}s\leq \int_0^\pi\frac{\sin s}{j\pi}\,\mathrm{d}s=\frac{2}{j\pi}.
$$
Writing
$$
\sum_{j=1}^{N-1}\frac{1}{j+1}=\sum_{j=1}^N\frac{1}{j}-1=\sum_{j=1}^{N-1}\frac{1}{j}-1+\frac{1}{N},
$$
it follows that 
$$
\frac{2}{\pi}\int_\pi^{N\pi}\frac{|\sin s|}{s}\,\mathrm{d}s=\frac{4}{\pi^2}\sum_{j=1}^N\frac{1}{j}+O(1).
$$
Recalling that $\dst\sum_{j=1}^N\frac{1}{j}=\ln N+O(1)$ the result follows.
\end{proof}

\subsection{Lacunary trigonometric polynomials}

In this section, we consider a sequence of integers $(n_k)$ such that $\dfrac{n_{k+1}}{n_k}\geq q>1$.
Such sequences are called $q$-lacunary in the sense of Hadamard (or simply lacunary).

First note that a $q$-lacunary sequence is a finite union of $q'$-lacunary
sequences with $q'\geq 3$. Indeed, if $q\geq 3$ there is nothing to prove and for $1<q<3$, take $N$
an integer such that $q^N\geq 3$
and write $n^{(\ell)}_k=n_{\ell+kN}$ then $\bigl(n^{(\ell)}_k\bigr)_k$ is $q^N$-Lacunary
and $\{n_k\}=\bigcup_{\ell=0}^{N-1}\{n^{(\ell)}_k\}$.
Next, for $q\geq3$, $q$-lacunary sequences have a particular arithmetic property:

\begin{lemma}
Let $q\geq 3$ and $(n_k)_{k\geq 0}$ a sequence such that $n_0\geq 1$ and 
$n_{k+1}\geq qn_k$. Consider two finite sequences $\eps_\ell,\eta_\ell\in\{-1,0,1\}$
for $\ell=0,\ldots,m$ and assume that
\begin{equation}
\label{eq:quasi}
\sum_{\ell=0}^m\eps_\ell n_\ell=\sum_{\ell=0}^m\eta_\ell n_\ell
\end{equation}
then $\eps_\ell=\eta_\ell$ for every $\ell$.
\end{lemma}

In other words, an integer can be represented in at most one way as $\sum\pm n_\ell$.
Such a sequence is called {\em quasi-independent}.

\begin{proof}
First observe that $n_j\leq \dfrac{1}{q^{m-j}}n_m=\dfrac{q^j}{q^m}n_m$ for $j=0,\ldots,m$.

Assume that \eqref{eq:quasi} holds and define $\nu_\ell=\eps_\ell-\eta_\ell$ so that
$$
\sum_{\ell=0}^m\nu_\ell n_\ell=0.
$$
Assume towards a contradiction that there is an $\ell$ such that $\nu_{\ell}\not=0$.
Without loss of generality, we may assume that the largest such $\ell$ is $m$
and, up to exchanging $\eps_\ell$ and $\eta_\ell$, that $\nu_m\geq 1$.

Observe that $\nu_\ell\in\{-2,-1,0,1,2\}$ so that we obtain the desired contradiction writing
\begin{eqnarray*}
0=\sum_{\ell=0}^m\nu_\ell n_\ell&=&\nu_mn_m+\sum_{\ell=0}^{m-1}\nu_\ell n_\ell
\geq n_m-2\sum_{\ell=0}^{m-1}n_\ell
\geq n_m-2\sum_{\ell=0}^{m-1}\dfrac{q^\ell}{q^m}n_m\\
&=& \left(1-\frac{2}{q^m}\frac{q^m-1}{q-1}\right)n_m
=\frac{q^{m+1}-3q^m+2}{(q-1)q^m}n_m>0
\end{eqnarray*}
since $q\geq 3$ and $n_m>0$.
\end{proof}

Note that this result is valid when the $n_k$'s are real, not only for integers.

The aim of this section is to prove that trigonometric polynomials with lacunary frequencies 
have large $L^1$-norms of which the following estimate is a particular case: there exists a constant $C>0$
such that, for every $N$,
$$
\int_{-1/2}^{1/2}\left|\sum_{k=1}^N e^{2i\pi n_k t}\right|\,\mbox{d}t\geq C\sqrt{N}.
$$
This follows from a more general theorem
which estimates $L^p$-norms of lacunary Fourier series. The aim of this section
is to present this result. To do so, we follow closely \cite[Chapter V.8]{Zy}
which goes through Rademacher series. First let us introduce those series.

To start, let us recall that $\dd_k=\{[j2^{-k-1},(j+1)2^{-k-1}[,\ j=0,\ldots,2^{k+1}-1\}$
the dyadic intervals of generation $k$ and $\dd=\dst\bigcup_{k\geq 0}\dd_k$ the set of all dyadic intervals.
Also, if $I,J\in\dd$ then either $I\cap J=\emptyset$ or $I\subset J$ or $J\subset I$.
The Rademacher functions of generation $k$ are then functions that take alternative values $+1$ and $-1$
on successive intervals in $\dd_k$, that is
$$
r_k(t)=\sum_{j=0}^{2^{k+1}-1}(-1)^j\mathbbm{1}_{]j2^{-k-1},(j+1)2^{-k-1}[}=\mathrm{sign}\bigl(\sin (2\pi2^jt)\bigr).
$$
The first observation is that, if $I\in\dd_\ell$ and $k>\ell$ then $r_k$ takes the value $+1$ on half of $I$
and $-1$ on the other half so that $\int_Ir_k=0$. A first consequence is that
$r_k$ is orthogonal to $r_\ell$ in $L^2([0,1])$ since $r_\ell$ is constant on each $I\in\dd_\ell$ so that $\int_I r_kr_\ell=0$ and $\dd_\ell$ 
is a covering of $[0,1]$. Moreover, as $|r_k|=1$, the family $(r_k)_{k\geq 0}$
is an orthonormal sequence in $L^2([0,1])$.

In particular, we now fix a sequence $(c_j)_{j\geq 0}$ such that
$\dst\sum_{k=0}^{+\infty}|c_k|^2$ converges, we can define
$$
f=\sum_{k=0}^{+\infty}c_kr_k
$$
and this series converges in $L^2([0,1])$ thus $f\in L^2([0,1])$. We actually have a bit better:

\begin{theorem}\label{thm:aecvg}
If $\dst\sum_{k=0}^{+\infty}|c_k|^2<+\infty$ then  there exists $f\in L^2([0,1])$ defined by
$$
f=\sum_{k=0}^{+\infty}c_kr_k
$$
and this series converges both in $L^2([0,1])$ and almost everywhere.
\end{theorem}

\begin{proof}
The $L^2$ convergence has already been established. Further,
let $F=\int f$ be the indefinite integral of $f$ and let $E\subset[0,1]$ be the set of Lebesgue points of $f$
so that $|E|=1$ and on $E$, $F'$ exists and is finite.

Now let, $S_n[f]$ be the $n$-th partial sum of this series
$$
\ss_n[f](x)=\sum_{k=0}^nc_kr_k(x).
$$
As $\ss_n[f]\to f$ in $L^2([0,1])$, for every $0\leq a<b\leq 1$,
$$
\abs{\int_a^b \bigl(f(x)-\ss_n[f](x)\bigr)\ud x}\leq
\int_0^1|f(x)-\ss_n[f](x)|\ud x\leq \left(\int_0^1|f(x)-\ss[f](x)|^2\ud x\right)^{1/2}\to 0.
$$
We have just shown that, if $I$ is an interval, then $\int_I\ss_n[f]\to\int_If$ thus also, 
if we fix $j\geq 1$, $\int_I(\ss_n[f]-\ss_{j-1}[f])\to\int_I(f-\ss_{\ell-1}[f])$.
On the other hand, if $I\in\dd_{\ell-1}$ and $k\geq \ell$, then $\int_I r_k=0$ so that
$\int_I\ss_n[f]=\int_I\ss_{\ell-1}[f]$. Letting $n\to+\infty$ we obtain that
$$
\int_If(x)\ud x=\int_I\ss_{\ell-1}[f](x)\ud x\quad\mbox{for every}I\in\dd_{\ell-1}.
$$

Next, let $x_0\in E$ not a dyadic rational ($x_0\not=\dfrac{p}{2^q}$, $p,q\in\N$)
and let $I_k=]j2^{-k},(j+1)2^{-k}[$ be such that $x_0\in E\cap I_k$. Then, as $\ss_{k-1}[f]$ is constant over 
$I_k$
$$
\ss_{k-1}[f](x_0)=\dfrac{1}{|I_k|}\int_{I_k}\ss_k[f](x)\,\mbox{d}x=\dfrac{1}{|I_k|}\int_{I_k}f(x)\,\mbox{d}x
\to F'(x_0)
$$
when $k\to+\infty$.
\end{proof}

The second result is that $f$ is actually in every $L^p$ space:

\begin{theorem}\label{th:radlp}
Let $(c_k)\in\ell^2$ and $\dst f=\sum_{k=0}^{+\infty}c_kr_k$. Then, for $1\leq p<+\infty$, $f\in L^p([0,1])$.
Moreover, there exists $A_p,B_p$, depending on $p$ only, such that
$$
A_p\left(\sum_{k=0}^{+\infty}|c_k|^2\right)^{\frac{1}{2}}\leq\left(\int_0^1|f(x)|^p\,\mathrm{d}x\right)^{\frac{1}{p}}
\leq B_p\left(\sum_{k=0}^{+\infty}|c_k|^2\right)^{\frac{1}{2}}.
$$
\end{theorem}

\begin{proof}
Let us first notice that the theorem holds for $p=2$ since
$$
\gamma:=\left(\int_0^1|f(x)|^2\,\mathrm{d}x\right)^{\frac{1}{2}}=\left(\sum_{k=0}^{+\infty}|c_k|^2\right)^{\frac{1}{2}}
$$
{\it i.e.} the inequalities are equalities with $A_2=B_2=1$.

Next, let us notice that this implies the lower bound when $p>2$ with $A_p=1$ since then, with H\"older
$$
\left(\int_0^1|f(x)|^p\,\mathrm{d}x\right)^{\frac{1}{p}}\geq \left(\int_0^1|f(x)|^2\,\mathrm{d}x\right)^{\frac{1}{2}}=\gamma
$$
and also implies the upper bound with $B_p=1$ for $p<2$ since now H\"older implies that
$$
\left(\int_0^1|f(x)|^p\,\mathrm{d}x\right)^{\frac{1}{p}}\leq \left(\int_0^1|f(x)|^2\,\mathrm{d}x\right)^{\frac{1}{2}}=\gamma
$$
Further, as H\"older also implies that if $2(m-1)<p\leq2m$ for some integer $m\geq 2$, then
$$
\left(\int_0^1|f(x)|^p\,\mathrm{d}x\right)^{\frac{1}{p}}\leq \left(\int_0^1|f(x)|^{2m}\,\mathrm{d}x\right)^{\frac{1}{2m}}
$$
so that the upper bound 
\begin{equation}
\label{eq:upbdrad}
\left(\int_0^1|f(x)|^{2k}\,\mathrm{d}x\right)^{\frac{1}{2m}}
\leq B_{2m}\gamma
\end{equation}
would also implies the upper bound for $2(m-1)<p\leq2m$
with $B_p\leq B_{2k}$.

Next, let us show that the upper bound for $p=4$ implies the lower bound for $p<2$. Assume for the moment that we are able to prove that
$$
\left(\int_0^1|f(x)|^4\,\mathrm{d}x\right)^{\frac{1}{4}}\leq B_4\gamma.
$$
Let $1\leq q<2$ and write $2=qt+4(1-t)$, that is, take $t=\dfrac{2}{4-q}$. Then, from H\"older
\begin{eqnarray*}
\gamma^2&=&\int_0^1|f(x)|^2\,\mathrm{d}x=\int_0^1|f(x)|^{qt}|f(x)|^{4(1-t)}\,\mathrm{d}x
\leq\left(\int_0^1|f(x)|^q\,\mathrm{d}x\right)^{t}
\left(\int_0^1|f(x)|^4\,\mathrm{d}x\right)^{1-t}\\
&\leq&(B_4\gamma)^{4(1-t)}\left(\int_0^1|f(x)|^q\,\mathrm{d}x\right)^{t}
=(B_4\gamma)^{2-qt}\left(\int_0^1|f(x)|^q\,\mathrm{d}x\right)^{t}
\end{eqnarray*}
thus
$$
\left(\int_0^1|f(x)|^q\,\mathrm{d}x\right)^{\frac{1}{q}}\geq B_4^{1-\frac{4-q}{q}}
\gamma.
$$

So it remains to prove \eqref{eq:upbdrad}. Notice also that it is enough
to prove this inequality with real $c_k$'s. The constant in the complex case is then
multiplied by $2$: write $f=f_r+if_i$ where $f_r=\sum\Re(c_k)r_k$ and $f_i=\sum\Im(c_k)r_k$.
Then
$$
\|f\|_{2m}\leq \|f_r\|_{2m}+\|f_i\|_{2m}
\leq B_{2m}^\R\left[\left(\sum_{k=0}^{+\infty}|\Re(c_k)|^2\right)^{\frac{1}{2}}+\left(\sum_{k=0}^{+\infty}|\Im(c_k)|^2\right)^{\frac{1}{2}}\right]\leq 2B_{2m}^\R\gamma
$$
since $|\Re(c_k)|,|\Im(c_k)|\leq|c_k|$.

To conclude, we write
$$
\int_0^1S_n[f](x)^{2m}\ud x
=\sum_{\ell_0+\cdots+\ell_n=2m}A_{\ell_0,\ldots,\ell_n}c_0^{\ell_0}\cdots c_n^{\ell_n}
\int_0^1r_0^{\ell_0}(x)\cdots r_n^{\ell_n}(x)\ud x
$$
where $\ell_j\geq 0$ for every $j$ and
$$
A_{\ell_0,\ldots,\ell_n}=\frac{(\ell_0+\cdots+\ell_n)!}{\ell_0!\cdots\ell_j!}.
$$
Now observe that
$$
\int_0^1r_0^{\ell_0}(x)\cdots r_n^{\ell_n}(x)\ud x\begin{cases}1&\mbox{all the $\ell_j$'s are even}\\
0&\mbox{otherwise}\end{cases}
$$
and that
$$
\left(\sum_{k=0}^n c_k^2\right)^m
=\sum_{\ell_0+\cdots+\ell_n=m}A_{\ell_0,\ldots,\ell_n}(c_0^2)^{\ell_0}\cdots (c_n^2)^{\ell_n}.
$$
Further, when $\ell_0+\cdots+\ell_n=m$,
$$
\frac{A_{2\ell_0,\ldots,2\ell_n}}{A_{\ell_0,\ldots,\ell_n}}
=\frac{(m+1)(m+2)\cdots 2m}{\prod_{j=0}^n(\ell_j+1)(\ell_j+2)\cdots 2\ell_j}
\leq \frac{(m+1)(m+2)\cdots 2m}{2^m}\leq m^m
$$
(with the convention that on the denominator is $(\ell_j+1)(\ell_j+2)\cdots 2\ell_j=1$ when $\ell_j=0$).
It follows that
$$
\int_0^1S_n[f](x)^{2m}\ud x\leq m^m\left(\sum_{k=0}^n |c_k|^2\right)^m.
$$
As $S_n[f]\to f$ a.e., we conclude that 
$$
\left(\int_0^1|f(x)|^{2m}\,\mathrm{d}x\right)^{\frac{1}{2m}}
\leq m^{1/2}\left(\sum_{k=0}^{+\infty}|c_k|^2\right)^{\frac{1}{2}}
$$
that is $B_{2m}=2m^{1/2}$.
\end{proof}

The estimate $B_{2m}=2m^{1/2}$ allows to improve a bit the result:

\begin{corollary}
Let $(c_k)\in\ell^2$ and $\dst f=\sum_{k=0}^{+\infty}c_kr_j$.
Then, for every $\mu>0$, $\exp(\mu|f|^2)\in L^1([0,1])$.
\end{corollary}

\begin{proof}
Let us fix $\mu>0$. We first show that if $\gamma:=\|c_j\|_2$ is small enough, then
$\exp(\mu|f|^2)\in L^1([0,1])$. Indeed
\begin{equation}
\label{ex:estexp}
\int_0^1\exp(\mu|f(x)|^2)\ud x=\sum_{m=0}^{+\infty}\frac{\mu^m}{m!}\int_0^1|f(x)|^{2m}\ud x
\leq \sum_{m=0}^{+\infty}\frac{m^m}{m!}(4\mu\gamma^2)^m.
\end{equation}
But $\dst\frac{m^m}{m!}\leq\sum_{n=0}^{+\infty}\frac{m^n}{n!}=e^m$ so that
$$
\int_0^1\exp(\mu|f(x)|^2)\ud x\leq \sum_{m=0}^{+\infty}(4e\mu\gamma^2)^m=\frac{1}{1-4e\mu\gamma^2}<+\infty
$$
provided $\gamma^2<\dfrac{1}{4e\mu}$.

Next, take any $f\in L^1(0,1)$, and apply the first part to
$f-\ss_n[f]=\dst\sum_{j=n+1}^{+\infty}c_jr_j$. As $\gamma_n^2:=\dst\sum_{j=n+1}^{+\infty}|c_j|^2\to 0$,
for $n$ large enough $\gamma_n^2<\dfrac{1}{8e\mu}$ thus $\exp(2\mu|f-\ss_n[f]|^2)\in L^1([0,1])$.

Finally, as $|f|^2\leq 2|f-\ss_n[f]|+2|\ss_n[f]|^2$, we have
$$
\exp(\mu|f|^2)\leq \exp(2\mu|f-\ss_n[f]|^2)\exp(2\mu|\ss_n[f]|^2)\in L^1
$$
since $|\ss_n[f]|\in L^\infty$ thus also $\exp(2\mu|\ss_n[f]|^2)\in L^\infty$.
\end{proof}

Next, we consider series of the form
$$
\sum_{j=0}^{+\infty}c_je^{2i\pi jt}r_j(x).
$$
The idea is that such series are of the form $\sum\pm c_je^{2i\pi jt}$, that is, choosing $x\in(0,1)$
at random, we randomly change the sign of $c_j$. Our first result is the following:

\begin{theorem}\label{th:aelacune}
Let $(c_k)\in\ell^2$ and $\dst f_x(t)=\sum_{k=0}^{+\infty}c_kr_k(x)e^{2i\pi kt}$.
Then, for almost every $x\in(0,1)$, the series converges almost everywhere in $t\in(0,1)$
and $f_x\in L^p([0,1])$ for every $1\leq p<+\infty$.
\end{theorem}

\begin{proof}
Let $E$ be the set of $(x,t)\in[0,1]^2$ where the series defining $f$ converges.

According to Theorem \ref{thm:aecvg}, for every $t\in[0,1]$, the set $E_t^2=\{(x,t)\in E\}$ has measure $|E_x|=1$.
It follows that $|E|=1$ but then, for almost every $x\in[0,1]$, $E_x^1=\{(x,t)\in E\}$ has also
measure $|E_x^1|=1$.

Next, set $\gamma=\|c_k\|_2$ and fix $n\geq 1$. As in \eqref{ex:estexp},
\begin{equation}
\label{eq:radfour}
\frac{\mu^n}{n!}\int_0^1|f_x(t)|^{2n}\ud x=
\sum_{m=0}^{+\infty}\frac{\mu^m}{m!}\int_0^1|f_x(t)|^{2m}\ud x=
\int_0^1\exp(\mu|f_x(t)|^2)\ud x\leq \frac{1}{1-4e\mu\gamma^2}
\end{equation}
provided $\mu<\dfrac{1}{4e\gamma^2}$. It follows that
$$
\int_0^1\int_0^1|f_x(t)|^{2n}\ud t\ud x
=\int_0^1\int_0^1|f_x(t)|^{2n}\ud x\ud t
\leq \frac{n!}{(1-4e\mu\gamma^2)\mu^n}<+\infty.
$$
But then, for every $n$, there is a set $F_n\subset[0,1]$ with $|F_n|=0$ such that, if $x\in[0,1]\setminus F_n$,
$\dst\int_0^1|f_x(t)|^{2n}\ud t<+\infty$. Setting $F=\bigcup F_n$, $|F|=0$ and, 
for every $x\in[0,1]\setminus F$, for every $n$, $f_x\in L^{2n}$. Using the inclusion of $L^{2n}([0,1])\subset L^p([0,1])$ when $p\leq 2n$, we obtain that, for almost every $x$, $f_x\in L^p([0,1])$ for every $p\geq 1$,
as claimed.
\end{proof}

We can now prove the main result of this section:

\begin{theorem}
\label{th:lacunary}
Let $q>1$ and $(n_j)_{j\geq 0}$ be a $q$-lacunary sequence of integers, $n_0\geq 1$ and $n_{j+1}\geq qn_j$.
Let $1\leq p<+\infty$. There are two constants $A_{p,q},B_{p,q}$ such that,
if $(c_j)_{j\geq 0}\in \ell^2$, then $g(t)=\dst\sum_{j\geq 0}c_je^{2i\pi n_jt}$ is in $L^p([0,1])$
with
\begin{equation}
\label{eq:lacfs}
A_{p,q}
\left(\sum_{j=0}^{+\infty}|c_j|^2\right)^{\frac{1}{2}}\leq\left(\int_0^1\abs{\sum_{j\geq 0}c_je^{2i\pi n_jt}}^p\,\mathrm{d}t\right)^{\frac{1}{p}}
\leq B_{p,q}\left(\sum_{j=0}^{+\infty}|c_j|^2\right)^{\frac{1}{2}}.
\end{equation}
\end{theorem}

\begin{remark}
Note that a simple change of variable also shows that, for every integer $M$,
\begin{equation}
\label{cor:lacunary}
A_{p,q}
\left(\sum_{j=0}^{+\infty}|c_j|^2\right)^{\frac{1}{2}}\leq\left(\frac{1}{M}\int_{-M/2}^{M/2}\abs{\sum_{j\geq 0}c_je^{2i\pi \frac{n_j}{M}t}}^p\,\mathrm{d}t\right)^{\frac{1}{p}}
\leq B_{p,q}\left(\sum_{j=0}^{+\infty}|c_j|^2\right)^{\frac{1}{2}}.
\end{equation}

Also, we may assume that $q\to A_{p,q},B_{p,q}$ are continuous.
\end{remark}

\begin{proof}
The beginning of the proof is the same as for Theorem \ref{th:radlp}. Parseval's identify shows that \eqref{eq:lacfs} is satisfied
when $p=2$ with $A_{2,q}=B_{2,q}=1$. The lower bound is then automatically satisfied
for $p\geq 2$ with $A_{p,q}=1$ while the upper bound is satisfied for $p\leq 2$ with $B_{2,q}=1$.
Finally, if we establish the upper bound for $p>2$, using H\"older's inequality
in the same way as in the proof of Theorem \ref{th:radlp}, the lower bound
follows for $p<2$ with $A_{2,q}=B_{4,q}^{1-\frac{4-p}{p}}$.
Also, it is enough to prove the upper bound when $p=2m$, $m\geq 2$ and then, if $2(m-1)<p\leq 2m$,
$B_{p,q}\leq B_{2k,q}$. Another reduction is that, by homogeneity, it is enough to prove
the theorem when $\dst\sum_{j=0}^{+\infty}|c_j|^2=1$. 

A further restriction is that it is enough to prove the theorem for $q\geq 3$. Indeed, for $1<q<3$,
we introduce an integer $N_q$ such that $q^{N_q}\geq 3$ and write $n_k^{(\ell)}=n_{kN_q+\ell}$ for $\ell=0,\ldots,N_q-1$.
Then $n_{k+1}^{(\ell)}\geq q^{N_q}n_k^{(\ell)}$. If the theorem is established when $q\geq 3$ then, for each
$\ell$, the upper bound in \eqref{eq:lacfs} reads
$$
\left(\int_0^1\abs{\sum_{k\geq 0}c_{kN_q+\ell}e^{2i\pi n_k^{(\ell)}t}}^p\,\mathrm{d}t\right)^{\frac{1}{p}}
\leq B_{p,q^{N_q}}\left(\sum_{k=0}^{+\infty}|c_{kN_q+\ell}|^2\right)^{\frac{1}{2}}.
$$
But then, with the triangular inequality in $L^p$,
\begin{eqnarray*}
\left(\int_0^1\abs{\sum_{j\geq 0}c_{j}e^{2i\pi n_jt}}^p\,\mathrm{d}t\right)^{\frac{1}{p}}
&=&\left(\int_0^1\abs{\sum_{\ell=0}^{N_q-1}\sum_{k\geq 0}c_{kN_q+\ell}e^{2i\pi n_k^{(\ell)}t}}^p\,\mathrm{d}t\right)^{\frac{1}{p}}\\
&\leq&\sum_{\ell=0}^{N_q-1}\left(\int_0^1\abs{\sum_{k\geq 0}c_{kN_q+\ell}e^{2i\pi n_k^{(\ell)}t}}^p\,\mathrm{d}t\right)^{\frac{1}{p}}\\
&\leq&B_{p,q^{N_q}}\sum_{\ell=0}^{N_q-1}\left(\sum_{k=0}^{+\infty}|c_{kN_q+\ell}|^2\right)^{\frac{1}{2}}\\
&\leq&N_q^{1/2}B_{p,q^{N_q}}\left(\sum_{\ell=0}^{N_q-1}\sum_{k=0}^{+\infty}|c_{kN_q+\ell}|^2\right)^{\frac{1}{2}}\\
&=&N_q^{1/2}B_{p,q^{N_q}}\left(\sum_{j\geq 0}|c_j|^2\right)^{\frac{1}{2}},
\end{eqnarray*}
where we have used Cauchy-Schwarz in $\R^{N_q}$ in the next to last line.

A last reduction comes from the observation that, for every $k$
$$
\int_0^1\exp(\mu |g(t)|^2)\ud t=\sum_{n=0}^{+\infty}\frac{\mu^n}{n!}\int_0^1|g(t)|^{2n}\ud t
\geq \frac{\mu^m}{m!}\int_0^1|g(t)|^{2m}\ud t.
$$
It is therefore enough to prove that there is a $\mu(q)$ and a $C>0$ such that, if $\mu<\mu(q)$
\begin{equation}
\label{eq:toprove}
\int_0^1\exp(\mu |g(t)|^2)\ud t\leq C
\end{equation}
which would then imply that
$$
\int_0^1|g(t)|^{2k}\ud t\leq C\frac{m!}{\mu^m}
$$
as desired.

\smallskip

In order to prove \eqref{eq:toprove}, let us introduce
$$
f_x(t)=\sum_{j\geq 0}c_jr_{n_j}(x)e^{2i\pi n_jt}.
$$
Integrating \eqref{eq:radfour} with respect to $t$ and using Fubini, we deduce that
$$
\int_0^1\int_0^1\exp(\mu|f_x(t)|^2)\ud t\ud x\leq K:=\frac{1}{1-4e\mu\gamma^2}.
$$
But then, there is an $x_0$ (that we can assume not to be a dyadic rational $x_0\not=2^j/k$) such that
$$
\int_0^1\exp(\mu|f_{x_0}(t)|^2)\ud t\leq K.
$$
Next, we consider the Riesz product
$$
P_k(x)=\prod_{j=0}^k\bigl(1+r_{n_j}(x_0)\cos 2\pi n_jt\bigr)
=\prod_{k=0}^m\left(1+r_{n_k}(x_0)\frac{e^{2\pi n_kt}+e^{-2\pi n_kt}}{2}\right)
=\sum_{j\in\Z}\gamma_je^{2i\pi jt}
$$
where the Fourier coefficients have the following property:

-- $\gamma_0=1$

-- $\gamma_j=0$ if $j$ is an integer that is not of the form $\sum\pm n_\ell$, in particular when $|j|>\dst\sum_{k=0}^nn_k$

-- if $j=\sum\eps_\ell n_\ell$ with $\eps_\ell\in\{-1,0,1\}$. As $q>3$,
this $\eps_\ell$'s are unique. Then $\gamma_j=\prod_{\eps_j\not=0}\dfrac{r_{n_j}(x_0)}{2}$.
In particular, $\gamma_{n_j}=\dfrac{r_{n_j}(x_0)}{2}$ for $j=0,\ldots, k$
and $\gamma_{n_j}=0$ for $j>k$.

As a consequence, the partial sums of the Fourier series of $g$ are given by
$$
S_{n_k}[g](t):=\sum_{j=0}^{k}c_je^{2i\pi n_jt}
=\sum_{j=0}^{k}c_jr_j(x_0)^2e^{2i\pi n_jt}
=2\int_0^1f_{x_0}(s)P_k(t-s)\ud s.
$$
Note that $P_k\geq 0$ and $\dst\int_0^1 P_k(t)\ud t=\gamma_0=1$ so that $\nu_k=P_k(t)\ud t$ is a probability measure.
As $\ffi(s)=\exp(\mu s^2)$ is increasing and convex, we
apply Jensen's inequality (with the measure $\nu_k$) to obtain
$$
\ffi\left(\frac{1}{2}S_{n_k}[g](t)\right)
\leq\ffi\left(\int_0^1|f_{x_0}(s)|P_k(t-s)\ud s\right)
\leq \int_0^1\ffi\bigl(|f_{x_0}(s)|\bigr)|P_k(t-s)\ud s.
$$
Integrating over $[0,1]$ and using Fubini, we get
$$
\int_0^1\ffi\left(\frac{1}{2}S_{n_k}[g](t)\right)\ud t
\leq \int_0^1\ffi\bigl(|f_{x_0}(s)|\bigr)\int_0^1P_k(t-s)\ud t\ud s
=\int_0^1\ffi\bigl(|f_{x_0}(s)|\bigr)\ud s\leq K.
$$
Letting $k\to+\infty$, we obtain
$$
\int_0^1\exp\left(\frac{\mu}{2}|g(t)|^2\right)\ud t\leq K
$$
as claimed (up to $\mu/2$ replacing $\mu$.
\end{proof}

\subsection{The proof of Littlewood's conjecture by Mc Gehee, Pigno and Smith}

We will now give the proof of the Littlewood conjecture, following closely \cite{CQ}.

\begin{theorem}[Mc Gehee, Pigno, Smith]
\label{th:MPSrecall}
There exists a constant  $C_{MPS}\leqslant \dfrac{4}{\pi^2}$ such that, 
if $(n_k)_{k\in\N}$ is an increasing sequence of integers and $(a_k)$ is a complex sequence with finite support,
then
$$
C_{MPS}\sum_{k=0}^{+\infty}\frac{|a_k|}{k+1}\leqslant \int_0^1\abs{\sum_{k=0}^{+\infty} a_ke^{2i\pi n_k t}}\ud t.
$$
\end{theorem}

The proof given below will give $C_{MPS}=\dfrac{1}{96}$ which is not the best possible. We avoid minor technicalities which would lead to $C_{MPS}=\dfrac{1}{30}$ ({\it see e.g.} \cite{dvl,TB}). The proof in \cite{JKS} requires the introduction of various parameters and a cumbersome optimization of those parameters in the final step and leads to $C_{MPS}=\dfrac{1}{26}$.

Note also that the sum starts at $0$, we have so far been unable to obtain a lower bound
when the sum runs over $\Z$.

\subsubsection{Strategy of the proof}
\begin{notation}
    For a function $F\in L^2([0,1])$ and $s\in\Z$, we write
$$
\widehat{F}(s)=c_s(F)=\int_0^1F(t)e^{-2i\pi st}\ud t
$$
for the Fourier coefficients of $F$. Its Fourier series is then
$$
F(t)=\sum_{s\in\Z} c_s(F)e^{2i\pi st}
$$
and Parseval's relation reads
$$
\int_0^1|F(t)|^2\,\mathrm{d}t=\sum_{s\in\Z} |c_s(F)|^2
$$
and can also be written in the form
$$
\int_0^1F(t)G(t)\,\mathrm{d}t=\sum_{s\in\Z} c_s(F)c_{-s}(G).
$$
\end{notation}

We fix a trigonometric polynomial
\begin{equation}
\label{e11}
\phi(t)=\sum_{k=0}^Na_ke^{2 i\pi n_kt}
\quad\mbox{and}\quad
S=\sum_{k=0}^N\frac{|a_k|}{k+1}.
\end{equation}

We then write $|a_k|=a_ku_k$ with $u_k$ complex numbers of modulus $1$ and we introduce
\begin{equation}\label{e12}
T_0(t)=\sum_{k=0}^N\frac{u_k}{k+1}e^{-2i\pi n_k t}.
\end{equation}
Then by Parseval identity, written in the second form
we have
\begin{equation}\label{e13}
S :=\sum_{k=0}^N\frac{|a_k|}{k+1}=\sum_{k=0}^N\widehat{\phi}(n_k)\widehat{T_0}(-n_k)
=\int_{0}^{1}\phi(t)T_0(t)\ud t.   
\end{equation}
so that
$$
S\leqslant \|\phi\|_1\|T_0\|_{L^\infty([0,1])}.
$$
The issue is that, typically we have no control over the $L^\infty$ norm of $T_0$ other than the trivial and explosive control by $\sum \frac{1}{k+1}$, so we will correct $T_0$ into another test function $T_1$ as follows;
\begin{enumerate}
\item The $L^\infty$ norm of the corrected function $T_1$ is controlled by a constant $C$, $\|T_1\|_\infty\leq C$.

\item $\widehat{T_1}(-n_k)$ only differs a little from $\widehat{T_0}(-n_k)$ for $0\leqslant k \leqslant N$,
say $|\widehat{T_1}(-n_k)-\widehat{T_0}(-n_k)|\leq \dfrac{\widehat{T_0}(-n_k)}{2}=\dfrac{1}{2(k+1)}$, while we impose no condition on the behavior of $\widehat{T_1}(n)$ for $n\neq -n_k$.
\end{enumerate}

We would then conclude as follows:
$$
\left| S-\int_{0}^{1}\phi(t)T_1(t) \ud t \right|
=\abs{\int_0^1\phi(t)(T_0-T_1)(t)\ud t}
=\abs{\sum_{k=0}^N \widehat{\phi}(\lambda_k)(\widehat{T_0}(-n_k)-\widehat{T_1}(-n_k)) }
$$
with Parseval. With the triangular inequality and our expected estimate
$|\widehat{T_1}(-n_k)-\widehat{T_0}(-n_k)|\leq\dfrac{1}{2(k+1)}$, we then
conclude that
$$
\left| S-\int_{0}^{1}\phi(t)T_1(t) \ud t \right|\leq \sum_{k=1}^N |\widehat{\phi}(\lambda_k)|\,|\widehat{T_0}(-n_k)-\widehat{T_1}(-n_k)|\leq \sum_{k=1}^N\frac{|a_k|}{2(k+1)}=\frac{S}{2}.
$$
But then
$$
\frac{1}{2}S\leq \abs{\int_{0}^{1}\phi(t)T_1(t) \ud t}\leq \|T_1\|_\infty \int_{0}^{1}|\phi(t)|\ud t
\leq C \int_{0}^{1}|\phi(t)|\ud t
$$
which is the expected result with $A=\dfrac{1}{2C}$.

The proof of the Theorem \ref{th:MPS} thus amounts to proving the following lemma:

\begin{lemma}\label{l12}
    There exists a universal constant $C$ and a $T_1\in L^\infty$ such that
    \begin{enumerate}
        \item $\|T_1\|_\infty \leqslant C$ 
        \item $|\widehat{T_1}(-n_k)-\widehat{T_0}(-n_k)|\leqslant \dfrac{1}{2} |\widehat{T_0}(-n_k)|=\dfrac{1}{2(k+1)}$ for $0\leqslant k \leqslant N$
    \end{enumerate}
    where $T_0$ is the function given by \eqref{e12}.
\end{lemma}

The remaining of this section is devoted to the proof of this lemma.

\subsubsection{Proof of Lemma \ref{l12}}
First note that, up to eventually adding extra zeros to the sequence $(a_k)$ and adding $\lambda_{N+j}=\lambda_N+j$
to the sequence $(\lambda_k)$, we may assume that $N=2^{m+1}-2$.

We start by decomposing $T_0$ into a sum of dyadic blocs on which the amplitude 
$\displaystyle|\widehat{T_0}(-\lambda_k)|=\frac{1}{k+1}$ is more or less constant. More precisely, 
for $j=0,\ldots m$ we set
$I_j=[ 2^j, 2^{j+1}[$ and 
$$
f_j=\sum_{k+1\in I_j}\frac{u_k}{k+1}e^{-2i\pi n_k t}.
$$
The function $T_0$ now appears as the partial sum of order $m$ of the series $\sum f_j$, in other words
$$
T_0=\sum_{j=0}^m f_j.
$$

Let us start with a simple lemma:

\begin{lemma}\label{l13}
With the above notations, we have
\begin{enumerate}
    \item $\|f_j\|_{L^\infty([0,1])} \leqslant 1$.
    \item $\|f_j\|_{L^2([0,1])}\leqslant 2^{-j/2}$.
\end{enumerate}
\end{lemma}

\begin{proof}
 By construction, $|I_j|=2^j$ hence
    $$
    \|f_j\|_\infty \leqslant \sum_{k+1 \in I_j}\frac{1}{k+1} \leqslant \frac{|I_j|}{2^j}=1.
    $$

On the other hand, Parseval implies that
$$
    \|f_j\|_2^2=\sum_{k+1=2^j}^{2^{j+1}-1}\frac{1}{(k+1)^2}\leqslant \frac{2^j}{4^j}=2^{-j}
$$
as claimed.
\end{proof}

Now, write the Fourier series of each $|f_j|\in L^2([0,1])$ 
$$
|f_{j}|=\sum_{s\in \mathbb{Z}}c_s(|f_j|)e^{2i\pi st}.
$$
To each $|f_j|$, we associate $h_{j}\in L^{2}([-\pi,\pi])$ defined via its Fourier series as
$$
h_j(t)=c_0(|f_j|)+2\sum_{s=1}^{\infty}c_s(|f_j|)e^{2i\pi st}.
$$ 

\begin{lemma}\label{l15}
     For $0\leq j \leq n$, the following  properties hold
\begin{enumerate}
    \item $\mathrm{Re}(h_{j})=|f_j|\leqslant 1$;
    \item $\|h_{j}\|_{L^2([0,1])}\leq \sqrt{2}\|f_{j}\|_{L^2([0,1])}$.
\end{enumerate}
\end{lemma}

\begin{proof}
First, as $|f_j|$ is real valued, $c_0(|f_j|)$ is also real while $\overline{c_s(|f_j|)}=c_{-s}(|f_j|)$ for every $s\geq 1$. Hence
    $$
    \overline{h_j}(t)=c_0(|f_j|)+2\sum_{s=1}^\infty c_{-s}(|f_j|)e^{-ist}
    $$
    and thus
    $$
    \mathrm{Re}(h_j)=\frac{h_j+\overline{h_j}}{2}=c_0(|f_j|)+\sum_{s\neq 0}c_s(|f_j|)e^{ist} =|f_j|\leqslant 1
    $$
    by lemma \ref{l13}.

By Parseval's identity and again $\overline{c_s(|f_j|)}=c_{-s}(|f_j|)$,
\begin{eqnarray*}
\|h_j\|_2^2&=&\bigl|c_0(|f_j|)\bigr|^2+4\sum_{s=1}^\infty\bigl|c_s(|f_j|)\bigr|^2
=\bigl|c_0(|f_j|)\bigr|^2+2\sum_{s=1}^\infty\bigl|c_s(|f_j|)\bigr|^2+2\sum_{s=1}^\infty\bigl|c_{-s}(|f_j|)\bigr|^2\\
&\leqslant& 2\sum_{s\in \Z}|c_s(|f_j|)|^2=2\|f_j\|_2^2
    \end{eqnarray*}
    as claimed.
\end{proof}

We now define a sequence $(F_{j})_{j=0,\ldots,n}$ inductively through
$$
F_0= f_0
\quad\mbox{and}\quad
F_{j+1}=F_{j}e^{-\eta h_{j+1}}+f_{j+1}
$$
where $0<\eta\leq 1$ is a real number to be adjusted later on. Further set
$$
E_{\eta}:=\sup_{0 <x \leq 1} \dfrac{x}{1-e^{-\eta x}}=\frac{1}{\eta}\sup_{0 <x \leq \eta} \dfrac{x}{1-e^{- x}}
=\frac{1}{1-e^{-\eta}}\leq\dfrac{2}{\eta}.
$$

\begin{lemma}\label{l16}
For $0\leq j\leq n$, $\| F_{j}\|_{\infty}\leqslant\dfrac{2}{\eta}$.
\end{lemma}

\begin{proof}
By definition of $E_\eta$, if $C\leq E_{\eta}$ and
$0\leq x\leq 1$, then  $Ce^{-\eta x}+x\leq E_{\eta} e^{-\eta x}+x\leq E_{\eta}$.

We can now prove by induction over $j$ that $|F_j|\leq E_{\eta}$.
First, when $j=0$, from Lemma \ref{l13} we get
$$
\|F_0\|_{\infty}=\|f_0\|_{\infty} \leq 1\leq E_{\eta}.
$$
Assume now that $\|F_j\|_\infty\leq E_{\eta}$, then
\begin{eqnarray*}
|F_{j+1}(t)|&=&|F_{j}(t)e^{-\eta h_{j+1}(t)}+ f_{j+1}(t)|\leq
|F_{j}(t)|e^{-\eta\Re\bigl(h_{j+1}(t)\bigr)}+|f_{j+1}(t)|\\
&=&|F_{j}(t)|e^{-\eta|f_{j+1}(t)|}+|f_{j+1}(t)|.
\end{eqnarray*}
As $|f_{j+1}(t)|\leq 1$ and $|F_{j}(t)|\leq E_{\eta}$, we get
$|F_{j+1}(t)|\leq E_{\eta}$. It remains to prove that $E_{\eta}\leqslant \dfrac{2}{\eta}$. To do so, it suffices to see that

$$
e^{-y}\leqslant1-\left(\dfrac{e-1}{e}\right)y\qquad \text{for~} 0\leqslant y \leqslant 1,
$$ 
yielding the result immediately.
\end{proof}

\begin{lemma}\label{l17}
For $0\leq \ell\leq n$ and $j=0,\ldots,k$, let $g_{j,k}=e^{-\eta H_{j,k}}$
with
$$
H_{j,k}=\begin{cases}
h_{j+1}+\ldots +h_{k} & \mbox{when } j<k\\
0  & \mbox{when }j=k
\end{cases}.
$$
Then
$$
F_{k}=\sum_{j=0}^{k}f_{j}g_{j,k}.
$$

\end{lemma}

\begin{proof} By induction on $k$, when $k=0$, $H_{0,0}=0$ thus $g_{0,0}=1$ and, indeed, we have
$$
F_0=f_0=f_{0}g_{0,0}.
$$ 
Assume now that the formula has been established at rank $k-1$
and let us show that $F_{k}=\displaystyle\sum_{j=0}^{k}f_{j}g_{j,k}$.
By construction, we have
$$
F_{k}=F_{k-1}e^{-\eta h_{k}}+ f_{k}
=\left(\sum_{j=0}^{k-1}f_{j}g_{j,k-1}\right)e^{-\eta h_{k}}+f_{k}.
$$
with the induction hypothesis. It remains to notice that $g_{k,k}=e^{-\eta H_{k,k}}=1$
and that, for $j=0,\ldots,k-1$, $H_{j,k}=H_{j,k-1}+h_k$ thus
$g_{j,k}=g_{j,k-1}e^{-\eta h_{k}}$ so that, indeed, we have
$F_k=\displaystyle\sum_{j=0}^{k}f_{j}g_{j,k}$ as claimed.
\end{proof}

Recall that
$$
T_0=\sum_{j=0}^m f_j
$$
and we set
$$
T_1^\eta=F_m=\sum_{j=0}^m f_jg_{j,m}
$$
where the dependence on $\eta$ comes from the definition of the $g_{j,n}$'s, in particular
$$
\|T_1^\eta\|_\infty \leqslant E_\eta.
$$

The first part of Lemma \ref{l12} is thus established and it remains to prove the second part. To do so, we start by some intermediary results;

\begin{lemma}\label{l18}
If $H \in H^{\infty}$(Hardy space) and $Re(H)\geqslant 0,$ then $e^{-H}\in H^{\infty}$ and 
$$
\| e^{-H}-1\|_{2}\leqslant \|H\|_{2}.
$$
\end{lemma}

\begin{proof}
Since $H^{\infty}$ is a Banach algebra, the partial sums $\displaystyle \sum_{k=0}^{n}(-1)^{k} \frac{H^{k}}{k!}$ of $e^{-H}$ are elements of $H^{\infty}$. Moreover, since $H$ is bounded, these sommes converge uniformly toward $e^{-H}$, with $e^{-H} \in H^{\infty}$. Finally, if $z \in \C$ and $\Re (z) \geqslant 0$,
$$
\left|e^{-z}-1\right|=\left|\int_{0}^{1} z e^{-t z} d t\right| \leqslant  \int_{0}^{1}|z| e^{-t \Re (z)} d t \leqslant |z| .
$$
In our case $z=H(t)$, and we have 
$$
\left|e^{-H(t)}-1\right| \leqslant |H(t)|
$$
and by integration we have the desired inequality.
\end{proof}

Next let us introduce the following notations:
\begin{enumerate}
    \item Let $f \in L^1([0,1])$, the spectrum of $f$, denoted by $\spec(f)$ is the set of non zero 
    Fourier coefficients, more precisely
    $$
    \spec(f)=\{n \in \Z\,:\ \widehat{f}(n)\neq 0 \}.
    $$
It is easy to show that, if $f,g\in L^2([0,1])$, then 
$$
\spec(fg)\subset\spec(f)+\spec(g)
=\{\lambda+\mu\,:\ \lambda\in\spec(f),\mu\in\spec(g)\}.
$$

    \item Let A be a subset of $[1, N [$, we denote by $\Lambda_A$ the set of $n_j,~j\in A$ \it{i.e}
    $$
    \Lambda_A=\{n_j,~j\in A   \}.
    $$

\end{enumerate}
\begin{lemma}\label{l19}
Let $k+1\in I_\ell=[ 2^\ell,2^{\ell+1} [$ then $\widehat{f_jg_{j,m}}(-n_k)=0$ if $j<\ell$.
\end{lemma}

\begin{proof}
We must show that $-n_k \notin \spec(f_jg_{j,m})$. By contradiction, we suppose that $-n_k \in sp(f_jg_j)$.

But since $\spec(f_j)\subset -\Lambda_{I_j}$ and, from Lemma \ref{l18}, $\spec(g_{j,m})\subset\N$, thus
$$
\spec(f_jg_{j,m})\subset -\Lambda_{I_j}+\N.
$$
However, since $j<\ell$, $I_j$ is to the left of ${I_\ell}$, then $-\Lambda_{I_\ell}$ 
is completely to the left of $-\Lambda_{I_j}+\N$
(since $k+1\in I_\ell$, then $-n_k \in-\Lambda_{I_\ell}$ by definition of $\Lambda_{I_\ell}$).

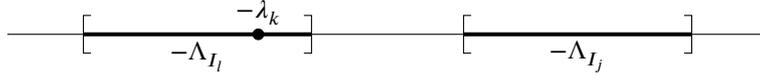
\begin{center}
\begin{tikzpicture}
\draw (-5,0) -- (5,0);

\draw[line width=1.5pt] (-4,0) -- (-1,0);
\draw  (-3.9,-.25) -- (-4,-.25) -- (-4,.25) -- (-3.9,.25);
\draw  (-1.1,-.25) -- (-1,-.25) -- (-1,.25) -- (-1.1,.25);
\draw (-2.5, -0.3) node {$-\Lambda_{I_l}$};
\filldraw[black] (-1.7,0) circle (2pt) node[anchor=south]{$-\lambda_k$};

\draw[line width=1.5pt] (1,0) -- (4,0);
\draw  (1.1,-.25) -- (1,-.25) -- (1,.25) -- (1.1,.25);
\draw  (3.9,-.25) -- (4,-.25) -- (4,.25) -- (3.9,.25);
\draw (2.5, -0.3) node {$-\Lambda_{I_j}$};

\end{tikzpicture}
\end{center}
\end{proof}

\begin{proof}[Proof of $(2)$ in Lemma \ref{l12}]
Let $k\in[ 0,N [$ and $\ell\leqslant m$ such that $k+1\in I_\ell$. Hence 
$$
\widehat{T_1}(-n_k)=\sum_{j=0}^m\widehat{f_jg_{j,m}}(-n_k)
=\sum_{j=\ell}^m\widehat{f_jg_{j,m}}(-n_k)
$$
with Lemma \ref{l19}.

On the other hand, $\widehat{T_0}(-n_k)=\dfrac{u_k}{k+1}$
while $\widehat{f_j}(-n_k)=\dfrac{u_k}{k+1}$ if $k+1\in I_j$ {\it i.e.} if $j=\ell$
and $\widehat{f_j}(-n_k)=0$ otherwise. So we can write
$$
\widehat{T_0}(-n_k)=\sum_{j=\ell}^m\widehat{f_j}(-n_k),
$$
hence
$$
\left |  \widehat{T_1}(-n_k)-\widehat{T_0}(-n_k)   \right|
=\sum_{j=\ell}^m \widehat{[f_j(g_{j,m}-1)]}(-n_k)
=\sum_{j=\ell}^m c_{-n_k}[f_j(g_{j,m}-1)].
$$
But
\begin{eqnarray*}
|c_{-n_k}[f_j(g_{j,m}-1)]|&\leq& \| f_j(g_{j,m}-1)\|_{L^1([0,1)]}\\
&\leq&  \| f_j\|_{L^2([0,1)]} \| g_{j,m}-1\|_{L^2([0,1)]}
=\| f_j\|_{L^2([0,1)]} \|e^{-\eta H_{j,m}}-1\|_{L^2([0,1)]}
\end{eqnarray*}
by definition of $g_{j,m}$. Using Lemma \ref{l18} and then Lemma \ref{l17}, we obtain
\begin{eqnarray*}
|c_{-n_k}[f_j(g_{j,m}-1)]|&\leqslant&\eta \|f_j\|_2\|H_{j,m}\|_2 
\leqslant \eta \|f_j\|_2\sum_{r=j+1}^m\|h_r\|_2\\
&\leqslant& \eta \|f_j\|_2\sqrt{2}\sum_{r=j+1}^m \|f_r\|_2 
\end{eqnarray*}    
with Lemma \ref{l15}. Then, from Lemma \ref{l13}, we obtain
$$
|c_{-n_k}[f_j(g_{j,m}-1)]|\leq
\eta 2^{-j/2}\sqrt{2}\sum_{r=j+1}^{+\infty} 2^{-r/2}
=\eta 2^{-j/2}\sqrt{2}\frac{2^{-(j+1)/2}}{1-2^{-1/2}}.
$$
We conclude that
$$
\left |  \widehat{T_1}(-\lambda_k)-\widehat{T_0}(-\lambda_k)   \right|
\leq 3\eta \sum_{j=\ell}^m 2^{-j}\leq 6\eta2^{-\ell}.
$$

On the other hand, as $k+1\in I_\ell=[2^\ell,2^{\ell+1}[$, 
$$
|\widehat{T_0}(-\lambda_k)|=\frac{1}{k+1}> 2^{-\ell-1}
$$
so that
$$
\left |  \widehat{T_1}(-\lambda_k)-\widehat{T_0}(-\lambda_k)   \right|\leq 12\eta|\widehat{T_0}(-\lambda_k)|.
$$
Choosing $\eta=\dfrac{1}{24}$ gives the result.
\end{proof}

Note that, when $\eta=\dfrac{1}{24}$, $E_\eta\leq 48$ so that we can take $C=48$
in Lemma \ref{l12}, leading to $C_{MPS}=\dfrac{1}{96}$ in Theorem \ref{th:MPSrecall}.

\subsection{Strongly multidimensional sets}

Let $\delta>0$ and $(m,n)\in \mathbb{N}^2$. A subset $A$ of $\mathbb{Z}$ is $(\delta;m,n)$-strongly 2-dimensional if there exists numbers $d$ and $D$ with $D> (2+\delta)d$ such that
$$
    A=\bigcup_{k\in I}(A_k+kD)
$$
for some set $I$ containing at least $m$ integers and subsets $A_k \subseteq \{-d,\ldots,d\}$ verifying $|A_k|\geqslant n$.

\begin{theorem}[Hanson \cite{Ha}]\label{t421}
Let $\delta>0$ and $m,n$ be two positive integers satisfying
\begin{equation}\label{e421}
n\geqslant \pi^3 2^{21}C_{MPS}^3\ln(n)^3
\quad \text{and} \quad 
m \geqslant \pi^3 2^{21}C_{MPS}^3\ln(n)^3 \ln(m)^3,
\end{equation}
where $C_{MPS}$ is the constant in theorem \ref{th:MPS}. Suppose $A$ is $(\delta;m,n)$ strongly 2-dimensional subset of $\mathbb{Z}.$ Then
$$
\int_0^1\left| \sum_{a\in A}e^{2i\pi a t} \right|\ud t\geqslant \dfrac{C_{MPS}^2}{(2^9\pi)^2\left(2+\ln(1+\frac{2}{\delta})\right)}\ln(m)\ln(n)
$$
\end{theorem}
Combining this result with Theorem 3.3 in \cite{shao2013character}, we see that this theorem is also best possible up to the constant.

Given a set $I\subseteq \mathbb{Z}$, a positive integer $q$, and an arbitrary integer $s$, we define
$$
    I(q;s)=\{k\in I: k=s~ (\text{mod q})  \}.
$$
The proof of Theorem \ref{t421} is a direct consequence of two lemmas. The first one is the following:

\begin{lemma}\label{l422}
Let $I$ be a set of integers with $|I|\geqslant 8$. Then there are positive integers q and s such that
$$
    \frac{~|I|^{\frac{1}{3}}}{8}\leqslant |I(q;s)|\leqslant q^{1/2}.
$$
\end{lemma}

\begin{proof}
For each $j\geqslant 1$, we choose any $s_j$ such that $|I(4^j;s_j)|$ is maximal. But, on one hand,
$$
    I=\bigcup_{s=0}^{4^j-1}I(4^j,s)
$$
and on the other hand, for $j$ fixed, the sets $I(4^j;s)$ are disjoints,
so at least one of them has cardinality larger than $4^{-j}|I|$. In particular,
\begin{equation}\label{e425}
|I(4^j,s_j)|\geqslant 4^{-j}|I|.
\end{equation}

For $j=1$, we thus have $|I(4;s_1)|\geqslant \dfrac{|I|}{4}\geqslant 2$. On the other hand,
if $j=s \mod kp$ then $j=s\mod p$ so that, for any $s$
$$
    I(4^m;s)\subset I(4^\ell;s)    \quad \text{for} \quad \ell<m,
$$
and, for sufficiently large $j$ we have $|I(4^j;s_j)|=1\leqslant 2^j$. Therefore, there exists a minimal $j_0$ such that $|I(4^{j_0};s_{j_0}) |\leqslant 2^{j_0}$. Let $q=4^{j_0}$, and $s=s_{j_0}$ then using \eqref{e425} and the definition of $j_0$
 $$
\frac{|I|}{q}\leqslant |I(q;s)|=|I(4^{j_0};s_{j_0})|\leqslant 2^{j_0}=q^{\frac{1}{2}}.
 $$
 In particular 
\begin{equation}\label{e426}
|I|^{\frac{1}{3}}\leqslant q^{\frac{1}{2}}.
\end{equation}
By minimality of $j_0-1$
$$
    \frac{q^{\frac{1}{2}}}{2}=2^{j_0-1}\leqslant \left|I\left(4^{j_0-1};s_{j_0-1} \right)\right|
\leq \sum_{r=0}^3 \left|I\left(4^{j_0};s_{j_0-1}+ r4^{j_0-1}\right)\right|  
\leq 4|I(4^{j_0};s_{j_0})|=4|I(q;s)|
$$
by definition of $s_{j_0}$. We thus get
$$
    |I(q;s)|\geqslant \frac{q^{\frac{1}{2}}}{8}\geqslant \frac{~|I|^{\frac{1}{3}}}{8},
$$
with \eqref{e426}.
\end{proof}

\begin{lemma}\label{l423}
Let $\delta>0$ and let $d$ and $D$ be positive integers with $(2+\delta)d<D$. Suppose $I$ is a finite set of integers, and let 
$$
    F(t)=\sum_{k\in I}f_k(t)e^{2i\pi Dkt}
$$
where
$$
    f_k(t)=\sum_{|n|\leqslant d}a_{n,k}e^{2i\pi nt}
$$
Let $q$ and $s$ with $q>4\pi$ and suppose $I(q;s)=\{k_1,\ldots,k_J \}$ then we have
$$
    \|F\|_{L^1([0,1]}\geqslant \frac{1}{32\pi(2+\ln(1+\frac{2}{\delta} ))}\sum_{j=1}^J\|f_{k_j}\|_{L^1([0,1])}\left( \frac{C_{MPS}}{2j}-\dfrac{2\pi d}{qD}  \right)
$$
\end{lemma}

We split the remaining of this section into two parts. In the first one, we show that Lemmas
\ref{l422} and \ref{l423} imply Theorem \ref{e421}. 
In the second one, we prove the Lemma \ref{l423}.

\subsubsection{Proof of Theorem \ref{t421}}
Let $\delta>0$ $m,n$ be two integers satisfying the conditions of the theorem
and let $A$ be strongly $(\delta,m,n)$-regular. Thus, there are two integers
$d,D$ with $D>(2+\delta)d$, such that we can write
$$
    A=\bigcup_{k\in I}(A_k+kD),
$$
with $|I|\geqslant m$  and $A_k\subset\{-d,\ldots,d\}$ with $|A_k|\geqslant n$.
We can then write
$$
F(t):= \sum_{a\in A}e^{2i\pi a t}=\sum_{k\in I}f_k(t)e^{2i\pi Dkt}
$$
with
$$
f_k(t)=\sum_{a\in A_k}e^{2i\pi a t}=\sum_{n=-d}^da_{n,k}e^{2i\pi n t}
$$
with $a_{n,k}=1$ if $n\in A_k$ and $a_{n,k}=0$ otherwise.

\smallskip

Assume first that there exists $k_1\in I$ such that
\begin{equation}\label{e422}
    \|f_{k_1}\|_{L^1([0,1])}\geqslant \dfrac{C_{MPS}}{2^9\pi}\ln(m)\ln(n).
\end{equation}
We then choose $q\geqslant\dfrac{16\pi}{7C_{MPS}}$ in such a way that there is an
$s$ such that $ I(q,s)=\{k_1\}$.
Hence by Lemma \ref{l423},
\begin{eqnarray*}
\|F\|_{L^1([0,1])}&\geq&
\frac{\|f_{k_1}\|_{L^1([0,1])}}{32\pi\left(2+\ln\left(1+\frac{2}{\delta} \right)\right)}\left( \frac{C_{MPS}}{2}-\dfrac{2\pi d}{qD}  \right)\\
&\geq&\frac{C_{MPS}\|f_{k_1}\|_{L^1([0,1])}}{2^6\pi\left(2+\ln\left(1+\frac{2}{\delta} \right)\right)}
\left( 1-\dfrac{7d}{8D}  \right).
\end{eqnarray*}
As $D>2d$, using \eqref{e422}, we conclude that, in this case
$$
\|F\|_{L^1([0,1])}\geq\left(\frac{C_{MPS}}{2^9\pi}\right)^2
\frac{\ln n\ln m}{2+\ln\left(1+\frac{2}{\delta} \right)}
$$
which establishes the theorem.

\smallskip

We will thus assume that, for each $k\in I$,
\begin{equation}\label{e423}
    \|f_k\|_{L^1([0,1])}\leqslant \dfrac{C_{MPS}}{2^9\pi}\ln(m)\ln(n).
\end{equation}

Note that, from Theorem \ref{th:MPS},
$$
    \|f_k\|_{L^1([0,1])}\geqslant C_{MPS}\ln(n)
$$
so that $2^9\pi \leqslant \ln(m)$, in particular, $m\geq 8$.
We then take $q$ and $s$ given by Lemma \ref{l422} applied to the set $I$ so that
$J=|I(q;s)|$ satisfies
\begin{equation}\label{e424}
    \dfrac{m^\frac{1}{3}}{8}\leqslant J \leqslant q^{\frac{1}{2}}.
\end{equation}
We write
$$
    I(q;s)=\{k_1<\ldots <k_J\}
$$
From Lemma \ref{l423}, we get that
\begin{eqnarray*}
\|F\|_{L^1([0,1])}&\geqslant& \frac{1}{2^5\pi\left(2+\ln\left(1+\frac{2}{\delta} \right) \right)}
\sum_{j=1}^J\|f_{k_j}\|_{L^1}\left( \frac{C_{MPS}}{2j}-\frac{2\pi d}{qD}   \right)\\
&=&\frac{1}{2^5\pi\left(2+\ln\left(1+\frac{2}{\delta} \right) \right))}(T_1-T_2)
\end{eqnarray*}
with
$$
T_1=\frac{C_{MPS}}{2}\sum_{j=1}^J\dfrac{\|f_{k_j}\|_{L^1}}{j} \quad \text{and} \quad T_2= \frac{2\pi d}{qD}\sum_{j=1}^J\|f_{k_j}\|_{L^1}.
$$

Next, as $ \|f_{k_j}\|_{L^1} \geqslant C_{MPS}\ln(n)$,
$$
T_1\geq\frac{C_{MPS}^2\ln(n)}{2}\sum_{j=1}^J\dfrac{1}{j}
\geq\frac{C_{MPS}^2\ln(n)}{2}\ln J
\geq \frac{C_{MPS}^2\ln(n)\ln(m)}{8}
$$
with \eqref{e424}.

On the other hand, from \eqref{e423}, we get
$$
    T_2\leqslant\dfrac{2\pi J d }{qD} \frac{C_{MPS}}{2^9\pi}\ln(m)\ln(n)
    \leqslant \dfrac{2\pi   }{q^{\frac{1}{2}}} \frac{C_{MPS}}{2^9\pi}\ln(m)\ln(n),
$$
since $d\leqslant \dfrac{D}{2}$ and $J\leqslant q^{\frac{1}{2}}$ with \eqref{e423}.
Further, \eqref{e423} also implies that
$$
q^{\frac{1}{2}} \geqslant \frac{m^{\frac{1}{3}}}{8}\geq  2^4\pi C_{MPS}\ln(m)\ln(n)
$$
with \eqref{e421}, leading to $T_2\leqslant \dfrac{1}{2^{13}\pi}$.

We have  established that
     \begin{eqnarray*}    
    \|F\|_{L^{1}([0,1])}&\geqslant& \frac{1}{2^5\pi\left(2+\ln\left(1+\frac{2}{\delta} \right) \right)} \left( \frac{C_{MPS}^2\ln(m)\ln(n)}{2^3} -\frac{1}{2^{13}\pi}  \right)\\
    &=& \frac{C_{MPS}^2}{(2^9)^2\pi\left(2+\ln\left(1+\frac{2}{\delta} \right) \right)}\left( 2^{10}\ln(m)\ln(n)-\frac{1}{C_{MPS}^2}\right)\\
    &\geqslant& \dfrac{C_{MPS}^2}{(2^9\pi)^2\left(2+\ln(1+\frac{2}{\delta})\right)} \ln(m)\ln(n)
    \end{eqnarray*}
since $\ln(m)\ln(n)C_{MPS}^2\geqslant \dfrac{1}{2^{10}\pi-1}.$

\subsubsection{Proof of Lemma \ref{l423}}
The rest of this chapter consist in proving Lemma \ref{l423}. the proof is divided into several lemmas.

The first one is a simple lemma about numerical integration of trigonometric polynomials:

\begin{lemma}\label{l428}
Let $N$ be a positive integer. Then for any trigonometric polynomial $f$ of degree $d$
$$
    \left|  \|f\|_{L^1([0,1])}-\frac{1}{N}\sum_{j=0}^{N-1}\abs{f\left(\frac{j}{N}\right)}\right|\leqslant \frac{2\pi d}{N} \|f\|_{L^1([0,1])}.
$$
\end{lemma}

Note that, as $f$ is $1$-periodic, writing $N=R+S$,
$\dst\frac{1}{N}\sum_{j=0}^{N-1}\abs{f\left(\frac{j}{N}\right)}=\frac{1}{R+S}\sum_{j=-R}^{S-1}\abs{f\left(\frac{j}{R+S}\right)}$.

\begin{proof}
We write, using the triangular and reverse triangular inequalities
\begin{eqnarray*}
\abs{\int_0^1|f(t)|\ud t-\frac{1}{N}\sum_{j=0}^{N-1}\abs{f\left(\frac{j}{N}\right)}}
&=&\abs{\sum_{j=0}^{N-1}\int_{\frac{j}{N}}^{\frac{j+1}{N}}|f(t)|-\abs{f\left(\frac{j}{N}\right)}\ud t}\\
&\leq&\sum_{j=0}^{N-1}\int_{\frac{j}{N}}^{\frac{j+1}{N}}\abs{f(t)-f\left(\frac{j}{N}\right)}\ud t
=\sum_{j=0}^{N-1}\int_{\frac{j}{N}}^{\frac{j+1}{R}}\abs{\int_{\frac{j}{N}}^tf'(s)\ud s}\ud t\\
&\leq&\sum_{j=0}^{N-1}\int_{\frac{j}{N}}^{\frac{j+1}{N}}\int_{\frac{j}{N}}^t|f'(s)|\ud s\ud t
=\sum_{j=0}^{N-1}\int_{\frac{j}{N}}^{\frac{j+1}{N}}|f'(s)|\int_s^{\frac{j+1}{N}}\ud t\ud s\\
&\leq&\frac{1}{N}\sum_{j=0}^{N-1}\int_{\frac{j}{N}}^{\frac{j+1}{N}}|f'(s)|\ud s=\int_0^1|f'(s)|\ud s
\end{eqnarray*}
where we have used Fubini and the bound $\dst\int_s^{\frac{j+1}{N}}\ud t\leq \dfrac{1}{N}$
when $\dfrac{j}{N}\leq s\leq\dfrac{j+1}{N}$.
We conclude with Bernstein's inequality, $\|f'\|_{L^1([0,1])}\leqslant 2\pi d \|f\|_{L^1([0,1])}$.
\end{proof}

For a finitely supported sequence $\bigl(A(k)\bigr)_{k\in\Z}$ we define its discrete Fourier transform (or 
$\mathbf{Z}$-Fourier transform)
as
$$
\ff_d[A](t)=\sum_{k\in\Z}A(k)e^{-2i\pi kt}.
$$
If $A,B$ are two  finitely supported sequences, their convolution is the sequence $A*B$ defined by
$$
A*B(k))B*A(k)=\sum_{n\in\Z}A(k-n)B(n).
$$
The Convolution Theorem is also valid here: $\ff_d[A*B](t)=\ff_d[A](t)\ff_d[B](t)$.
Two classical examples are 

-- the Dirichlet kernel: set $\mathbbm{d}_L=\mathbbm{1}_{-L,\ldots,L}$ so that
$$
D_L(t)=\ff_d[\mathbbm{d}_L](t)=\sum_{|k|\leqslant L}e^{2i\pi k t}=\frac{\sin (\pi(2L+1)t )}{\sin (\pi t)};
$$

-- the Fejer kernel: set $\mathbbm{f}_L(k)=\left( 1-\frac{|k|}{L+1}\right)\mathbbm{1}_{-L,\ldots,L}(k)$ so that
$$
F_L(t)=\ff_d\ent{\mathbbm{f}_L}(t)=
\sum_{|k|\leqslant L}\left( 1-\frac{|k|}{L+1}\right)e^{2i\pi k t}=\frac{1}{L+1}\frac{( \sin(\pi(L+1)t))^2}{(\sin(\pi t))^2}.
$$

\begin{lemma}\label{l429}
Let $M,N,R,S$ be integers with $2\leqslant M<N$. Then there exists a function $K_{M,N}$ with the following properties:
\begin{enumerate}
    \item $K_{M,N}(k)=1$ for $|k|\leqslant N,$
    \item $K_{M,N}(k)=0$ for $|k|\geqslant N+2M$
    \item when $R+S\geqslant 2N+4M$, $\dst \frac{1}{R+S}\sum_{j=-R}^{S-1}\abs{\ff_d[K_{M,N}]\left(\frac{j}{R+S}\right)}\leqslant 16 \pi ( 2+\ln(1+N/M) )$
\end{enumerate}
\end{lemma}

\begin{proof}
Define 
\begin{eqnarray*}
K_{M,N}(k)&=&\frac{1}{M}\mathbbm{d}_{N+M}*\mathbbm{f}_{M-1}(k)
=\frac{1}{M}\sum_{n\in\Z}\mathbbm{1}_{\{-M-N,\ldots,M+N\}}(k-n)\mathbbm{1}_{\{-M+1,\ldots,M+1\}}(k)\\
&=&\frac{1}{M}\sum_{\substack{|n|\leqslant M-1 \\ |n-k|\leqslant N+M}}\left(1-\frac{|n|}{M}\right).
\end{eqnarray*}

First, for $|k|\leqslant N$, if $|n|\leqslant M-1$, then $|n-k|\leqslant |n|+|k|\leq N+M-1$,
so that
$$
    K_{M,N}(k)=\frac{1}{M}\sum_{|n|\leqslant M-1}\left( 1-\frac{|n|}{M}\right)=1,
$$
since
\begin{eqnarray*}
\sum_{|n|\leqslant M-1}\left(1-\frac{|n|}{M}\right)&=&\sum_{n=1-M}^0\left( 1+\frac{n}{M} \right)+\sum_{n=1}^{M-1}\left( 1-\frac{n}{M}\right)\\
&=& \frac{M+1}{2}+\frac{M-1}{2}=M.
\end{eqnarray*}

On the other hand, if $|k|\geqslant N+2M$ and $|n|\leq M-1$ then $|k-n|\geq |k|-|n|\geq N+M+1$
so that the sum defining $K_{M,N}$ is empty and $K_{M,N}=0$.

\smallskip

To prove the last item, the Convolution Theorem shows that 
$$
    \ff_d[K_{M,N}](t)=\frac{1}{M}D_{N+M}(t)F_{M-1}(t).
$$ 

As $D_{N+M}$ and $F_{M-1}$ are both even, so is $K_{M,N}$ thus
$$
    \int_0^1|\ff_d[K_{M,N}](t)|\ud t=2\int_0^{\frac{1}{2}}|\ff_d[K_{M,N}](t)|\ud t=2(I_1+I_2+I_3)
$$
where
\begin{gather*}
    I_1=\frac{1}{M}\int_{0}^{\frac{1}{N+M}}|D_{N+M}(t)F_{M-1}(t)|\ud t\\
    I_2=\frac{1}{M}\int_{\frac{1}{N+M}}^{\frac{1}{M}}|D_{N+M}(t)F_{M-1}(t)|\ud t,\\
    I_3=\frac{1}{M}\int_{\frac{1}{M}}^{\frac{1}{2}}|D_{N+M}(t)F_{M-1}(t)|\ud t.
\end{gather*}

We have
$$
|D_{N+M}(t)|\leqslant 2(N+M)+1,~0\leqslant F_{M-1}(t)\leqslant M,~ \int_0^1F_{M-1}(t)\ud t=1.
$$
It follows that
$$
    I_1\leqslant \dfrac{1}{M}\int_0^{\frac{1}{N+M}}(2(M+N)+1)M\ud t= 2 +\frac{1}{M+N}\leqslant 3.
$$
since $M,N$ are positive integers.

Using the explicit expressions of $D_{N+M}$ and $F_{M-1}$, we have
\begin{eqnarray*}
    I_2&=& \frac{1}{M^2}\int_{\frac{1}{N+M}}^{\frac{1}{M}}\frac{|\sin(\pi(2(N+M)+1)t)|\sin^2(\pi M t)}{\sin^3(\pi t)}\ud t\leq \frac{1}{M^2}\int_{\frac{1}{N+M}}^{\frac{1}{M}}\frac{\sin^2(\pi M t)}{\sin^3(\pi t)}\ud t\\
    &\leqslant& \frac{\pi^2}{8}\int_{\frac{1}{N+M}}^{\frac{1}{M}}\frac{\ud t}{t}
    = \frac{\pi^2}{8} \ln\left(1+\frac{N}{M}\right).
\end{eqnarray*}
using that $\sin \pi t\leq \pi t$ for $t\geq 0$ and that $\sin \pi t\geq 2t$ for $0\leq t\leq \dfrac{1}{2}$.

Finally, for $I_3$, we do the same computation to bound
$$
I_3\leq\frac{1}{M^2}\int_{\frac{1}{M}}^{\frac{1}{2}}\frac{\sin^2\pi Mt}{t^3}\ud t
=\int_{1}^{\frac{M}{2}}\frac{\sin^2\pi s}{s^3}\ud s
\leq \int_{1}^{\frac{M}{2}}\frac{\ud s}{s^3}
\leq \frac{1}{2}.
$$

Grouping all terms and slightly upper bounding the numerical constants, we obtain
$$
    \int_0^1|\ff_d[K_{M,N}](t)|\ud t \leqslant 8\left(2+\ln\left( 1+\frac{N}{M} \right) \right).
$$
By Lemma \ref{l428} (and the $1$-periodicity of $\ff_d[K_{M,N}]$), we obtain
$$
    \frac{1}{R+S}\sum_{j=-R}^{S-1} \abs{\ff_d[K_{M,N}]\left(\frac{j}{2R+1}\right)}\leqslant 
    \left( 1+\frac{2\pi d}{R+S}\right)\|\ff_d[K_{M,N}]\|_{L^1([0,1])}.
$$
But $d=N+2M-1$ and $R+S\geqslant 2N+4M$ so 
$$
    \frac{2\pi d}{R+S}\leqslant \frac{\pi(2N+4M-2)}{2N+4M} \leqslant \pi.
$$
Hence we get
\begin{eqnarray*}
    \frac{1}{R+S}\sum_{j=-R}^{S-1}\abs{\ff_d[K_{M,N}]\left(\frac{j}{R+S}\right)} &\leqslant& (1+\pi)\|\ff_d[K_{M,N}]\|_{L^1([0,1])}\\
    &\leqslant& 16\pi \left( 2+\ln\left(1+\frac{N}{M}\right)  \right),
\end{eqnarray*}
concluding the proof.
\end{proof}

\begin{lemma}\label{l4210}
Let $R,S$ be positive integers and $K=\bigl(K(-R),\ldots,K(S-1)\bigr)\in\C^{R+S}$.
We extend $K$ into 

-- an $R+S$-periodic sequence $K^{(p)}\bigl(j(R+S)+\ell\bigr)=K(\ell)$, for $\ell=-R,\ldots,S-1$ and $j\in\Z$;

-- a finitely supported sequence $K^{(0)}$ by setting $K^{(0)}(\ell)=K(\ell)$ for $\ell=-R,\ldots,S-1$
and $K^{(0)}(\ell)=0$ for $\ell\geq S$ and for $\ell\leq -R-1$.

Then
$$
\int_0^1\left|\sum_{m\in \mathbb{Z}}a_mK^{(p)}(m)e^{2i\pi m t}\right|\ud t 
\leq
\frac{1}{R+S}\sum_{\ell=-R}^{S-1} \abs{\ff_d[K^{(0)}]\left(\frac{\ell}{R+S}\right)}\int_0^1\left|\sum_{m\in \mathbb{Z}}a_me^{2i\pi m t}\right|\ud t$$
\end{lemma}

\begin{proof}
Write elements of $\C^{R+S}$ as $(a_{-R},\ldots,a_{S-1})$
and the scalar product $\dst\scal{a,b}=\sum_{\ell=-R}^{S-1}a_\ell\overline{b_\ell}$.
For $j=-R,\ldots,S-1$, denote by
$\dst e_k:=\left[\frac{1}{\sqrt{R+S}}e^{2i\pi\frac{k\ell}{R+S}} \right]_{\ell=-R}^{S-1}$ 
so that $(e_k)_{k=-R,\ldots,S-1}$ is an orthonormal basis of $\C^{R+S}$.

Write $A_\ell=\dst\sum_{n\in \mathbb{Z}}\overline{a_{n(R+S)+\ell}}e^{-2i\pi (n(R+S)+\ell)t}$ for
$\ell=-R,\ldots,S-1$ and $A=(A_\ell)\in\C^{R+S}$.
Then, by periodicity of $K^{(p)}$
\begin{eqnarray*}
\sum_{m\in \mathbb{Z}}a_m K^{(p)}(m)e^{2i\pi m t}
&=&\sum_{\ell=-R}^{S-1} K(\ell)\sum_{j\in \mathbb{Z}}a_{j(R+S)+\ell}e^{2i\pi (j(R+S)+\ell)t}
=\scal{K,A}=\sum_{k=-R}^{S-1}\scal{K,e_k}\overline{\scal{A,e_k}}\\
&=&\frac{1}{R+S}\sum_{k=-R}^{S-1}\left(\sum_{\ell=-R}^{S-1} K(\ell)e^{-2i\pi k \frac{\ell}{R+S}}\right)
\left( \sum_{\ell=-R}^{S-1}\sum_{n\in \mathbb{Z}}a_{n(R+S)+\ell} e^{2i\pi(n(R+S)+\ell)t}e^{2i\pi k\frac{\ell}{R+S}} \right)\\
&=&\frac{1}{R+S}\sum_{k=-R}^{S-1} \ff_d[K_0]\left(\frac{k}{R+S}\right)
\sum_{\ell=-R}^{S-1}\sum_{n\in \mathbb{Z}}a_{n(R+S)+\ell}e^{2i\pi[(n(R+S)+\ell)t +\ell\frac{k}{R+S}]} ).
\end{eqnarray*}
Noticing that $\dst e^{2i\pi \ell \frac{k}{R+S}}=e^{2i\pi[(n(R+S)+\ell)\frac{k}{R+S}]}$, we may write
\begin{eqnarray*}
\sum_{m\in \mathbb{Z}}a_m K^{(p)}(m)e^{2i\pi m t}
&=&\frac{1}{R+S}\sum_{k=-R}^{S-1}\ff_d[K^{(0)}]\left(\frac{k}{R+S}\right)
\sum_{\ell=-R}^{S-1}\sum_{n\in \mathbb{Z}}a_{n(R+S)+\ell}e^{2i\pi(n(R+S)+\ell)(t+\frac{k}{R+S})}\\
&=&\frac{1}{R+S}\sum_{k=-R}^{S-1}\ff_d[K^{(0)}]\left(\frac{k}{R+S}\right)
\sum_{m\in \mathbb{Z}}a_me^{2i\pi m(t+\frac{k}{R+S})}.
\end{eqnarray*}
From this, we deduce that
\begin{eqnarray*}
\int_0^1\left|\sum_{m\in \mathbb{Z}}a_mK^{(p)}(m)e^{2i\pi m t}\right|\ud t 
&\leq&\frac{1}{R+S}\sum_{k=-R}^{S-1}\abs{\ff_d[K^{(0)}]\left(\frac{k}{R+S}\right)}
\int_0^1\abs{\sum_{m\in \mathbb{Z}}a_me^{2i\pi m(t+\frac{k}{R+S})}}\ud t \\
&=&\frac{1}{R+S}\sum_{k=-R}^{S-1} \abs{\ff_d[K^{(0)}]\left(\frac{k}{R+S}\right)}\int_0^1\left|\sum_{m\in \mathbb{Z}}a_me^{2i\pi m u}\right|\ud u,
\end{eqnarray*}
with the change of variable $u=t+\dfrac{j}{R+S}$ and periodicity of $\dst u\longrightarrow \sum_m a_me^{2i\pi m u}$.
\end{proof}

\begin{lemma}\label{l4211}
Let $d$, $D$ and $q$ be positive integers with $(2+2\delta)d+4\leqslant D$ for some $\delta>0$ and $q\geqslant 4\pi$.

Suppose $I$ is a finite set of integers and, for each $k\in I$ let
$f_k$ be a trigonometric polynomial of degree at most $d$.
Then for any integer $s$, we have:
$$
\int_0^1\left| \sum_{k\in I(q;s)}f_k(t)e^{2i\pi Dkt} \right|\ud t\leqslant 32\pi (2+\ln(1+2/\delta ) )
\int_0^1\left| \sum_{k\in I}f_k(t)e^{2i\pi Dkt} \right|\ud t.
$$
\end{lemma}

\begin{proof}
Assume we can prove the lemma for $s=0$, that is, for any sequence $(f_k)_{k\in\Z}$
of trigonometric polynomials of degree at most $d$ and any finite set $I$,
$$
\int_0^1\left| \sum_{q\ell\in I}f_{\ell q}(t)e^{2i\pi D\ell qt} \right|\ud t\leqslant 32\pi (2+\ln(1+2/\delta ) )
\int_0^1\left| \sum_{k\in I}f_k(t)e^{2i\pi Dkt} \right|\ud t.
$$
Then, replacing $I$ with $I-s$ and replacing $(f_k)$ with $(f_{k+s})$ we get
\begin{eqnarray*}
\int_0^1\left| \sum_{k\in I(q;s)}f_{k}(t)e^{2i\pi D\ell qt} \right|\ud t
&=&\int_0^1\left| \sum_{s+\ell q\in I}f_{s+\ell q}(t)e^{2i\pi D\ell qt} \right|\ud t\\
&\leqslant& 32\pi (2+\ln(1+2/\delta ) )\int_0^1\left| \sum_{k\in I-s}f_{s+k}(t)e^{2i\pi Dkt} \right|\ud t\\
&=&32\pi (2+\ln(1+2/\delta ) )\int_0^1\left| e^{-2i\pi Dst}\sum_{\ell\in I}f_{\ell}(t)e^{2i\pi Dkt} \right|\ud t
\end{eqnarray*}
which is the desired estimate since $|e^{-2i\pi Dst}|=1$.
So there is no loss of generality is assuming $s=0$. 

Next, we write $f_k=\dst\sum_{-d\leqslant \ell \leqslant d}a_\ell^ke^{2i\pi (Dk+\ell)t}$
and
$$
F(t)=\sum_{k\in I}\sum_{-d\leqslant \ell \leqslant d}a_\ell^ke^{2i\pi (Dk+\ell)t}
=\sum_m a_me^{2i\pi m t}
$$
where
\begin{equation}\label{e427}
 \begin{cases}a_m=a_\ell^k&\mbox{when }m=Dk+\ell,\quad k\in I \quad \text{and}\quad |\ell|\leqslant d\\
 0&\mbox{otherwise}\end{cases}.
\end{equation}
Let  $N=d$, $M=\left\lceil \dfrac{\delta d}{2} \right\rceil$ the smallest integer larger than $\dfrac{\delta d}{2}$
and let $K_{M,N}$ the sequence from Lemma \ref{l429}. Since
$$
    N+2M\leqslant d +2\left( \dfrac{\delta d}{2}+1 \right)=d
    +\delta d+2\leqslant \dfrac{D}{2}
$$ 
then we have
$$
    \supp (K_{M,N})= [-N-2M,N+2M]\subseteq \left[-\dfrac{D}{2},\dfrac{D}{2}\right].
$$
Further $K_{M,N}(m)=1$ for $m\in [-d,d]$. Next, 
$K^{(p)}_{M,N}$ to be the $qD$-periodic sequence defined by 
$K^{(p)}_{M,N}(jqD+\ell)=K_{M,N}(\ell)$ for $\ell=-N-2M,\ldots,N+2M$ and $K^{(p)}_{M,N}(k)=0$
for all other $k$'s. We take $R,S>N+2M$ such that $R+S=qD$ and then $K^{(p)}_{M,N}(k)=0$
for $k=-R,\ldots,-N-2M-1$ and for $k=N+2M+1,\ldots,S-1$.

From Lemma \ref{l4210}, we get
\begin{eqnarray*}
\int_0^1\left| \sum_m a_m K_{M,N}^{(p)}(m)e^{2i\pi m t} \right|\ud t
&\leq& \frac{1}{R+S}\sum_{j=-R}^{S-1} \abs{\ff_d[K_{M,N}]\left(\frac{j}{R+S}\right)}\int_0^1\left|\sum_{m\in \mathbb{Z}}a_me^{2i\pi m u}\right|\ud u\\
&\leq& 16\pi \left( 2+\ln\left(1+\frac{N}{M}\right)  \right)\|F\|_{L^1([0,1])}    
\end{eqnarray*}
with Lemma \ref{l429}.

As $\supp(K_{M,N})\subseteq \left[-\dfrac{D}{2},\dfrac{D}{2}\right]$ and $K_{M,N}^{(p)}$ is periodic of period $R+S=qD$ then
\begin{equation}\label{e428}
    K_{M,N}^{(p)}(m)\neq 0 \quad \text{if} \quad m=jqD+\ell' \quad \text{for}\quad  j\in \mathbb{Z}\quad \text{and}\quad |\ell'|\leqslant \dfrac{D}{2}.
\end{equation}
Combining \eqref{e427} and \eqref{e428} we have that $a_mK_{M,N}^{(p)}(m)\not=0$ only when
$$
   m= Dk+\ell=jqD+\ell'.
$$
Hence $|jqD-Dk|=|\ell-\ell'|\leqslant d+\dfrac{D}{2}<D$. But this can only happen when $jq=k$,
which then also implies $\ell=\ell'$. In particular, $m=jqD+\ell$ with $|\ell|\leq d$ and then
$$
a_mK_{M,N}^{(p)}(m)=a_{jqD+\ell}K_{M,N}^{(p)}(jqD+\ell)=a_{jqD+\ell}K_{M,N}^{(p)}(\ell)=a_{\ell}^{jqD}.
$$
It follows that
\begin{eqnarray*}
    \sum_{m\in \mathbb{Z}} a_mK_{M,N}(m)e^{2i\pi m t}&=& \sum_{j\in \mathbb{Z}}\sum_{-d\leqslant \ell \leqslant d}
    a_{jqD+\ell}e^{2i\pi[(jqD+\ell)t]}\\
    &=&\sum_{\substack{ k=0 \mod q\\ k\in I}}\sum_{-d\leqslant \ell \leqslant d}a_{kD+\ell}e^{2i\pi kDt}e^{2i\pi \ell t}\\
    &=&\sum_{k\in I(q;0)}e^{2i\pi k Dt}\sum_{-d\leqslant \ell \leqslant d}a_{kD+\ell}e^{2i\pi \ell t}\\
    &=&\sum_{k\in I(q;0)}f_k(t)e^{2i\pi k Dt}.
\end{eqnarray*}
Finally we get
\begin{eqnarray*}
    \int_0^1\left| \sum_{k\in I(q;0)}f_k(t)e^{2i\pi Kdt} \right|\ud t&=&\int_0^1\left| \sum_m a_mK_{M,N}(m)e^{2i\pi m t} \right|\ud t\\
    &\leqslant& 32\pi \left( 2+\ln( 1+N/M ) \right)\|F\|_{L^1([0,1])}\\
    &\leqslant&32\pi (2+\ln(1+2/\delta))\|F\|_{L^1([0,1])}
\end{eqnarray*}
since $\dfrac{N}{M} \leqslant \dfrac{2d}{\delta D}\leqslant \dfrac{2}{\delta}$.
\end{proof}

We can now prove the lemma.

\begin{proof}[Proof of lemma \ref{l423}]
Write $I(q;s)=\{k_1,\ldots,k_J\}$ and write each $k_j$ in the form $k_j=r_jq+s$. 
Applying lemma \ref{l4211} yields
\begin{eqnarray*}
\|F\|_{L^1([0,1])}&\geqslant&
\frac{1}{32\pi(2+\ln(1+2/\delta ))}\int_0^1\left|\sum_{j=1}^J f_{k_j}(t)e^{2i\pi k_j Dt} \right|\ud t
=\int_0^1\left| e^{2i\pi Dst}\sum_{j=1}^J f_{k_j}(t)e^{2i\pi r_jqDt}\right| \ud t\\
&=&\frac{1}{qD}\int_0^{qD}\left|\sum_{j=1}^J f_{k_j}\left(\frac{s}{qD}\right)e^{2i\pi r_js}  \right|\ud s\\
&=&\frac{1}{qD}\sum_{m=0}^{qD-1}\int_{m}^{m+1} \left| \sum_{j=1}^J f_{k_j}\left(\frac{s}{qD}\right)e^{2i\pi r_js} \right|\ud s.
\end{eqnarray*}

But
$$
\sum_{j=1}^Jf_{k_j}\left(\frac{s}{qD}\right)e^{2i\pi r_js}
=\sum_{j=1}^Jf_{k_j}\left(\frac{m}{qD}\right)e^{2i\pi r_js}
+\sum_{j=1}^J \left[ f_{k_j}\left(\frac{s}{qD}\right)-f_{k_j}\left(\frac{m}{qD}\right) \right]e^{2i\pi r_js}
$$
so that
$$
\|F\|_{L^1([0,1])}\geqslant \frac{1}{32 (2+\ln(1+2/\delta))}(T_1-T_2)
$$
with
$$
T_1=\frac{1}{qD}\sum_{m=0}^{qD-1}\int_{m}^{m+1} \left| \sum_{j=1}^J f_{k_j}\left(\frac{m}{qD}\right)e^{2i\pi r_js}\right|\ud s,
$$
and
$$
T_2=\frac{1}{qD}\sum_{m=0}^{qD-1}\int_{m}^{m+1}\left| \sum_{j=1}^J \left(f_{k_j}\left(\frac{s}{qD}\right)
-f_{k_j}\left(\frac{m}{qD}\right) \right)e^{2i\pi r_js} \right|\ud s.
$$
It remains to show that 
$$
    T_1\geqslant \frac{C_{MPS}}{2}\sum_{j=1}^J\frac{\|f_{k_j}\|_{L^1}}{j}
\quad\mbox{and}\quad
T_2\leq  \frac{2\pi d}{qD}  \sum_{j=1}^J\|f_{k_j}\|_{L^1}.
$$

\smallskip

Let us start with $T_1$: using the $1$-periodicity in $s$,
\begin{eqnarray*}
    T_1&=&\frac{1}{qD}\sum_{m=0}^{qD-1}\int_{0}^{1} \left| \sum_{j=1}^J f_{k_j}\left(\frac{m}{qD}\right)e^{2i\pi r_js}\right|\ud s,\\
    &\geqslant&\frac{C_{MPS}}{qD}\sum_{m=0}^{qD(1}\sum_{j=1}^J\frac{\abs{f_{k_j}\left(\frac{m}{qD}\right)}}{j}
    =\frac{C_{MPS}}{qD}\sum_{j=1}^J\frac{1}{j}\sum_{m=0}^{qD-1}\abs{f_{k_j}\left(\frac{m}{qD}\right)}
\end{eqnarray*}
with Theorem \ref{th:MPS}.
Applying Lemma \ref{l428} to $f_{k_j}$ and using that $\dfrac{2\pi d}{qD}<\dfrac{1}{2}$ with our hypothesis on $q$ and $D$, we get 
$$
    \frac{1}{qD}\sum_{m=1}^{qD}\abs{f_{k_j}\left(\frac{m-1}{qD}\right)}\geqslant \frac{\|f_{k_j}\|_{L^1}}{2}
$$
and the desired estimate of $T_1$ follows immediately.

\smallskip

Let us now estimate $T_2$. For $s\in [m,m+1]$, we have
\begin{eqnarray*}
\abs{\sum_{j=1}^J \left(f_{k_j}\left(\frac{s}{qD}\right)-f_{k_j}\left(\frac{m}{qD}\right)\right)e^{2i\pi r_js}}
&=&\abs{\sum_{j=1}^J \int_{\frac{m}{qD}}^{\frac{s}{qD}}f'_{k_j}(t)e^{2i\pi r_js}\ud t }\\
&\leqslant& \int_{\frac{m}{qD}}^{\frac{m+1}{qD}}\abs{\sum_{j=1}^J f'_{k_j}(t)e^{2i\pi r_js}}\ud t.
\end{eqnarray*}
From the $1$-periodicity in $s$, the integral of this quantity over $[m,m+1]$ is the same as
the integral over $[0,1]$. Thus
\begin{eqnarray*}
T_2&\leqslant& 
\frac{1}{qD}\sum_{m=0}^{qD-1}\int_m^{m+1}\int_{\frac{m}{qD}}^{\frac{m+1}{qD}}
\abs{\sum_{j=1}^J f'_{k_j}(t)e^{2i\pi r_js}}\ud t\ud s\\
&=&\frac{1}{qD}\sum_{m=0}^{qD-1}\int_{\frac{m}{qD}}^{\frac{m+1}{qD}}\int_0^{1}
\abs{\sum_{j=1}^J f'_{k_j}(t)e^{2i\pi r_js}}\ud s\ud t\\
&=& \frac{1}{qD}\int_0^1\int_0^1\abs{\sum_{j=1}^J f'_{k_j}(t)e^{2i\pi r_js}}\ud s\ud t\\
    &\leqslant& \frac{1}{qD}\sum_{j=1}^J\|f'_{k_j}\|_{L^1([0,1])}\leqslant \frac{2\pi d}{qD}\sum_{j=1}^J\|f_{k_j}\|_{L^1([0,1])}
\end{eqnarray*}
with Bernstein's inequality. 
\end{proof}

\subsection{Newman's extremal trigonometric polynomial}

Finally, let us insist on the fact that all results in this section are worse case scenarios
{\it i.e.} the lowest possible $L^1$ norm of a trigonometric polynomial with eventually
a constraint on the spectrum or on the coefficients.

If one looks at the maximal possible value of the $L^1$-norm under the constraint that all
coefficients have modulus $1$, then, from Cauchy-Schwarz
$$
\int_0^1\abs{\sum_{j=0}^Na_je^{2i\pi jt}}\ud t\leq \left(\int_0^1\abs{\sum_{j=0}^Na_je^{2i\pi jt}}^2\ud t\right)^{\frac{1}{2}}
=\left(\sum_{j=0}^N|a_j|^2\right)^{\frac{1}{2}}=\sqrt{N+1}
$$
if $|a_j|=1$ for all $j$'s. However, if one wants to reach the bound $\sqrt{N+1}$,
one would need that equality holds in the Cauchy-Schwarz Inequality so $\dst\abs{\sum_{j=0}^Na_je^{2i\pi jt}}$
should be constant. But, it is easy to show that the only trigonometric polynomials with constant
modulus are monomials (thus in our case, equality can only occur when $N=0$).
Indeed, consider the
polynomial $\dst P(z)=\sum_{j=0}^Na_jz^j$ with $a_N\not=0$. Introduce $\dst P^*(z)=\overline{P(\bar z)}=\sum_{j=0}^N\overline{a_j}z^j$
 Then $|P(z)|^2=c$ when $|z|=1$ can be written in the form
$P(z)P^*(1/z)=c$ for $|z|=1$, which is now an equality between meromorphic functions over $\C\setminus\{0\}$.
As this is true on $\{|z|=1\}$ it is true over all of $\C\setminus\{0\}$, in particular, $P$ has only $0$ as zero
thus $P(z)=a_Nz^N$.

So the natural question is, how near to $\sqrt{N+1}$ one can get.
A result by E. Beller and D.\,J. Newman \cite{BN} shows that
one can obtain $\sqrt{N}-O(1)$. For sake of simplicity, we will only present a previous result
by D.\,J. Newman \cite{Ne} which contains similar ingredients:

\begin{theorem}[Newman, \cite{Ne}]
There is an absolute constant $c>0$ and an integer $N_0$ such that, for every $N\geq N_0$
there is a sequence $a_0,\ldots,a_N$ with $|a_j|=1$ for all $j$
and such that
$$
\int_0^1\abs{\sum_{j=0}^Na_je^{2i\pi j t}}\,\ud t\geq \sqrt{N}-c.
$$
\end{theorem}

\begin{proof}
The example is related to Gauss sums.
We fix $N$ and set $\omega=\exp\left(i\dfrac{\pi}{N+1}\right)$, $a_j=\omega^{j^2}$
and define
$$
P(t)=\sum_{j=0}^Na_je^{2i\pi jt}=\sum_{j=0}^N\exp\left(2i\pi\left(\frac{j^2}{2(N+1)}+jt\right)\right).
$$

The theorem is based on the following lemma:

\begin{lemma}
\label{lem:newman}
With the previous notation, when $N\to+\infty$
$$
\int_0^1|P(t)|^4\ud t=N^2+O(N^{3/2}).
$$
\end{lemma}

We postpone the proof of the lemma and show how we can conclude with it.

Parseval's identity gives $\|P\|_2=\sqrt{n+1}$ and then H\"older's inequality implies
$$
N\leq N+1=\int_0^1|P(t)|^2\ud t\leq \left(\int_0^1|P(t)|\ud t\right)^{\frac{2}{3}}\left(\int_0^1|P(t)|^4\ud t\right)^{\frac{1}{3}}.
$$
From the lemma, we deduce that there is a constant $A>0$ such that $\dst\int_0^1|P(t)|^4\ud t\leq N^2+AN^{3/2}$
when $N$ is large enough, so that
\begin{eqnarray*}
\int_0^1|P(t)|\ud t&\geq&\frac{N^{3/2}}{(N^2+AN^{3/2})^{1/2}}
=N^{1/2}\frac{1}{(1+AN^{-1/2})^{1/2}}
=N^{1/2}\bigl(1-\frac{A}{2}N^{-1/2}+o(N^{-1/2})\bigr)\\
&\geq&N^{1/2}-A
\end{eqnarray*}
for $N$ large enough.
\end{proof}

It remains to prove the lemma:

\begin{proof}[Proof of Lemma \ref{lem:newman}]
We write
$$
\abs{\sum_{j=0}^Na_je^{2i\pi jt}}^2=\left(\sum_{j=0}^Na_je^{2i\pi jt}\right)\left(\sum_{k=0}^N\overline{a_k}e^{-2i\pi kt}\right):=\sum_{\ell=-N}^Nc_\ell e^{2i\pi \ell t}.
$$
Note that $c_0=\dst\sum_{j=0}^Na_j\overline{a_j}=N+1$, $c_{-\ell}=\overline{c_\ell}$ and, for $\ell\geq1$,
\begin{eqnarray*}
c_\ell&=&\sum_{k=0}^{N-\ell}a_{\ell+k}\overline{a_k}
=\sum_{k=0}^{N-\ell}\omega^{(\ell+k)^2}\omega^{-k^2}=\omega^{\ell^2}\sum_{k=0}^{N-\ell}\omega^{2k\ell}\\
&=&\omega^{\ell^2}\frac{1-\omega^{2(N-\ell+1)\ell}}{1-\omega^{2\ell}}=-\omega^{-\ell}\frac{\omega^{\ell^2}-\omega^{-\ell^2}}{\omega^\ell-\omega^{-\ell}}
\end{eqnarray*}
since $\omega^{2(N+1)}=1$. We thus obtain that
$$
|c_\ell|^2=\frac{\sin^2 \frac{\pi}{N+1}\ell^2}{\sin^2 \frac{\pi}{N+1}\ell}.
$$
On the other hand, from Parseval,
$$
\int_0^1|P(t)|^4\ud t=\sum_{\ell=-N}^N|c_\ell|^2=(N+1)^2+2\sum_{\ell=1}^N|c_\ell|^2
$$
since $c_0=N+1$ and $|c_{-\ell}|=|c_\ell|$.

To compute the last sum, we split it into 4 parts
$$
\begin{matrix}
S_1&=&\sum_{1\leq \ell\leq\sqrt{N}}|c_\ell|^2&\ &S_2&=&\sum_{\sqrt{N}<\ell\leq\dfrac{N+1}{2}}|c_\ell|^2\\
S_3&=&\sum_{\dfrac{N+1}{2}<\ell<N+1-\sqrt{N}}|c_\ell|^2&&S_4&=&\sum_{N+1-\sqrt{N}\leq \ell\leq N}|c_\ell|^2.
\end{matrix}
$$
As $|c_{N+1-\ell}|=|c_\ell|$ we have $S_4\leq S_1$ and $S_3\leq S_2$
so that we only need to show that $S_1$ and $S_3$ are both $O(N^{3/2})$.

For $S_1$, we use the estimate $\abs{\dfrac{\sin \ell\theta}{\sin\theta}}\leq \ell$ to bound $|c_\ell|^2\leq \ell^2$, leading to
$$
S_1\leq \sum_{1\leq \ell\leq\sqrt{N}}\ell^2\leq N\sum_{1\leq \ell\leq\sqrt{N}}1=N^{3/2}.
$$

Further, as for $0\leq\theta\leq\dfrac{\pi}{2}$, $\sin\theta\geq\dfrac{2}{\pi}\theta$, we bound
$|c_\ell|^2\leq \dfrac{(N+1)^2}{4\ell^2}$ so that
$$
S_2\leq\dfrac{(N+1)^2}{4}\sum_{\sqrt{N}<\ell\leq\dfrac{N+1}{2}}\frac{1}{\ell^2}
\leq \dfrac{(N+1)^2}{4}\int_{\sqrt{N}-1}^{+\infty}\dfrac{\mathrm{d}t}{t^2}
=\dfrac{(N+1)^2}{4(\sqrt{N}-1)}=O(N^{3/2})
$$
as claimed.
\end{proof}

\section{$L^1$ estimates with real frequencies}

\subsection{Some results for the Besicovitch norm}

We will now present three results for the Besicovitch norms of non-harmonic trigonometric polynomials.

Let us start with the analogue of Mc\,Gehee, Pigno and Smith's result.
Let us present a nice argument by Hudson and Leckband that shows that the estimate for the Besicovitch
norm follows from Theorem \ref{th:MPS}. The converse is of course true as well and the result
was obtained directly in \cite{JKS} and even allowed to slightly improve the constant.

\begin{theorem}[Hudson \& Leckband \cite{HL}]
\label{th:HL}
For $(\lambda_j)_{j\geq0}$ be real numbers with $\lambda_{j+1}-\lambda_j>0$ and let
$(a_j)_{j\geq 0}$ be a sequence of complex numbers with finite support.
Then
$$
\lim_{T\to+\infty}\frac{1}{T}\int_{-T/2}^{T/2}\left|\sum_{j=0}^{+\infty} a_je^{2i\pi \lambda_j t}\right|\,\mathrm{d}t
\geq C_{MPS}\sum_{j=0}^{+\infty}\frac{|a_j|}{j+1}
$$
where $C_{MPS}$ is the same constant as in Theorem \ref{th:MPS}.

Moreover, assume that $(\lambda_k)$ is $q$-lacunary {\it i.e.} $\lambda_0\geq 1$ and
$\lambda_{k+1}\geq q\lambda_k$. Then, for $1\leq p<+\infty$,
$$
A_{p,q}\left(\sum_{j=0}^{+\infty}|a_j|^2\right)^{1/2}\leq
\lim_{T\to+\infty}\left(\frac{1}{T}\int_{-T/2}^{T/2}\left|\sum_{j=0}^{+\infty} a_je^{2i\pi \lambda_j t}\right|^p\,\mathrm{d}t\right)^{\frac{1}{p}}
\leq B_{p,q}\left(\sum_{j=0}^{+\infty}|a_j|^2\right)^{1/2}
$$
with $A_{p,q}$, $B_{p,q}$ the constants in Theorem \ref{th:lacunary}.
\end{theorem}

Actually, only the first part is in \cite{HL}. The second part is here obtained with almost the same proof
and seems to be new.

\begin{proof}
Let $a_0,\ldots,a_N$ be complex numbers,
$\lambda_0<\lambda_1<\cdots<\lambda_N$ be real numbers 
and 
$$
\Phi(t)=\sum_{j=0}^N a_je^{2i\pi \lambda_j t}.
$$

Let $\eps>0$. By a lemma of Dirichlet (\cite[p 235]{Zy}, \cite{Gr}), there 
is an increasing sequence of integers $(M_n)_{n\geq 1}$ and,
for each $n\geq 1$ a finite family of integers 
$(N_{j,n})_{j=0,\ldots,N}$ such that
$$
\abs{\lambda_j-\frac{N_{j,n}}{M_n}}<\frac{\eps}{M_n}\qquad\mbox{for }j=0,\ldots,N
$$
which implies that
$$
\abs{e^{2i\pi \lambda_j t}-e^{2i\pi \frac{N_{j,n}}{M_n}t}}
\leq 2\pi\abs{\lambda_j-\frac{N_{j,n}}{M_n}}|t|\leq 2\pi\frac{\eps}{M_n}|t|\qquad\mbox{for }j=0,\ldots,N.
$$
Define the $M_n$-periodic function
$$
\Psi_n(t)=\sum_{j=0}^N a_je^{2i\pi N_{j,n} t/M_n}
$$
and note that, for $t\in[-M_n/2,M_n/2]$,
$$
\abs{\Phi(t)-\Psi_n(t)}\leq \sum_{j=0}^N |a_j|\abs{e^{2i\pi \lambda_j t}-e^{2i\pi \frac{N_{j,n}}{M_n}t}}
\leq 2\pi \eps \sum_{j=0}^N |a_j|.
$$
But then
\begin{equation}
\begin{aligned}
\abs{\left(\frac{1}{M_n}\int_{-M_n/2}^{M_n/2}\abs{\Phi(t)}^p\ud t\right)^{1/p}
-\left(\frac{1}{M_n}\int_{-M_n/2}^{M_n/2}\abs{\Psi_n(t)}^p\ud t\right)^{1/p}}\\
&\leq\left(\frac{1}{M_n}\int_{-M_n/2}^{M_n/2}\abs{\Phi(t)-\Psi_n(t)}^p\ud t\right)^{1/p}\notag\\
&\leq  \left(2\pi  \sum_{j=0}^N |a_j|\right)^{1/p}\eps^{1/p}.
\end{aligned}
\label{eq:approxHL}
\end{equation}
\smallskip

We can now conclude as follows. First, in the general case,
As $\Psi_n$ is $M_n$-periodic, we may apply Theorem \ref{th:MPS} to obtain
\begin{eqnarray*}
C_{MPS}\sum_{j=0}^N\frac{|a_j|}{j+1}&\leq&\frac{1}{M_n}\int_{-M_n/2}^{M_n/2}|\Psi_n(t)|\,\mbox{d}t\\
&\leq& \frac{1}{M_n}\int_{-M_n/2}^{M_n/2}|\Phi(t)|\,\mbox{d}t+
+\frac{1}{M_n}\int_{-M_n/2}^{M_n/2}|\Phi(t)-\Psi_n(t)|\,\mbox{d}t\\
&\leq& \frac{1}{M_n}\int_{-M_n/2}^{M_n/2}|\Phi(t)|\,\mbox{d}t
+2\pi \eps \sum_{j=0}^N |a_j|
\end{eqnarray*}
Letting $n\to+\infty$ and then $\eps\to 0$ we obtain Theorem \ref{th:HL}:
$$
\lim_{T\to+\infty}\frac{1}{T}\int_{-T/2}^{T/2}\left|\sum_{k=0}^N a_ke^{2i\pi \lambda_k t}\right|\,\mathrm{d}t
\geq C_{MPS}\sum_{k=0}^N\frac{|a_k|}{k+1}.
$$

Let us now assume further that $\lambda_{k+1}\geq q\lambda_k$. Let $\tilde q$ be such that
$1<\tilde q<q$ and, assume that $\eps$ has been chosen such that $q-(1+q)\eps>\tilde q$.
Observe that
\begin{eqnarray}
\dfrac{N_{j+1,n}}{M_n}&\geq& \lambda_{j+1}-\abs{\lambda_{j+1}-\dfrac{N_{j+1,n}}{M_n}}
\geq q\lambda_j-\frac{\eps}{M_n}\\
&\geq& q\dfrac{N_{j,n}}{M_n}-q\abs{\lambda_{j}-\dfrac{N_{j,n}}{M_n}}-\frac{\eps}{M_n}\\
&\geq& q\dfrac{N_{j,n}}{M_n}-(1+q)\frac{\eps}{M_n},
\end{eqnarray}
that is $N_{j+1,n}\geq qN_{j,n}-(1+q)\eps\geq \tilde qN_{j,n}$.

Applying \eqref{cor:lacunary} to $n_k=N_{k,n}$ and $M=M_n$ we obtain
$$
A_{p,\tilde q}
\left(\sum_{j=0}^{+\infty}|a_j|^2\right)^{\frac{1}{2}}\leq\left(\frac{1}{M_n}\int_{-M_n/2}^{M_n/2}\abs{\sum_{j\geq 0}a_je^{2i\pi \frac{N_{j,n}}{M_n}{t}}}^p\,\mathrm{d}t\right)^{\frac{1}{p}}
\leq B_{p,\tilde q}\left(\sum_{j=0}^{+\infty}|a_j|^2\right)^{\frac{1}{2}}.
$$
From \eqref{eq:approxHL} we conclude that
\begin{multline*}
A_{p,\tilde q}
\left(\sum_{j=0}^{+\infty}|a_j|^2\right)^{\frac{1}{2}}-\left(2\pi  \sum_{j=0}^N |a_j|\right)^{1/p}\eps^{1/p}
\leq\left(\frac{1}{M_n}\int_{-M_n/2}^{M_n/2}\abs{\sum_{j\geq 0}a_je^{2i\pi \frac{N_{j,n}}{M_n}{t}}}^p\,\mathrm{d}t\right)^{\frac{1}{p}}\\
\leq B_{p,\tilde q}\left(\sum_{j=0}^{+\infty}|a_j|^2\right)^{\frac{1}{2}}+\left(2\pi  \sum_{j=0}^N |a_j|\right)^{1/p}\eps^{1/p}.
\end{multline*}
The result follows by letting $\eps\to 0$ and then $\tilde q\to q$.
\end{proof}

\subsection{Nazarov's theorem}

\begin{theorem}[Quantitative version of Nazarov theorem for small $\delta$] \label{t2}
There exists positive constants $c_*,\delta_*$ such that, for $0<\delta<\delta_*$ 
the following holds:

for every $\lambda_{0}<...< \lambda_{N}$ of real numbers satisfying
$\lambda_{k+1}- \lambda_{k}\geqslant 1$ and $a_{0},...,a_{N}$ complex numbers,
\begin{equation}\label{e15}
\sum_{k=0}^{N} \frac{|a_{k}|}{k+1}\leqslant \dfrac{c_*}{\delta^{15/2}}\int_{-\frac{1+\delta}{2}}^{\frac{1+\delta}{2}} \abs{\sum_{k=0}^{N}a_{k}e^{2i\pi\lambda_{k}t}}\ud t.
\end{equation}
\end{theorem}

First, we start by some preliminary notations and results. 
Let $I=I_\delta:=[-\frac{1+\delta}{2},\frac{1+\delta}{2}]$.

Let us fix $(\lambda_k)_{k=0,\ldots,N}\subset\R$ with $\lambda_{k+1}-\lambda_k\geq 1$ for every $k$,
$(a_k)_{k=0,\ldots,N}$ a sequence of complex numbers and write $|a_k|=a_ku_k$ with $|u_k|=1$.
Let
$$
S=\sum_{k=0}^{N}\frac{|a_{k}|}{k+1} 
\qquad\mbox{and}\qquad \phi(t)=\sum_{k=0}^{N}a_{k}e^{2i\pi\lambda_{k}t} ,
$$
so that we must find $C_{\delta}$ such that: 
\begin{equation}
\label{eq:aimnaz}
\|\phi\|_{L^1(I)}\geq C_{\delta} S.
\end{equation}
We define
$$
S_{\delta}=\sum_{k=0}^{N}\frac{|a_{k}|}{k+N_{\delta}}
\quad\mbox{and}\quad
T_{\delta}(t)=\sum_{k=0}^N\frac{u_{k}}{k+N_{\delta}}e^{-2i\pi\lambda_{k}t}
$$
where $N_{\delta}$ is a large integer that we will adjust through the proof. 
This integer will be of the form $N_{\delta}=2^{m_{\delta}}$.
We will prove that
\begin{equation}
\label{eq:aimnaz2}
\|\phi\|_{L^1(I)}\geq B_{\delta} S_{\delta}
\end{equation}
and, as $k+N_{\delta}\leq (k+1)N_{\delta}$, $S_{\delta}\geq\dfrac{S}{N_{\delta}}$
so that we obtain the desired inequality \eqref{eq:aimnaz} 
with a constant $C_{\delta}=\dfrac{B_{\delta}}{N_{\delta}}$.

The meaning of $N_{\delta}$ is the following.
Consider:

-- a new sequence of frequencies $(\tilde\lambda_j)_{j\in\Z}$
such that $\tilde\lambda_j=\lambda_{j-N_{\delta}}$ for $j=N_{\delta},\ldots,N_{\delta}+N$.
and then  $\tilde\lambda_j=\lambda_0+j-N_{\delta}$ for $j<N_{\delta}$
and  $\tilde\lambda_j=\lambda_N+j-N$ for $j>N$. In particular, we still have $|\tilde\lambda_j-\tilde\lambda_k|\geq 1$.
In other words, the sequence $(\lambda_j)_{j=0,\ldots,N}$ is completed into a sequence $(\lambda_j)_{j\in\Z}$
that is still $1$-separated and then shifting it by $N_{\delta}$.

-- A new sequence of complex numbers $(\tilde a_j)_{j\in\Z}$ with 
$\tilde a_j=a_{j-N_{\delta}}$ for $j=N_{\delta},\ldots,N_{\delta}+N$ and $\tilde a_j=0$ for other 
$j$'s. In other words, the sequence $(a_j)_{j=0,\ldots,N}$ is completed into a sequence $(a_j)_{j\in\Z}$
by $0$-padding it and then shifting it by $N_{\delta}$.

Then \eqref{eq:aimnaz2} reads
$$
\int_{I_\delta}\abs{\sum_{j\in\Z}\tilde a_je^{2i\pi\tilde\lambda_j t}}\,\mbox{d}t
\geq B_{\delta}\sum_{j\in\Z}\frac{|\tilde a_j|}{j+N_{\delta}}
$$
with the convention that $0/0=0$. 

Note that, up to adding $0$ terms at the end of the sequence $a_j$, we may assume that $N+N_{\delta}$
is of the form $2^{n_{\delta}}-1$ for some integer $n_\delta$. This will allow us to write
$$
\sum_{k=0}^N=\sum_{j=m_{\delta}}^{n_{\delta}}\sum_{2^j\leq r+N_{\delta}<2^{j+1}}.
$$

\smallskip

Next, as in Ingham's proof, we will introduce an auxiliary function. Again, we consider
$$
h(t) = \begin{cases}
       \cos(\pi t ) & \mbox{if } |t| \leqslant \frac{1}{2}\\[6pt]
        0  & \mbox{otherwise}
        \end{cases}
$$
whose Fourier transform is given by
$$
\widehat{h}(\lambda)=\dfrac{2}{\pi}\dfrac{\cos(\pi\lambda)}{1-4\lambda^{2}}.
$$
We will need to smooth a bit this function to obtain a better decay of the Fourier transform and thereby 
slightly enlarge its support. More precisely, let $p=10$, $q=8$ and let
$$
f_\delta(t)=\dfrac{p+q}{\delta} \mathbbm{1}_{[-\frac{\delta}{2(p+q)} ,\frac{\delta}{2(p+q)}]}(t)
$$
and define its $p+q$-fold self convolution\footnote{If $\chi$ is a function, we write $*_1\psi=\psi*\psi$
for the convolution of $\psi$ with itself and then define inductively $*_{k+1}\psi=(*_k\psi)*\psi$.}
$$
g_{\delta}= *_{p+q} f_\delta.
$$ 
Clearly $g_{\delta}$ is non-negative, even and with support $[-\frac{\delta}{2},\frac{\delta}{2} ]$.
Finally, we define $\varphi_\delta$ as
\begin{equation}
\label{eq:defffi}
\varphi_\delta=\displaystyle\frac{\pi}{2}h*g_{\delta}.
\end{equation}

In the following lemma, we list the properties needed on $\varphi_\delta$. They
are all established via easy calculus and straight forward Fourier analysis.

\begin{lemma} \label{l282}
There is a $c_0>0$ and a $\delta_0>0$ such that,
if $0<\delta < \delta_0$ then,
\begin{enumerate}
\item $\widehat{\varphi_\delta}(\lambda)=\dfrac{\cos(\pi\lambda)}{1-4\lambda^{2}}\sinc^{p+q}\left(\dfrac{\pi \delta \lambda}{p+q}\right)$,

\item $\| \varphi_\delta \|_{\infty}\leqslant \dfrac{\pi}{2}$.

\item\label{l2823} Let $D_\delta=c_0\delta^{-\frac{p+q}{p+2}}\geq 1$ and $\nu>0$ then, for $|\lambda|\geqslant \max(1,D_\delta\nu^{\frac{1}{p+2}})$,
$$
|\widehat{\varphi}(\lambda)|\leqslant \displaystyle\frac{\nu}{|\lambda|^q}
$$

\item Let $\gamma_\delta= \left(\dst\sinc \dfrac{\pi \delta}{p+q}\right)^{p+q}$ then, for $|\lambda|\geq 1$,
$$
|\widehat{\varphi_\delta}(\lambda)|\leqslant\frac{\gamma_\delta}{4\lambda^{2}-1}.
$$
\end{enumerate}
\end{lemma}

From now on, we will assume that $0<\delta < \frac{7}{3\pi}<1$ so that $D_\delta\geq 1$.
In particular, when $|\lambda|\geq D_\delta$, we have $|\widehat{\varphi}(\lambda)|\leqslant |\lambda|^{-3}$.

\begin{proof}[Proof of Lemma \ref{l282}] The first one is simple Fourier analysis.

For the second one, we write
$$
\| \varphi_\delta\|_{\infty} \leqslant \frac{\pi}{2}\|g_\delta \|_{1} \|h \|_{\infty}
\leqslant \frac{\pi}{2}\|f_\delta \|_{1}^{p+q} \| h \|_{\infty}
$$
and use simple calculus to conclude.

For the third one, we notice that $4\lambda^2-1\geq3\lambda^2$ when $\lambda\geq 1$ thus
$$
|\widehat{\varphi_\delta}(\lambda)|=\left | \dfrac{\cos(\pi \lambda)}{4\lambda^2-1}\left(\dfrac{\sin(\frac{\pi \delta \lambda}{p+q})}{\frac{\pi \delta \lambda}{p+q}}\right)^{2p}  \right|
\leqslant \dfrac{1}{3}\left(\dfrac{p+q}{\pi}\right)^{p+q}\delta^{-(p+q)}
\frac{1}{|\lambda|^{p+2}}\frac{1}{|\lambda|^q}  \leqslant \dfrac{\nu}{|\lambda|^q},
$$
since $|\lambda|\geq\dfrac{1}{3^{\frac{1}{p+2}}}\left(\dfrac{p+q}{\pi}\right)^{\frac{p+q}{p+2}}\delta^{-\frac{p+q}{p+2}}\nu^{\frac{1}{p+2}}$. Thus this fact is established with 
$c_0=\dfrac{1}{3^{\frac{1}{p+2}}}\left(\dfrac{p+q}{\pi}\right)^{\frac{p+q}{p+2}}$.

For the last one, we take $\delta_0$ small enough to have 
$\sinc \dfrac{\pi \delta_0}{p+q}=\sup_{t\geq\delta_0}|\sinc t|$ and then
$|\widehat{g_\delta}(\lambda)|\leq\gamma_\delta$ when $|\lambda|\geq 1$.
The first identity allows to conclude.
\end{proof}

We can now state the first crucial result in this proof.

\begin{lemma} \label{l285}
There exist $\delta_1>0$, $c_1>0$ such that,
if $0<\delta<\delta_1$, $0<c_1\delta^2<1-\sqrt{\gamma_\delta}$.

Moreover, let
$$
\alpha_{\delta}=1-\left(c_1\delta^2+\dfrac{\gamma_\delta}{1-c_1\delta^2}\right)
$$
and let $m_{\delta}$ be such that
$$
N_{\delta}=2^{m_{\delta}}\geq \delta^{-7/2}.
$$
Then, for $0\leqslant k \leqslant N$,
$$
\sum_{\substack{0\leqslant j\leqslant N \\ j\neq k}} 
\frac{|\widehat{\varphi}(\lambda_{j}-\lambda_{k})|}{j+N_{\delta}}\leqslant \frac{1-\alpha_{\delta}}{k+N_{\delta}}.
$$
\end{lemma}

\begin{proof}
The Taylor expansion when $\delta\to 0$ of $1-\sqrt{\gamma_\delta}$
is of the form 
$$
1-\sqrt{\gamma_\delta}=A\delta^2+O(\delta^4),
$$
with $A>0$. Thus, if $c_1<A$, for $\delta$ small enough, $0<c_1\delta^2<1-\sqrt{\gamma_\delta}$.
Next, notice that $0<1-\left(\beta+\dfrac{\gamma_\delta}{1-\beta}\right)<1$
if $0<\beta<1-\sqrt{\gamma_\delta}$ which shows that $0<\alpha_\delta<1$.
For future use, note that there is a $\kappa>0$ such that
\begin{equation}
\label{eq:estalpha}
\alpha_{\delta}=\kappa\delta^2+O(\delta^4).
\end{equation}

We will further assume that $\delta$ is small enough for
$$
\delta^{-7/2}
\geq \max\left(c_1^{-\frac{q+1}{q-2}}\delta^{-2\frac{q+1}{q-2}},c_2\delta^{-2-\frac{p+q}{p+2}}\right)
=\max\left(c_1^{-\frac{q+1}{q-2}}\delta^{-3},\dfrac{c_0}{c_1}\delta^{-7/2}\right)
$$
since we chose $q=8$ and $p=10$.

Note that, for every $\eps>0$, the power $7/2$ could be reduced to $3+\eps$ by taking $p$ large enough,
but could not be reduced below $3$ with this construction.

We can now turn to the estimate itself. Set $\beta=c_1\delta^2$ and split the sum in the left hand side of the main inequality into two sums 
$$
E:=\sum_{\substack{0\leqslant j\leqslant N \\ j\neq k}} \frac{|\widehat{\varphi}(\lambda_{j}-\lambda_{k})|}{j+N_{\delta}}
=E_{1}+E_{2}
$$
where
$$
E_{1}=\sum_{j+N_{\delta}<(1-\beta)(k+N_{\delta})}
\frac{|\widehat{\varphi}(\lambda_{j}-\lambda_{k})|}{j+N_{\delta}}
$$
and 
$$
E_{2}=\sum_{\substack{j+N_{\delta}\geqslant (1-\beta)(k+N_{\delta}) \\ j\neq k}} \frac{|\widehat{\varphi}(\lambda_{j}-\lambda_{k})|}{j+N_{\delta}}.
$$
The result is obtained if we prove the two estimates
$$
E_{1}\leqslant \frac{\beta}{k+N_{\delta}}\quad \text{et} \quad E_{2}\leqslant \frac{\gamma_\delta /(1-\beta)}{k+N_{\delta}}.
$$ 

\smallskip

Now, as $N_\delta\geq \dfrac{c_0}{c_1}\delta^{-2-\frac{p+q}{p+2}}$,  $\beta N_{\delta}\geqslant D_\delta$.
Then, if $j$ is an index corresponding to  $E_{1}$, then
$$
|\lambda_{k}-\lambda_{j}|\geqslant |k-j|=(k+N_{\delta})-(j+N_{\delta})
\geqslant\beta(k+N_{\delta})\geqslant\beta N_{\delta}\geqslant D_\delta
$$
hence, from Lemma \ref{l282}\ref{l2823} with $\nu=1$,
\begin{eqnarray*}
E_{1}&\leqslant&\sum_{j+N_{\delta}<(1-\beta)(k+N_{\delta})}
|\widehat{\varphi}(\lambda_{j}-\lambda_{k})|
\leqslant \sum_{j+N_{\delta}<(1-\beta)(k+N_{\delta})}\frac{1}{|\lambda_{j}-\lambda_{k}|^{q}}\\
&\leqslant&\sum_{j+N_{\delta}<(1-\beta)(k+N_{\delta})}\frac{1}{\bigl(\beta(k+N_{\delta})\bigr)^{q}}.
\end{eqnarray*}
But, $E_1$ contains less than $k$ terms so
\begin{equation}\label{e285}
E_{1}\leqslant\frac{k}{k+N_{\delta}}\frac{\beta^{-(q+1)}}{(k+N_{\delta})^{q-2}}
\frac{\beta}{k+N_{\delta}}\leqslant \frac{\beta}{k+N_{\delta}}
\end{equation}
since $N_{\delta}\geqslant \beta^{-\frac{q+1}{q-2}}=c_1^{-\frac{q+1}{q-2}}\delta^{-2\frac{q+1}{q-2}}$.

\smallskip

We shall now bound $E_{2}$. In this sum,
$$
j+N_{\delta}\geqslant(1-\beta)(k+N_{\delta})\quad \text{ and }\quad |\lambda_{j}-\lambda_{k}|\geqslant 1
$$
then
\begin{align*}
E_{2}&\leqslant
\frac{1}{(1-\beta)(k+N_{\delta})}\sum_{\substack{1 \leqslant j \leqslant N \\ j\neq k}}\frac{\gamma_\delta}{4(\lambda_{j}-\lambda_{k})^{2}-1}\\
&\leqslant \frac{\gamma_{\delta}}{(1-\beta)(k+N_{\delta})}
\sum_{\substack{1 \leqslant j \leqslant N \\ j\neq k}}\frac{1}{4(j-k)^{2}-1}\\
&\leqslant\frac{\gamma_{\delta}}{(1-\beta)(k+N_{\delta})} \sum_{\ell=1}^{\infty}\frac{2}{4\ell^{2}-1}.
\end{align*}
Since  $\displaystyle \frac{2}{4\ell^{2}-1}=\frac{1}{2\ell-1}-\frac{1}{2\ell+1}$,
we obtain the expected bound
$\dst
E_{2}\leqslant \frac{\gamma_\delta/(1-\beta)}{k+N_{\delta}}$.
\end{proof}

The following lemma is a first step towards proving Theorem \ref{t2} and is a consequence of Lemma \ref{l285}.

\begin{lemma}\label{l286}
Let us use the notations of the lemma \ref{l285}.
Then, 
\begin{equation}\label{e1}
\abs{\int_{I_\delta}T_{\delta}(t)\left(\sum_{k=0}^Na_ke^{2i\pi\lambda_{k}t}\right)\varphi(t)\,\mathrm{d}t}
\geq \alpha_{\delta}\sum_{k=0}^N\frac{|a_{k}|}{k+N_{\delta}}.
\end{equation}
\end{lemma}

\begin{proof}
By definition of $T_{\delta}$,
\begin{eqnarray*}
\int_{I_\delta}T_{\delta}(t)e^{2i\pi\lambda_{k}t}\varphi(t)\ud t
&=&\sum_{j=0}^M\frac{u_{j}}{j+N_{\delta}}\int_{I_\delta}e^{-2i\pi\lambda_{j}t}e^{2i\pi\lambda_{k}t}\varphi(t)dt\\
&=&\sum_{j=0}^M\frac{u_{j}}{j+N_{\delta}}\widehat{\varphi}(\lambda_{j}-\lambda_{k})\\
&=& \frac{u_{k}}{k+N_{\delta}}+\sum_{\substack{0\leqslant j \leqslant M \\ j\neq k}}
\frac{u_{j}}{j+N_{\delta}}\widehat{\varphi}(\lambda_{j}-\lambda_{k})
\end{eqnarray*}
thus
$$
\abs{\int_{I_\delta}T_{\delta}(t)e^{2i\pi\lambda_{k}t}\varphi(t)dt-\frac{u_{k}}{k+N_{\delta}}}
\leq \sum_{\substack{0 \leqslant j \leqslant M \\ j\neq k}}\frac{1}{j+N_{\delta}}\widehat{\varphi}(\lambda_{j}-\lambda_{k}).
$$ 
By applying the lemma \ref{l285}, that
$$
   \bigg|\int_{I_\delta}T_{\delta}(t)e^{2i\pi\lambda_{k}t}\varphi(t)dt-\frac{u_{k}}{k+N_{\delta}}\bigg|\leqslant \frac{1-\alpha_{\delta}}{k+N_{\delta}}.
$$

It follows that
$$
   \bigg|\int_{I_\delta}T_{\delta}(t)a_ke^{2i\pi\lambda_{k}t}\varphi(t)dt-\frac{u_{k}a_k}{k+N_{\delta}}\bigg|\leqslant \frac{1-\alpha_{\delta}}{k+N_{\delta}}|a_k|.
$$
Using the fact that $u_ka_k=|a_k|$ and the triangular inequality, we obtain 
$$
\bigg|\int_{I_\delta}T_{\delta}(t)\left(\sum_{k=0}^Na_ke^{2i\pi\lambda_{k}t}\right)\varphi(t)dt-\sum_{k=0}^N\frac{|a_{k}|}{k+N_{\delta}}\bigg|\leqslant (1-\alpha_{\delta})\sum_{k=0}^N\frac{|a_k|}{k+N_{\delta}}
$$
from which the lemma follows immediately.
\end{proof}

\medskip

\noindent{\bf Construction of $\tilde T_{\delta}$}

\medskip

Recall that, from our assumption on $N_{\delta}$ and $N$, we can write
$$
T_{\delta}(t)=\sum_{k=0}^M\frac{u_{k}}{k+N_{\delta}}e^{-2i\pi\lambda_{k}t}
=\sum_{j=m_{\delta}}^{n_{\delta}}f_j(t)
$$
where we set $\dd_j=\{k\in\N\,:\ 2^j\leq k<2^{j+1}\}$ and
$$
f_{j}(t)=\sum_{r+N_{\delta}\in \dd_{j}}\frac{u_{r}}{r+N_{\delta}}e^{-2i\pi\lambda_{r}t}.
$$

Before we estimate the norms of the $f_j$'s, let us recall Hilbert's inequality ({\it see e.g.} \cite[Chapter 10]{CQ}).

\begin{lemma}[Hilbert's inequality]\label{l225}
Let $\lambda_{1}, \ldots, \lambda_N$ be real numbers with $|\lambda_k-\lambda_\ell|\geq 1$ when $k \neq \ell$, 
and let $z_{1}, \ldots, z_N$ be complex numbers. We have
$$ 
\bigg| \sum_{\substack{1\leq k,\ell \leq N\\ k\neq \ell } } \frac{z_{k}\overline{z_{\ell}}}{\lambda_{k}-\lambda_{\ell}} \bigg|\leqslant \pi\sum_{k=1}^N|z_{k}|^{2}.
$$
\end{lemma}

We can now prove the following:

\begin{lemma}\label{l288} For $m_{\delta} \leqslant j \leqslant n_{\delta}$,
\begin{enumerate}
    \item $\|f_{j}\|_{\infty}\leqslant 1  $
    \item  $\|f_{j}\|_{L^{2}(I_\delta)}\leqslant 2^{-\frac{j}{2}}\sqrt{|I_\delta|+1}.$
\end{enumerate}
\end{lemma}

\begin{proof}
The first estimate is obtained by straightforward computations: for all $t$,
$$
|f_{j}(t)| \leqslant \sum_{r+N_{\delta}\in \dd_{j}}\frac{1}{r+N_{\delta}}\leqslant \frac{|\D_{j}|}{2^{j}}=1
$$
hence $\|f_{j}\|_{\infty} \leqslant 1$.

For the second one, set $v_{r}=\displaystyle\frac{u_{r}}{r+N_{\delta}}$. 
Then
\begin{eqnarray*}
\|f_{j}\|_{L^{2}(I_\delta)}^{2}
&=&\int_{I_\delta} f_{j}(t)\overline{f_{j}(t)}\ud t
=\int_{I_\delta}\sum_{r+N_{\delta},s+N_{\delta}\in \dd_{j}}v_{r}\overline{v_s}e^{-2i\pi(\lambda_r-\lambda_s)t}\ud t\\
&=&|I_\delta|\sum_{r+N_{\delta}\in \dd_j}|v_r|^2
+\sum_{\substack{r+N_{\delta},s+N_{\delta}\in \dd_{j}\\ r\neq s}}
v_{r}\overline{v_s}\int_{-|I_\delta|/2}^{|I_\delta|/2}e^{-2i\pi(\lambda_r-\lambda_s)t}\ud t\\
&=&|I_\delta|\sum_{r+N_{\delta}\in \dd_j}|v_r|^2
+\sum_{\substack{r+N_{\delta},s+N_{\delta}\in \dd_{j}\\ r\neq s}}v_r\overline{v_s}
\left(\frac{e^{i|I_\delta|\pi(\lambda_r-\lambda_s)}-e^{-i|I_\delta|\pi(\lambda_r-\lambda_s)}}{
2i\pi(\lambda_r-\lambda_s)}\right).
\end{eqnarray*}

It follows that
$$
\|f_{j}\|_{L^2(I)}^{2}\leqslant |I_\delta|\sum_{r+N_{\delta}}|v_r|^2
+\frac{1}{2\pi}\left|\sum_{\substack{r,s\\ r\neq s}}
\frac{v_re^{-i\pi\lambda_r(1+\delta)}\overline{v_se^{-i\pi\lambda_s(1+\delta) }}}{\lambda_s-\lambda_r}\right|
+\frac{1}{2\pi}\left|\sum_{\substack{r,s\\ r\neq s}}
\frac{v_re^{i\pi\lambda_r(1+\delta)} \overline{v_se^{i\pi\lambda_s(1+\delta) }}}{\lambda_s-\lambda_r}\right|
$$
We then apply Hilbert's inequality \ref{l225} to the last two sums to obtain
\begin{eqnarray*}
\|f_{j}\|_{2}^{2}&\leqslant& |I_\delta|\sum_{r+N_{\delta}}|v_r|^2+\frac{1}{2} \sum_{r+N_{\delta,\beta}\in \dd_j}|v_{r}|^{2}+\frac{1}{2} \sum_{r+N_{\delta}\in \dd_j}|v_{r}|^{2}\\
&=&(|I_\delta|+1)\sum_{r+N_{\delta}}|v_r|^2.
\end{eqnarray*}
Since
\begin{eqnarray*}
\sum_{r+N_{\delta}\in \dd_j}|v_r|^2
&=&\sum_{r+N_{\delta}\in \dd_{j}}\frac{|u_{r}|^{2}}{(r+N_{\delta})^2}
=\sum_{r+N_{\delta}\in \dd_j}\frac{1}{(r+N_{\delta})^{2}}\\
&\leqslant& \sum_{r+N_{\delta}\in \dd_{j}}\frac{1}{2^{2j}}\leqslant\frac{2^{j}}{2^{2j}}=2^{-j},
\end{eqnarray*}
 we get $ \|f_{j}\|_{L^2(I)}^2\leqslant (|I_\delta|+1)2^{-j}$ as claimed.
\end{proof}

We then write the Fourier series expansion of $|f_j|\in L^2(I_\delta)$
as
$$
|f_{j}(t)|=\sum_{s\in \mathbb{Z}}a_{s,j}e^{\frac{2i\pi}{|I_\delta|}st},
$$
and define $h_{j}\in L^{2}(I_\delta)$ via its Fourier series expansion 
$$
h_j(t)=a_{0,j}+2\sum_{s=1}^{\infty}a_{s,j}e^{\frac{2is\pi}{|I_\delta|}t}.
$$ 

\begin{lemma}\label{l289} For $m_{\delta} \leqslant j \leqslant n_{\delta}$,
the following properties hold
\begin{enumerate}
    \item $\mathrm{Re}(h_{j})=|f_{j}|$;
    \item$\|h_{j}\|_{2}\leqslant \sqrt{2}\|f_{j}\|_{2}$;
    \item $h_j\in H^\infty(I_\delta)$.
\end{enumerate}
\end{lemma}

\begin{proof} First, as $|f_j|$ is real valued, its zero Fourier coefficient
$a_{0,j}$ is also real, and for the remaining ones
$\overline{a_{s,j}}=a_{-s,j}$ for every $s\geq 1$. 
A direct computation then shows that $\mathrm{Re}(h_{j})=|f_{j}|$ which is less than 1 by lemma \ref{l288} while
Parseval shows that $\|h_{j}\|_{2}\leqslant \sqrt{2}\|f_{j}\|_{2}$.
\end{proof}

We now define a sequence $(F_j)_{j\geqslant m_\delta}$ inductively through 
$$
F_{m_{\delta}}=f_{m_{\delta}}
\quad\mbox{and}\quad
F_{j+1}=F_{j}e^{-\varepsilon h_{j+1}}+f_{j+1}
$$
where $0<\varepsilon< 1$ is a small parameter that we will adjust later.
Further set
$$
E_{\eps}:=\sup_{0 <x \leq 1} \dfrac{x}{1-e^{-\eps x}}=\frac{1}{\eps}\sup_{0 <x \leq \eps} \dfrac{x}{1-e^{- x}}
=\frac{1}{1-e^{-\eps}}.
$$
From the next to last identity, it is easy to obtain the following simple bound:
$$
\frac{1}{\eps}\leq E_{\eps}\leq \frac{2}{\eps}.
$$

\begin{lemma}\label{l2810}
For all $m_{\delta} \leqslant j \leqslant n_{\delta}$,
$$
\|F_{j} \|_{\infty}\leqslant \dfrac{2}{\varepsilon}.
$$
\end{lemma}

\begin{proof}
By definition of $E_\eps$, if $C\leq E_{\eps}$ and
$0\leq x\leq 1$, then  $Ce^{-\eps x}+x\leq E_{\eps} e^{-\eps x}+x\leq E_{\eps}$.

We can now prove by induction over $j$ that $|F_j|\leq E_{\eps}$ from which the lemma follows.
First, when $j=0$, from Lemma \ref{l288} we get
$$
\|F_0\|_{\infty}=\|f_0\|_{\infty} \leq 1\leq E_{\eps}.
$$
Assume now that $\|F_j\|_\infty\leq E_{\eps}$, then
\begin{eqnarray*}
|F_{j+1}(t)|&=&|F_{j}(t)e^{-\eps h_{j+1}(t)}+ f_{j+1}(t)|\leq
|F_{j}(t)|e^{-\eps\Re\bigl(h_{j+1}(t)\bigr)}+|f_{j+1}(t)|\\
&=&|F_{j}(t)|e^{-\eps|f_{j+1}(t)|}+|f_{j+1}(t)|.
\end{eqnarray*}
As $|f_{j+1}(t)|\leq 1$ and $|F_{j}(t)|\leq E_{\eps}$, we get
$|F_{j+1}(t)|\leq E_{\eps}$ as claimed.
\end{proof}

\begin{lemma}\label{l2811}
Let $m_{\delta} \leqslant n \leqslant n_{\delta}$. For $j=m_{\delta},\ldots,n$ we define $g_{j,n}=e^{-\varepsilon H_{j,n}}$
with
$$
H_{j,n}=\begin{cases}
h_{j+1}+\ldots +h_{n} & \mbox{if } j<n\\
0  & \mbox{if }j=n
\end{cases}.
$$
Then $F_{n}=\displaystyle\sum_{j=m_{\delta}}^{n}f_{j}g_{j,n}$.
Moreover $\| H_{j,n} \|_{2}\leqslant \dfrac{\sqrt{2(|I_\delta|+1)}}{\sqrt{2}-1}2^{-\frac{j}{2}}$.
\end{lemma}

\begin{proof}
The first part is obtained by a simple induction on $n$.
Moreover, the lemma \ref{l289} implies
\begin{eqnarray*}
\|H_{j,n}\|_{2}&=&\left\|\sum_{r=j+1}^{n}h_{r}\right\|_{2}\leqslant \sum_{r=j+1}^{n}\|h_r\|_{2}
\leqslant\sqrt{2}\sum_{r=j+1}^{n}\|f_{r}\|_{2}\\
&\leqslant& \sqrt{2(|I_\delta|+1)}\sum_{r=j+1}^{\infty}2^{-\frac{r}{2}}
=\dfrac{\sqrt{2(|I_\delta|+1)}}{\sqrt{2}-1}2^{-\frac{j}{2}},
\end{eqnarray*}
as stated.
\end{proof}

Next, we will need the following simple well-known lemma: 

\begin{lemma}\label{l2812}
If $H \in H^{\infty}$ and $Re(H)\geqslant 0,$ then $e^{-H}\in H^{\infty}$ and 
$$
\| e^{-H}-1\|_{2}\leqslant \|H\|_{2}.
$$
\end{lemma}

\begin{proof}
Since $H^{\infty}$ is a Banach algebra, the partial sums $\displaystyle \sum_{k=0}^{n}(-1)^{k} \frac{H^{k}}{k!}$ of $e^{-H}$ are elements of $H^{\infty}$. Moreover, since $H$ is bounded, these sums converge uniformly toward $e^{-H}$, with $e^{-H} \in H^{\infty}$. Finally, if $z \in \C$ and $\Re (z) \geqslant 0$,
$$
\left|e^{-z}-1\right|=\left|\int_{0}^{1} z e^{-t z} d t\right| \leqslant  \int_{0}^{1}|z| e^{-t \Re (z)} d t \leqslant |z| .
$$
In our case $z=H(t)$, and we have 
$$
\left|e^{-H(t)}-1\right| \leqslant |H(t)|
$$
and by integration we have the desired inequality.
\end{proof}

\begin{lemma}\label{l2813}
Assume that $0<\varepsilon\leqslant \dfrac{\sqrt{|I_\delta|}(\sqrt{2}-1)}{\sqrt{2(|I_\delta|+1)}}$.
Then, for $m_{\delta} \leqslant j \leqslant n\leqslant n_{\delta}$,
\begin{enumerate}
    \item $\|g_{j,n}-1\|_{2}\leqslant \varepsilon \|H_{j,n}\|_{2} \leqslant \dfrac{\sqrt{2(|I_\delta|+1)}}{\sqrt{2}-1}\varepsilon 2^{-\frac{j}{2}}$
    \item The Fourier series of $g_{j,n}(t)-1$ writes
    $g_{j,n}(t)-1=\displaystyle \sum_{s \geqslant 0}  c_{s,j}e^{\frac{2is\pi}{|I|}t}$,
    with $\dst\sum_{s=0}^{+\infty}|c_{s,j}|^2\leqslant 1$.
\end{enumerate}
\end{lemma}

\begin{proof}
By lemma \ref{l2812}, $\|g_{j}-1\|_{2}\leqslant \varepsilon \|H_{j}\|_{2}.$ 
Then, since $g_{j,n}$ is analytic, its Fourier series writes
$$
g_{j}(t)-1=\displaystyle\sum_{s\geqslant 0}c_{s,j}e^{\frac{2i\pi}{|I_\delta|}st}.
$$
But then, with Parseval
$$
\left(\sum_{s \geqslant 0}|c_{s,j}|^{2} \right)^{\frac{1}{2}}
=\frac{1}{\sqrt{|I_\delta|}}\|g_{j}-1\|_{2}
\leqslant \dfrac{\varepsilon}{\sqrt{|I_\delta|}}\|H_{j}\|_{2}
\leqslant \dfrac{\varepsilon}{\sqrt{|I_\delta|}}\dfrac{\sqrt{2(|I_\delta|+1)}}{\sqrt{2}-1}2^{-\frac{j}{2}}\leqslant 2^{-\frac{j}{2}}< 1
$$
which implies the claimed bound.
\end{proof}

Now recall that
$$
T_{\delta}=\sum_{m_{\delta} \leqslant j \leqslant n_{\delta}}f_{j}
$$
and define
$$
\tilde T_{\delta}=F_{n_{\delta}}.
$$
In particular, from Lemma \ref{l2810}, we have
\begin{equation}
\|\tilde T_{\delta}\|_\infty\leq\dfrac{2}{\eps}.
\label{eq:estttildedelta}
\end{equation}

The key estimates here is the following;
\begin{lemma}\label{l2814}
Once again, we use the notations of the lemma \ref{l285}.
There exists $\delta_2>0$ such that, if $0<\delta<\delta_2$ and
$N_{\delta}\geq\delta^{-7/2}$ then
\begin{equation}\label{e2}
\left| \int_{I_\delta}(\tilde T_{\delta}-T_{\delta})(t)\left(\sum_{k=0}^Na_ke^{2i\pi\lambda_{k}t}\right)\varphi(t)\ud t \right| 
\leqslant\frac{2}{3} \alpha_{\delta}\sum_{k=0}^N\frac{|a_k|}{k+N_{\delta}}.
\end{equation}
where $\varphi$ is the function defined in \eqref{eq:defffi}.
\end{lemma}

\begin{proof}
It is enough to prove that, under the conditions of the lemma, for $0\leqslant k \leqslant N$ we have  
\begin{equation}\label{e2bis}
\left| \int_{I_\delta}(\tilde T_{\delta}-T_{\delta})(t)e^{2i\pi\lambda_{k}t}\varphi(t)\ud t \right| 
\leqslant\frac{2}{3} \frac{\alpha_{\delta}}{k+N_{\delta}}.
\end{equation}
Once \eqref{e2bis} is established, it will then be enough to multiply the left hand side by $a_k$
and to use the triangular inequality

We fix $k \in [ 0,N]$ and let $\ell$ the index such that $k+N_{\delta} \in \dd_{\ell}$.
We define $R,R_{1}$ and $R_{2}$ as follows
\begin{eqnarray*}
R&=&\displaystyle\int_{I_\delta}(\tilde T_{\delta,\beta}-T_{\delta,\beta})(t)e^{2i\pi \lambda_{k}t}\varphi(t)\ud t\\
&=&\int_{I}\sum_{m_{\delta} \leqslant j \leqslant \ell-2}f_{j}(t)\bigl(g_{j}(t)-1\bigr)
e^{2i\pi\lambda_{k}t}\varphi(t\ud t
+\int_{I_\delta}\sum_{\ell-1 \leqslant j \leqslant n_{\delta}}f_{j}(t)\bigl(g_{j}(t)-1\bigr)
e^{2i\pi\lambda_{k}t}\varphi(t)\ud t\\
&:=&R_{1}+R_{2}
\end{eqnarray*}

We will first bound $R_{1}$. Note that if $s\in \mathbb{Z},$
\begin{eqnarray*}
\int_{I_\delta}f_{j}(t)\varphi(t)e^{2i\pi\lambda_{k}t} e^{\frac{2i\pi }{|I_\delta|}st}\ud t
&=&\int_{I_\delta}\sum_{r+N_{\delta}\in \dd_{j}}\frac{u_{r}}{r+N_{\delta}}\varphi(t)
e^{2i\pi (-\lambda_{r}+\lambda_{k}+\frac{s}{|I_\delta|})}\ud t\\
&=&\sum_{r+N_{\delta}\in \dd_{j}}\frac{u_{r}}{r+N_{\delta}}
\widehat{\varphi}(\lambda_{r}-\lambda_{k}-\frac{s}{|I_\delta|}).
\end{eqnarray*}
From there, we obtain
\begin{eqnarray*}
\int_{I_\delta}f_{j}(t)\bigl(g_{j}(t)-1\bigr)e^{2i\pi \lambda_{k}t}\varphi(t)\ud t 
&=& \int_{I_\delta}f_{j}(t)\varphi(t)e^{2i\pi \lambda_{k}t}\sum_{s \geqslant 0}c_{s,j}e^{\frac{2is\pi}{|I_\delta|}t}\ud t\\
&=&\sum_{s= 0}^{+\infty}c_{s,j}\int_{I_\delta}f_{j}(t)\varphi(t)e^{2i\pi\lambda_{k}t}e^{\frac{2i\pi }{|I_\delta|}st}\ud t \\
&=&\sum_{s= 0}^{+\infty}c_{s,j}\sum_{r+N_{\delta}\in \dd_{j}}\frac{u_{r}}{r+N_{\delta}}
\widehat{\varphi}(\lambda_{r}-\lambda_{k}-\frac{s}{|I_\delta|})\\
&=&\sum_{r+N_{\delta}\in \dd_{j}}\frac{u_{r}}{r+N_{\delta}}\sum_{s=0}^{\infty}c_{s,j}
\widehat{\varphi}(\lambda_{r}-\lambda_{k}-\frac{s}{|I_\delta|}).
 \end{eqnarray*}
So finally we get
$$
R_{1}=\sum_{m_{\delta} \leqslant j \leqslant \ell-2}\,\sum_{r+N_{\delta}\in \dd_{j}}\frac{u_{r}}{r+N_{\delta}}
\sum_{s=0}^{\infty}c_{s,j}\widehat{\varphi}(\lambda_{r}-\lambda_{k}-\frac{s}{|I_\delta|}).
$$
Let  $c_{s}(r)=c_{s,j}$ if $r+N_{\delta}\in I_{j}$. Since $\widehat{\varphi}$ is an even function, we can write 
$$
R_{1}=\sum_{2^{m_{\delta}} \leqslant r+N_{\delta} < 2^{\ell-1}}\frac{u_{r}}{r+N_{\delta}}
\sum_{s=0}^{\infty}c_{s}(r)\widehat{\varphi}(\lambda_{k}-\lambda_{r}+\frac{s}{|I_\delta|})
:=\sum_{2^{m_{\delta}} \leqslant r+N_{\delta} < 2^{\ell-1}}\frac{u_{r}}{r+N_{\delta}} E_{r}.
$$ 
From Lemma \ref{l282}, recall that, with 
$D_\delta=c_0\delta^{-\frac{p+q}{p+2}}$ and $\nu=\dfrac{\alpha_\delta}{|I_\delta|^{1/q}}$ 
then, for 
$$
|\lambda|\geqslant\max(1,D_\delta\nu^{-\frac{1}{p+2}}),
$$
we have
\begin{equation}
\label{eq:defF}
|\widehat{\varphi}(\lambda)|\leqslant \frac{\alpha_{\delta}}{(|I_\delta||\lambda|)^{q}}.
\end{equation}

Now, $s\geqslant 0$, $|I_\delta|\geq 1$ and, as 
$r+N_{\delta}< 2^{\ell-1}$, $2^{\ell}\leqslant k+N_{\delta}< 2^{\ell+1}$,
$\lambda_k>\lambda_r$ thus
\begin{eqnarray}
|I_\delta|\abs{\lambda_{k}-\lambda_{r}+\frac{s}{|I_\delta|}}
&=&|I_\delta|(\lambda_{k}-\lambda_{r})+s\geqslant\lambda_k-\lambda_r+s\geqslant k-r+s\label{eq:estaaa}\\
&=&(k+N_{\delta})-(r+N_{\delta})>2^{\ell}-2^{\ell-1}=2^{\ell-1}\geqslant 2^{m_{\delta}-1}.\nonumber
\end{eqnarray}
Further, from \eqref{eq:estalpha} and $1\leq |I_\delta|\leq 2$, $D_\delta\nu^{-\frac{1}{p+2}}
\leq c_3\delta^{-\frac{p+q+2}{p+2}}=c_3\delta^{-\frac{20}{12}}$.
Thus, choosing $m_{\delta}$ sufficiently large for $N_{\delta}=2^{m_{\delta}}\delta^{-7/2}$
and $\delta\leq\delta_2$ for some $\delta_2>0$ small enough,  
we are able to apply \eqref{eq:defF} and obtain, with \eqref{eq:estaaa}
$$
|\widehat{\varphi}(\lambda_{k}-\lambda_{r}+\frac{s}{|I_\delta|})|
\leqslant\frac{\alpha_{\delta}}{(k-r+s)^q}.
$$
We can now bound $E_{r}$. Since $\dst\sum_{s=0}^{\infty}|c_{s}(r)|^2\leqslant 1$, the previous
bound and Cauchy-Schwarz give us
\begin{eqnarray*}
|E_{r}|&=&\sum_{s=0}^{\infty}|c_{s}(r)|
\abs{\widehat{\varphi}\left(\lambda_{k}-\lambda_{r}+\frac{s}{1+\delta}\right)}
\leqslant \alpha_{\delta}\left(\sum_{s=0}^{\infty}\frac{1}{(k-r+s)^{2q}}\right)^{1/2}\\
&\leqslant&\alpha_{\delta}\left(\sum_{n=k-r}^{\infty}\sum_{n\geqslant k-r}\frac{1}{n^{2q}}\right)^{1/2}\\
&\leqslant&\alpha_{\delta}\left(\int_{k-r-1}^{\infty}\frac{dt}{t^{2q}}\right)^{1/2}
= \frac{\sqrt{2q-1}\alpha_{\delta}}{(k-r-1)^{q-1/2}}.
 \end{eqnarray*}

But $k-r> 2^{\ell-1}$ then
$k-r-1 \geqslant 2^{\ell-1} $. Since $k+N_{\delta} \in \dd_{\ell}$ {\it i.e}
$2^\ell \leqslant k+N_{\delta} \leqslant2^{\ell+1}$ we get $\dfrac{1}{k-r-1}\leqslant \dfrac{4}{k+N_{\delta}}$ and then
$$
|E_{r}|\leqslant \frac{4^{q-1/2}\sqrt{2q-1}\alpha_{\delta}}{25(k+N_{\delta})^{q-1/2}}
$$
Finally, we deduce that
\begin{eqnarray*}
|R_{1}|
&=&\left| \sum_{2^{m_{\delta}}\leqslant r+N_{\delta} < 2^{\ell-1} }\frac{u_{r}}{r+N_{\delta}}E_{r} \right|
\leqslant \sum_{2^{m_{\delta}}\leqslant r+N_{\delta} < 2^{\ell-1} } \frac{|E_{r}|}{r+N_{\delta}}\\
&\leqslant& \frac{4^{q-1/2}\sqrt{2q-1}\alpha_{\delta}}{(k+N_{\delta})^{q-1/2}}
 \end{eqnarray*}
since last sum has at most $2^{\ell-1}\leqslant k+N_{\delta}$.


\smallskip

We will now bound $R_2$.

\begin{eqnarray*}
|R_{2}| &\leqslant& \displaystyle \sum_{\ell-1 \leqslant j \leqslant n_\delta}\int_{I}|f_{j}(t)|\,|g_{j}(t)-1|\,|\varphi(t)|\ud t\\
&\leqslant& \|\varphi\|_{\infty}\sum_{\ell-1 \leqslant j \leqslant n_\delta} \|f_{j}\|_{2}\|g_{j}-1\|_{2}.
\end{eqnarray*}
According to the lemmas \ref{l288} $(1)$ and \ref{l2813} $(1)$ we get
$$
|R_2|\leqslant \|\varphi\|_{\infty}\varepsilon \dfrac{\sqrt{2}(2+\delta)}{\sqrt{2}-1}\sum_{l-1\leqslant j \leqslant m}2^{-j}.
$$
since
$$
\sum_{\ell-1 \leqslant j \leqslant n_\delta}2^{-j}
\leqslant \sum_{j=\ell-1}^{\infty}2^{-j}
=2^{-\ell+2}=8.2^{-(\ell+1)}
$$
and $k+N_{\delta}\in I_{l}$ then $\dfrac{1}{2^{\ell+1}} \leqslant \dfrac{1}{k+N_{\delta}}$. Consequently
$$
\sum_{\ell-1\leqslant j\leqslant n_\delta}2^{-j}\leqslant \frac{8}{k+N_\delta}.
$$
Finally, we deduce that
$$
R_2\leqslant \varepsilon  \| \varphi \|_{\infty}\dfrac{\sqrt{2}(|I_\delta|+1)}{\sqrt{2}-1}\frac{8}{k+N_{\delta}}
$$
 and we obtain  $R_{2}\leqslant \dfrac{\alpha_\delta}{3}\dfrac{1}{k+N_{\delta}}$ when  $\varepsilon \leqslant \dfrac{(\sqrt{2}-1)}{24\|\varphi\|_{\infty}(|I_\delta|+1)\sqrt{2}}\alpha_\delta$.
\end{proof}

Note that, from \eqref{eq:estalpha} and $|I_\delta|\leq 2$, we can take $\eps= c_4\delta^2$ for some $c_4>0$.

It is now easy to deduce the theorem \ref{t2} using the 2 inequalities \eqref{e1} and \eqref{e2}.

\begin{proof}[Proof of theorem \ref{t2}]
Let $S_{\delta}=\displaystyle\sum_{k=0}^N\frac{|a_k|}{k+N_\delta}$
and $\Phi(t)=\dst\sum_{k=0}^Na_ke^{2i\pi\lambda_kt}$ as previously defined. 
Recall that in \eqref{e1} , we have shown that
\begin{eqnarray*}
\alpha_\delta S_\delta&\leq&\left|\int_{I_\delta}T_{\delta}(t)\phi(t)\varphi(t)\ud t\right|\\
&\leq&\left|\int_{I_\delta}\tilde T_{\delta}(t)\phi(t)\varphi(t)\ud t\right|
+\left|\int_{I_\delta}\bigl(\tilde T_\delta(t)-T_{\delta}(t)\bigr)\phi(t)\varphi(t)\ud t\right|\\
&\leq&\left|\int_{I_\delta}\tilde T_{\delta}(t)\phi(t)\varphi(t)\ud t\right|+\frac{2}{3}\alpha_\delta S_\delta
\end{eqnarray*}
with \eqref{e2}.

It follows that
\begin{eqnarray*}
 S_\delta&\leq&\frac{3}{\alpha_\delta}\left|\int_{I}\tilde T_{\delta}(t)\phi(t)\varphi(t)\ud t\right|\\
 &\leq&\frac{3\|\tilde T\|_\infty\|\varphi\|_\infty}{\alpha_\delta}
 \int_{I_\delta}|\phi(t)|\ud t.
\end{eqnarray*}

But, if $\delta$ is small enough and $N_\delta\geq \delta^{-7/2}$,

-- from Lemma \ref{l282}, $\|\varphi\|_\infty\leq\dfrac{\pi}{2}$;

-- from \eqref{eq:estalpha}, $\alpha_\delta=\kappa\delta^2+O(\delta^4)$

-- as $\eps=c_4\delta^2$, from \eqref{eq:estttildedelta}, $\|\tilde T\|_\infty=\dfrac{2}{c_4\delta^2}$.

Therefore, there are two absolute constants $\delta_*$ and $c_*$ such that, if $\delta\leq\delta_*$,
$$
S_\delta\leq \frac{c_*}{\delta^4}\int_{I_\delta}|\phi(t)|\ud t.
$$
As noticed at the start of the proof, this implies that
$$
\sum_{n=0}^N\frac{|a_k|}{k+1}\leq\frac{c_*}{\delta^{\frac{15}{2}}}\int_{-\frac{1+\delta}{2}}^{\frac{1+\delta}{2}}
\abs{\sum_{k=0}^Na_ke^{2i\pi\lambda_k t}}\ud t
$$
for every $N$, every sequence of real numbers $(\lambda_j)_{j\geq 0}$ with $\lambda_{j+1}-\lambda_j\geq 1$
and every complex sequence $(a_j)_{j\geq 0}$.
\end{proof}
We have not fully optimised the proof, by taking $q$ sufficiently large and $p/q$ sufficiently large, one can replace
$\delta^{15/2}$ by $\delta^{7+\eta}$ for any fixed $\eta$.

\section{Some open problems}

The $L^2$ theory of exponential sums is rather well understood. This is not the case for $L^1$ theory for which there are still many open questions. Let us mention a few of them.

\begin{enumerate}
\item There is a major difference between the sums that appear in Ingham's Theorem and those that appear in Mc Geehe, Pigno, Smith and Nazarov's Theorems. In the $L^2$ case, the sums can be two sided and not in the $L^1$ case.
The proof given here does not work in this case (this is due to the construction of $\tilde T$) and we 
are tempted to conjecture the following

\begin{question}
Let $T>1$ and $(\lambda_k)_{k\in\Z}$ a real sequence such that $\lambda_{k+1}-\lambda_k\geq 1$ for every $k$
and $\lambda_k$ has same sign as $k$.

Does there exist a constant $C$ such that, for every $N$ and every sequence $(c_j)_{j=-N,\ldots,N}$ of complex numbers,
$$
\sum_{k=-N}^N\frac{|c_k|}{|k|+1}\leq C\int_{-T/2}^{T/2}
\abs{\sum_{k=-N}^Nc_ke^{2i\pi\lambda_kt}}\ud t.
$$
\end{question}

\item The second question is the optimality of the condition $T=1$ in Nazarov's Theorem.
In view of Ingham's counterexample, it is tempting to conjecture that $T>1$ is requested (as observed by Nazarov,
his proof requires this).

\begin{question}
Is the condition $T>1$ necessary in Nazarov's Theorem. If yes,
what is the right behaviour of the constants. 
\end{question}

We think the estimate shown here is not optimal.

\item Another question related to the previous one comes from Haraux's Theorem.
When $\lambda_{k+1}-\lambda_k\to+\infty$, $T$ can be chosen arbitrarily small.
A key element of Haraux's strategy is that the lower and upper bounds in Ingham's
inequality are both multiples of the $\ell^2$-norm of the coefficients. This is
no longer the case in the $L^1$ setting. 

In forthcomming work, we have been able to use a compactness argument to show
that one may take $T$ arbitrarily small in Nazarov's Theorem when $\lambda_{k+1}-\lambda_k\to+\infty$.
However, in doing so, we loose control of constants. This leads to the following:

\begin{question}
Prove a quantitative version Nazarov's Theorem with $T$ arbitrarily small when $\lambda_{k+1}-\lambda_k\to+\infty$.
\end{question}

\item Are $L^1([-T,T])$-norms of lacunary non-harmonic Fourier series comparable to the $\ell^2$-norms
of the coefficients? This is the case for Besikovich norms.

Note that
$$
\sum_{j=0}^{\infty}\frac{|a_j|}{j+1}\leq \left(\sum_{j=0}^{\infty}\frac{1}{(j+1)^2}\right)^{1/2}
\left(\sum_{j=0}^{\infty}|a_j|^2\right)^{1/2}=\frac{\pi}{\sqrt{6}}\left(\sum_{j=0}^{\infty}|a_j|^2\right)^{1/2}.
$$
The question is then
\begin{question}
Find a (gap) condition on $(\Lambda_k)$ and on $T$ that implies that there is a constant $C>0$
such that, for every $(a_k)$ with finite support
$$
\frac{1}{T}\int_{-T/2}^{T/2}\abs{\sum_{j=0}^{+\infty}a_je^{2i\pi\lambda_j t}}\ud t
\geq C\left(\sum_{j=0}^{\infty}|a_j|^2\right)^{1/2}.
$$
\end{question}

One may also go the other way and ask whether one can still obtain a result like Theorem \ref{th:MPS}
when $\lambda_{k+1}-\lambda_k\to 0$ but with a smaller power on the denominator.
For instance, when $\lambda_k=\ln k$ then the Fourier series become Dirichlet series
$$
\sum_{k=1}^{+\infty}\frac{a_k}{k^{2i\pi t}}=\sum_{k=1}^{+\infty}a_k e^{2i\pi(\ln k) t}.
$$
The following question was asked in \cite{BPS}:

\begin{question}
Is it true that, for every sequence $(a_k)_{k\geq 1}$ with finite support,
$$
C\left(|a_1|+\sum_{k\geq 2}\frac{|a_k|}{\sqrt{k}\ln k}\right)\leq \lim_{T\to+\infty}
\int_{-T/2}^{T/2}\abs{\sum_{k=1}^{+\infty}\frac{a_k}{k^{2i\pi t}}}\ud t.
$$
\end{question}
\end{enumerate}

\end{document}